\pdfoutput=1

\documentclass[11pt,a4paper]{article}

\usepackage[margin=1.0in]{geometry}

\usepackage{url}
\usepackage{enumitem}

\usepackage{amsmath,amssymb,euscript}
\usepackage{dsfont}

\usepackage{graphicx}
\usepackage{color}

\usepackage{ntheorem}

\usepackage{tikz}

\newenvironment{proof}[1][{}]{%\novbskip%
  \begin{trivlist}\item[]\textit{Proof #1}\quad}%
  {\hfill\hspace*{\fill}~$\square$\end{trivlist}}

\theorembodyfont{\normalfont}
\newtheorem{mainthm}{Theorem} % independently numbered
\newtheorem{thm}{Theorem}%[section]
\newtheorem{prop}[thm]{Proposition}

\newtheorem{lem}[thm]{Lemma}
\newtheorem{de}[thm]{Definition}

\newtheorem{remark}[thm]{Remark}

\theoremstyle{empty}
\newtheorem{dup}{Duplicate} % for duplicating a thm (give opt arg)

\definecolor{turquoise}{cmyk}{0.65,0,0.1,0.1}
\definecolor{darkgreen}{cmyk}{0.87,0.5,0.81,0.66}

\newcommand{\rawdef}[1]{\emph{#1}} % no index entry
\newcommand{\defn}[1]{\rawdef{#1}\index{#1}}

\newcommand{\Appref}[1]{Appendix~\ref{#1}}

\newcommand{\Defref}[1]{Definition~\ref{#1}}
\newcommand{\Eqnref}[1]{Equation~\eqref{#1}}

\newcommand{\Lemref}[1]{Lemma~\ref{#1}}

\newcommand{\Secref}[1]{Section~\ref{#1}}

\newcommand{\Thmref}[1]{Theorem~\ref{#1}}
\newcommand{\Propref}[1]{Proposition~\ref{#1}}

\DeclareMathOperator{\Id}{Id}

\DeclareMathOperator*{\argmin}{argmin}

\newcommand{\cinfty}{C^\infty}

\DeclareMathOperator{\convh}{conv}
\newcommand{\convhull}[1]{\convh(#1)}

\newcommand{\R}{\mathbb{R}}
\newcommand{\reel}{\mathbb{R}}

\newcommand{\E}{\mathbb{E}}

\DeclareMathOperator{\Grad}{grad}
\newcommand{\grad}{\Grad}

\newcommand{\norm}[1]{\left|#1\right|}
\newcommand{\onorm}[1]{\left\|#1\right\|}

\newcommand{\abs}[1]{\left|#1\right|}
\newcommand{\transp}[1]{{#1}^\mathsf{T}}%\intercal}%\mathsf{T}}

\newcommand{\inv}[1]{{#1}^{-1}}
\newcommand{\invtransp}[1]{{#1}^{-T}}

\newcommand{\bcdot}{\boldsymbol{\cdot}}
\newcommand{\dotprod}[2]{#1\bcdot#2}
\newcommand{\pts}{\mathsf{P}}

\newcommand{\tanspace}[2]{T_{#1}{#2}} % e.g. T_xS
\newcommand{\bdry}[1]{\partial{#1}}

\newcommand{\asimplex}[1]{\{#1\}} % abstract simplex, i.e. $\asimplex{i,j,k}$
\newcommand{\simplex}[1]{[#1]} % Euclidean simplex
\newcommand{\seg}[2]{\simplex{#1,#2}} % segment

\newcommand{\gdist}{d} % generic metric (distance function)

\newcommand{\gdistG}[1]{\gdist_{#1}} % metric on space to be specified
\newcommand{\distG}[3]{\gdist_{#1}(#2,#3)}

\newcommand{\close}[1]{\overline{#1}} % topological closure

\DeclareMathOperator{\interior}{int}
\newcommand{\intr}[1]{\interior(#1)}

\newcommand{\spaceball}[3]{B_{#1}(#2;#3)} % open ball on specified space
\newcommand{\cspaceball}[3]{\close{B}_{#1}(#2;#3)} % closed ball " " "

\DeclareMathOperator{\aff}{aff} % affine hull
\newcommand{\affhull}[1]{\aff(#1)}
\newcommand{\splxs}{\sigma}
\newcommand{\splxt}{\tau}
\newcommand{\splxu}{\mu}

\newcommand{\tsplxs}{\tilde{\splxs}}

\newcommand{\opface}[2]{{#2}_{#1}} % e.g., \sigma_p : face opposite p
\newcommand{\splxsp}{\opface{p}{\splxs}}

\newcommand{\thickbnd}{t_0} %{\Upsilon_0}

\newcommand{\gthickness}{t} %{\Upsilon}
\newcommand{\thickness}[1]{\gthickness(#1)} 
\newcommand{\fatness}[1]{\Theta(#1)}

\DeclareMathOperator{\vol}{vol}
\newcommand{\volop}[1]{\vol(#1)}

\newcommand{\gsplxalt}{a}
\newcommand{\glongedge}{L}
\newcommand{\splxalt}[2]{\gsplxalt_{#1}(#2)} % altitude of #1 in simplex #2
\newcommand{\longedge}[1]{\glongedge(#1)}
\newcommand{\scale}{h} 

\newcommand{\man}{M}

\DeclareMathOperator{\St}{St}
\newcommand{\str}[1]{\underline{\St}(#1)}
\newcommand{\carrier}[1]{\abs{#1}} % carrier of simplicial complex

\newcommand{\rem}{\reel^m}
\newcommand{\ren}{\reel^n}
\newcommand{\gdistEn}{\gdistG{\ren}} % Euclidean metric on R^n
\newcommand{\distEn}[2]{\distG{\ren}{#1}{#2}}

\newcommand{\ballEn}[2]{\spaceball{\ren}{#1}{#2}} % open ball in R^n
\newcommand{\cballEn}[2]{\cspaceball{\ren}{#1}{#2}} % closed ball in R^n

\newcommand{\genfct}[1]{\mathcal{E}_{#1}}
\newcommand{\genfctbc}{\genfct{\ggbarycoord}}
\newcommand{\enfct}[2]{\genfct{#1}(#2)}
\newcommand{\enfctbc}[1]{\genfctbc(#1)}

\newcommand{\gdistM}{\gdistG{\man}} % intrinsic metric on manifold
\newcommand{\distM}[2]{\distG{\man}{#1}{#2}}

\newcommand{\gdistA}{\gdistG{\acplx}} % intrinsic metric on complex
\newcommand{\distA}[2]{\distG{\acplx}{#1}{#2}}

\newcommand{\ballM}[2]{\spaceball{\man}{#1}{#2}}
\newcommand{\cballM}[2]{\cspaceball{\man}{#1}{#2}}

\newcommand{\ginjrad}{\iota} % injectivity radius
\newcommand{\injrad}[1]{\ginjrad(#1)}
\newcommand{\injradM}{\iota_{\man}}

\newcommand{\curvupbnd}{\Lambda_{+}} %{\mathcal{K}}
\newcommand{\loccurvupbnd}[1]{\curvupbnd(#1)}
\newcommand{\curvlowbnd}{\Lambda_{-}} %{\delta}
\newcommand{\curvabsbnd}{\Lambda}

\newcommand{\stdsplx}[1]{\boldsymbol{\Delta}^{#1}}
\newcommand{\stdsplxn}{\stdsplx{n}}

\newcommand{\riemsplxs}{\gsplxs_M}
\newcommand{\riemsplxt}{\gsplxt_M}
\newcommand{\riemsplxu}{\gsplxu_M}

\newcommand{\ggbarycoord}{\lambda} % vector of bary coords
\newcommand{\gbarycoord}[1]{\ggbarycoord_{#1}} % subscript is vtx index
\newcommand{\barycoord}[2]{\gbarycoord{#1}(#2)}

\newcommand{\barymapradbnd}{\rho_0}
\newcommand{\bmaprad}{\rho}
\newcommand{\bmapball}{B_{\bmaprad}}

\newcommand{\ggbarymap}{\mathcal{B}}
\newcommand{\gbarymap}[1]{\ggbarymap_{#1}}
\newcommand{\gbarymapsig}{\ggbarymap_{\splxs}}

\newcommand{\conn}{\nabla} % connection
\newcommand{\sing}[2]{s_{#1}(#2)}
\newcommand{\metlipconst}{C_0}
\newcommand{\pseudoinv}[1]{{#1}^{-1}_{\text{left}}} %{#1^\dagger}

\newcommand{\splxsE}{\gsplxs_{\E}}
\newcommand{\tsplxsE}{\gtsplxs_{\E}}
\newcommand{\splxtE}{\gsplxt_{\E}}

\newcommand{\bmap}{b}
\newcommand{\gglinmap}{\mathcal{L}}

\newcommand{\lmap}{\mathcal{L}}

\newcommand{\bfS}{\mathbf{S}} % Gert doesn't like bold, but it is std

\newcommand{\cnstcurv}{\kappa}
\newcommand{\seccurv}{K}

\newcommand{\acplx}{\mathcal{A}}
\newcommand{\ccplx}{\mathcal{C}}

\newcommand{\gsplxs}{\boldsymbol{\sigma}}
\newcommand{\gsplxt}{\boldsymbol{\tau}}
\newcommand{\gsplxu}{\boldsymbol{\mu}}

\newcommand{\gtsplxs}{\boldsymbol{\tilde{\sigma}}}

\usepackage[pagebackref=true]{hyperref}

\title{Riemannian simplices and triangulations
}

\author{Ramsay Dyer\thanks{Johann Bernoulli Institute,
        Rijksuniversiteit Groningen} \thanks{{\tt r.h.dyer@rug.nl}}
        \and
        Gert Vegter\footnotemark[1] \thanks{{\tt g.vegter@rug.nl}}
        \and
        Mathijs Wintraecken\footnotemark[1]
        \thanks{{\tt m.h.m.j.wintraecken@rug.nl}}}

\begin{document}

\pagenumbering{roman}
\maketitle

% -*- LaTeX -*-
% abstract.tex
% 20140221
% 
% for riem-tri

\begin{abstract}
  We study a natural intrinsic definition of geometric simplices in Riemannian manifolds of arbitrary dimension $n$, and exploit these simplices to obtain  criteria for triangulating compact Riemannian manifolds.  These geometric simplices are defined using Karcher means. Given a finite set of vertices in a convex set on the manifold, the point that minimises the weighted sum of squared distances to the vertices is the Karcher mean relative to the weights. Using barycentric coordinates as the weights, we obtain a smooth map from the standard Euclidean simplex to the manifold.  A Riemannian simplex is defined as the image of this barycentric coordinate map. In this work we articulate criteria that guarantee that the barycentric coordinate map is a smooth embedding. If it is not, we say the Riemannian simplex is degenerate.  Quality measures for the ``thickness'' or ``fatness'' of Euclidean simplices can be adapted to apply to these Riemannian simplices. For manifolds of dimension 2, the simplex is non-degenerate if it has a positive quality measure, as in the Euclidean case. However, when the dimension is greater than two, non-degeneracy can be guaranteed only when the quality exceeds a positive bound that depends on the size of the simplex and local bounds on the absolute values of the sectional curvatures of the manifold. An analysis of the geometry of non-degenerate Riemannian simplices leads to conditions which guarantee that a simplicial complex is homeomorphic to the manifold.
\end{abstract}


\thispagestyle{empty}

\clearpage
\tableofcontents

\clearpage
\pagenumbering{arabic}
% -*- LaTeX -*-
% intro.tex
% 20131106
% Introduction for the Riemannian Simplex CGL report

\section{Introduction}
\label{sec:intro}

In this work we study a natural definition of geometric simplices in Riemannian manifolds of arbitrary finite dimension. The definition is intrinsic; the simplex is defined by the positions of its vertices in the manifold, which need not be embedded in an ambient space.  The standard definition of a Euclidean simplex as the convex hull of its vertices is not useful for defining simplices in general Riemannian manifolds. Besides the problem that convex hulls are difficult to compute (almost nothing is known about the convex hull of three distinct points, for example \cite[\S 6.1.3]{Berger}), the resulting objects could not be used as building blocks for triangulations, i.e., they cannot be used to define geoemetric simplicial complexes. This is because if two full dimensional convex simplices share a boundary facet, that facet must itself be convex. This constrains the facet to lie on a totally geodesic submanifold (i.e., minimising geodesics between points on the facet must lie in the facet), and when the curvature is not constant such submanifolds cannot be expected to exist (see \cite[Thm 58]{Berger} or \cite[\S 11]{chen2000}).

Given the vertices, a geometric Euclidean simplex can also be defined as the domain on which the barycentric coordinate functions are non-negative. This definition \emph{does} extend to general Riemannian manifolds in a natural way. The construction is based on the fact that the barycentric coordinate functions can be defined by a ``centre of mass'' construction.  Suppose $\{v_0, \ldots, v_n\} \subset \ren$, and $\{\gbarycoord{i}\}_{0 \leq i \leq n}$ is a set of non-negative weights that sum to $1$. If $u$ is the point that minimises the function
\begin{equation}
  \label{eq:en.fct.eucl}
  y \mapsto \sum_{i=0}^n \gbarycoord{i}\distEn{y}{v_i}^2,
\end{equation}
where $\distEn{x}{y} = \norm{x-y}$ is the Euclidean distance, then $u = \sum \gbarycoord{i} v_i$, and the $\{\gbarycoord{i}\}$ are the barycentric coordinates of $u$ in the simplex $\simplex{v_0, \ldots, v_n}$.

We can view a given set of barycentric coordinates $\ggbarycoord = (\gbarycoord{0},\ldots,\gbarycoord{n})$ as a point in $\R^{n+1}$. The set $\stdsplxn$ of all points in $\R^{n+1}$ with non-negative coefficients that sum to $1$ is called the \defn{standard Euclidean $n$-simplex}. Thus the minimisation of the function~\eqref{eq:en.fct.eucl} defines a map from the standard Euclidean simplex to the Euclidean simplex $\simplex{v_0,\ldots,v_n} \subset \ren$

If instead the points $\{v_i\}$ lie in a convex set $W$ in a Riemannian manifold $\man$, then, by using the metric of the manifold instead of $\gdistEn$ in \Eqnref{eq:en.fct.eucl}, we obtain a function $\genfctbc: W \to \R$ that has a unique minimum $x \in W$, provided $W$ is sufficently small (See \Secref{sec:riem.com}).  In this way we obtain a mapping $\ggbarycoord \mapsto x$ from $\stdsplxn$ to $W$. We call the image of this map an \defn{intrinsic simplex}, or a \defn{Riemannian simplex}.

% (Euclidean) space, in the first part. These triangles will be
% defined using the centre of mass. This method aims to generalize the
% use of Barycentric coordinates, introduced by Mobius in 1827
% \cite{Mobius}, which we use in Euclidean space to define simplices.

\subsection{Previous work}

\Eqnref{eq:en.fct.eucl} defines a point with given barycentric coordinates as a weighted centre of mass.  Centres of mass were apparently introduced in this context in 1929 by Cartan \cite{Cartan} for a finite number of points in a symmetric setting \cite[\S 6.1]{Berger}. Fr{\'e}chet also studied such functions in a more general setting, with integrals instead of sums, in 1948 \cite{Frechet}.  However, Karcher~\cite{Karcher} gave an extensive treatment particular to the Riemannian setting, and averages defined in this way are often referred to as ``Karcher means''.

Karcher's exposition~\cite{Karcher} is the standard reference for Karcher means. However, for our purposes a particularly good resource is the work by Buser and Karcher~\cite[\S 6, \S 8]{buser1981}. This work was exploited by Peters~\cite{peters1984}, where Karcher means are used to interpolate between locally defined diffeomorphisms between manifolds in order to construct a global diffeomorphism in a proof of Cheeger's finiteness theorem. Chavel~\cite[Ch. IX]{Chavel} gives a detailed exposition of Peters's argument. Kendal~\cite{Kendall} provides another important reference for Karcher means.  Riemannian simplices are not explicitly considered in any of these works.

More recently, Rustamov~\cite{Rustamov} introduced barycentric coordinates on a surface via Karcher means. Sander~\cite{Sander1} used the method in arbitrary dimensions to define Riemannian simplices as described above. He called them \defn{geodesic finite elements}, reflecting the application setting in numerical solutions to partial differential equations involving functions which take values in a manifold. Independently, von Deylen~\cite{vonDeylen2014} has treated the question of degeneracy of Riemannian simplices. His work includes a detailed analysis of the geometry of the barycentric coordinate map, and several applications. He does not address the problem of sampling criteria for triangulation. 

Our work is motivated by a desire to develop sampling requirements for representing a compact smooth Riemannian manifold with a simplicial complex. By this we mean that we seek conditions on a finite set $\pts \subset \man$ that guarantee that $\pts$ can be the vertex set of an (abstract) simplicial complex that is homeomorphic to $\man$.  We are particularly interested in manifolds of dimension greater than 2. For 2-dimensional manifolds a triangulation is guaranteed to exist when $\pts$ meets density requirements that can be specified either in terms of extrinsic criteria, for surfaces embedded in Euclidean space~\cite{boissonnat2005gm}, or in terms of intrinsic criteria~\cite{leibon1999,dyer2008sgp}.  In higher dimensions, although it is well known that a smooth manifold admits a %piecewise linear 
triangulation, to the best of our knowledge well founded sampling conditions sufficient to guarantee the existence of a triangulation with a given sample points as vertices have yet to be described.

For arbitrary finite dimension, Cairns~\cite{cairns1934} first demonstrated that a smooth compact manifold admits a triangulation by embedding Euclidean complexes into the manifold via coordinate charts, and showing that if the complexes were sufficiently refined the embedding maps could be perturbed so that they remain embeddings and the images of simplices coincide where patches overlap, thus constructing a global embedding of a complex. Whitehead~\cite{whitehead1940} refined the technique into a general approximation theory which is described in detail by Munkres~\cite{munkres1968} and is not restricted to compact manifolds. Whitney~\cite{whitney1957} used his result that a manifold can be embedded into Euclidean space to triangulate the manifold by intersecting it with a fine Cartesian grid in the ambient space. The problem has been revisited more recently in the computational geometry community, where the focus is on the algorithm used to construct a triangulation when a compact submanifold is known only through a finite set of sample points. Cheng et al.~\cite{cheng2005} used the generic triangulation result of Edelsbrunner and Shah~\cite{edelsbrunner1997rdt} to argue that a weighted Delaunay complex will triangulate a manifold, and Boissonnat and Ghosh~\cite{boissonnat2014tancplx.dcg} adapted Whitney's argument to demonstrate a triangulation by a Delaunay-based complex whose computation does not involve the ambient dimension.

In every case a metric on the manifold was employed in the computation of the triangulation. However, for all of these results there is no explicit expression to describe the sampling density sufficient to guarantee a triangulation. There is only the assurance that if the maximum distance between adjacent vertices is small enough, a triangulation may be obtained. In fact, in all of these results, the required density depends not only on the geometric properties of the manifold, but also on the geometric properties of the simplices that are involved in the construction. Some measure of simplex quality is introduced, and a lower bound on this quality measure is an essential component of the construction. This dependence on simplex quality is also present in our results, but we actually quantify what bound on the edge lengths is small enough to ensure a triangulation, given a bound on the quality of the simplices.

\subsection{Overview}

When we speak about the quality of a Euclidean simplex, we are referring to a function that parameterises how close the simplex is to being degenerate. A common quality measure for an $n$-simplex is the ratio of the volume to the $n^{\text{th}}$ power of the longest edge length. Another useful quality measure is the ratio of the smallest altitude to the longest edge length. A Euclidean simplex is degenerate if and only if its quality measure vanishes.

In this work we shed light on the relationship between the local curvature in the manifold, and the size and quality of the simplices involved in a triangulation. We articulate explicit criteria that are sufficient to guarantee that a simplicial complex with vertices on the manifold is homeomorphic to the manifold.  The intrinsic simplices defined by the centre of mass construction provide a convenient tool for this purpose.

Although the idea of Riemannian simplices defined in this way has been in the mathematical community for some time (see Berger~\cite[6.1.5]{Berger}), we are not aware of any published work exploiting the notion (of simplices in particular) prior to that of Rustamov~\cite{Rustamov} and Sander~\cite{Sander1}. For our purposes we need to establish a property that Sander did not consider. We need to ensure that the map from the Euclidean simplex to the manifold is a smooth embedding (i.e., the map extends to a smooth map from an open neighbourhood of the Euclidean simplex). This ensures that the barycentric coordinates mapped to the manifold do in fact provide a local system of coordinates. If the map is not a smooth embedding, we call the Riemannian simplex \defn{degenerate}.

A Euclidean simplex is non-degenerate if and only if its vertices are affinely independent. We show that a Riemannian simplex is non-degenerate if and only if for every point in the simplex the vertices are affinely independent when they are lifted by the inverse of the exponential map to the tangent space of that point.

In a two dimensional manifold this condition is satisfied for a triangle as long as the vertices do not lie on a common geodesic. Similar to the Euclidean case, such a configuration can be avoided by applying an arbitrarily small perturbation to the vertices. However, when the dimension is greater than two, a non-trivial constraint on simplex quality is required. In dimension 2 a sampling density for triangulation can be specified in terms of the convexity radius \cite{leibon1999,dyer2008sgp} (maximal radius for which a geodesic ball is convex, see \Secref{sec:riem.com}), and depends only on an upper bound on the sectional curvatures (\Lemref{lem:convexity.bnd}). In dimension higher than 2, we require the simplex size (maximum edge length) to also be constrained by a lower bound on the sectional curvatures (the upper bound on the edge lengths is inversely proportional to the square root of an upper bound on the absolute value of the sectional curvatures), so we cannot express the sampling density requirements in terms of a convexity radius alone.

We may define a quality measure for a Riemannian simplex by considering the quality of the Euclidean simplex obtained by lifting the vertices to the tangent space at one of the vertices. For our purposes we require a lower bound on the smallest such quality measure when each of the vertices is considered.

The quality of the Riemannian simplex that is required to ensure that it is non-degenerate depends on the maximum edge length, as well as the magnitude of the sectional curvatures in the neighbourhood. We establish this relationship with the aid of the Rauch comparison theorem, which provides an estimate on the differences in edge lengths of Euclidean simplices obtained by lifting the vertices of the Riemannian simplex to different tangent spaces.  By exploiting previously established bounds on the degradation of the quality of a Euclidean simplex under perturbations of the edge lengths \cite{boissonnat2013manmesh.inria}, we establish conditions that guarantee that the Riemannian simplex is non-degenerate.

We use this result to establish conditions that guarantee that a simplicial complex is homeomorphic to the manifold. This is the primary motivation for this work. Given an abstract simplicial complex whose vertex set is identified with points on the manifold, we are ensured that it triangulates the manifold if certain conditions are met, the principle one being a relationship between the size and quality of the Riemannian simplices.

\subsection{Outline and main results}

In \Secref{sec:riem.splxs} we present the framework for centre of mass constructions, and introduce the barycentric coordinate map and Riemannian simplices. Riemannian simplices are defined (\Defref{def:riem.splx}) as the image of the barycentric coordinate map, so they are ``filled in'' geometric simplices. A Riemannian simplex $\riemsplxs$ is defined by its vertices $\splxs = \asimplex{p_0, \ldots, p_n} \subset \man$, which are constrained to lie in a convex neighbourhood $\bmapball \subset \man$. For any $x \in \bmapball$ we define a Euclidean simplex $\splxs(x) \subset \tanspace{x}{\man}$ by $\splxs(x) = \asimplex{v_0(x), \ldots, v_n(x)}$, where $v_i(x) = \exp_x^{-1}(p_i)$. In general we use a boldface symbol when we are referring to a simplex as a set of non-negative barycentric coordinates, and normal type refers to the finite vertex set; the convex hull of $\splxs(x)$ is $\splxsE(x)$.

In \Secref{sec:implicit.fct.affine} we give a characterisation of non-degenerate Riemannian simplices in terms of affine independence.  We show that $\riemsplxs$ is non-degenerate if and only if $\splxs(x)$ is non-degenerate for every $x \in \riemsplxs$.

In \Secref{sec:non.degen} we establish criteria to ensure that a Riemannian simplex is non-degenerate. We first review properties of Euclidean simplices, including \defn{thickness}, the quality measure we employ. The thickness is essentially the ratio of the smallest altitude to the longest edge length of the simplex.  If the edge lengths in a Euclidean simplex change by a small amount, we can quantify the change in the thickness. In particular, if $F: \ren \to \ren$ is a bi-Lipschitz map, we can quantify a bound on the thickness, $\thickness{\splxs}$, of a simplex $\splxs$ relative to the metric distortion (i.e., the bi-Lipschitz constant) that establishes when the Euclidean simplex $F(\splxs)$ is non-degenerate.

The Rauch theorem establishes bounds on the norm of the differential of the exponential map, relative to the sectional curvatures. Using this we obtain a bound on the metric distortion of the transition function
\begin{equation}
  \label{eq:exp.transition}
  \exp_x \circ \exp_p^{-1}: T_p\man \to T_x\man
\end{equation}
which maps $\splxs(p)$ to $\splxs(x)$, and so we are able to establish conditions ensuring that $\splxs(x)$ is non-degenerate, based on quality assumptions on $\splxs(p)$. 

An open \defn{geodesic ball} of radius $r$ centred at $x \in \man$ is the set $\ballM{x}{r}$ of all points in $\man$ whose geodesic distance from $x$ is less than $r$. The \defn{injectivity radius at $x$}, denoted $\injrad{x}$, is the supremum of the radii $r$ for which $\exp_x$ restricts to a diffeomorphism between the Euclidean ball of radius $r$ centred at $0\in T_x\man$, and $\ballM{x}{r}$.  The injectivity radius of $\man$ is the infimum of $\injrad{x}$ over all $x\in \man$, and is denoted $\injradM$.
\begin{dup}[\Thmref{thm:thick.nondegen.riem}~(Non-degeneracy criteria)]
%  \label{thm:thick.nondegen.riem}
  Suppose $\man$ is a Riemannian manifold with sectional curvatures $\seccurv$ bounded by $\abs{\seccurv} \leq \curvabsbnd$, and $\riemsplxs$ is a Riemannian simplex, with $\riemsplxs \subset \bmapball \subset \man$, where $B_{\rho}$ is an open geodesic ball of radius $\rho$ with
  \begin{equation}
    \label{eq:good.barymap.rad}
    \rho < \barymapradbnd = \min \left \{ \frac{\injradM}{2},
      \frac{\pi}{4\sqrt{\curvabsbnd}} \right \}.
  \end{equation}
  Then $\riemsplxs$ is non-degenerate if there is a point $p \in \bmapball$ such that the lifted Euclidean simplex $\splxs(p)$ has thickness satisfying
%  \begin{equation*}
%    \thickness{\splxs(p)} > 10 \sqrt{\curvabsbnd}\rho.
%  \end{equation*}
% The ball $\bmapball$ may be chosen so that this inequality is necessarily satisfied if
\begin{equation}
  \label{eq:thickness.vs.length}
   \thickness{\splxs(p)} > 10 \sqrt{\curvabsbnd}\longedge{\riemsplxs},
 \end{equation}
 where $\longedge{\riemsplxs}$ is the geodesic length of the longest edge in $\riemsplxs$.
\end{dup}

In \Secref{sec:triangulation} we develop our sampling criteria for triangulating manifolds. We establish properties of maps whose differentials are bounded close to a fixed linear isometry, and use these properties to reveal conditions under which a complex will be embedded into a manifold. We then exploit a refinement of the Rauch theorem, and other estimates established by Buser and Karcher~\cite{buser1981}, to bound the differential of the barycentric coordinate map in this way. 

If $p$ is a vertex in an abstract simplicial complex $\acplx$, we define the \defn{star of $p$} to be the subcomplex $\str{p}$ of $\acplx$ consisting of all simplices that contain $p$, together with the faces of these simplices. The underlying topological space (or \defn{carrier}) of a complex $\acplx$ is denoted $\carrier{\acplx}$. We say that $\str{p}$ is a \defn{full star} if $\carrier{\str{p}}$ is a closed topological ball of dimension $n$ with $p$ in its interior, and $\acplx$ contains no simplices of dimension greater than $n$.  We have:
\begin{dup}[ \Thmref{thm:triangulation}~(Triangulation criteria)]
  Suppose $\man$ is a compact $n$-dimensional Riemannian manifold with sectional curvatures $\seccurv$ bounded by $\abs{\seccurv} \leq \curvabsbnd$, and $\acplx$ is an abstract simplicial complex with finite vertex set $\pts \subset \man$.  
  Define a quality parameter $\thickbnd > 0$, and let
  \begin{align}
    \label{eq:sampling.density}
    \scale = \min \left\{ \frac{\injradM}{4},
      \frac{\sqrt{n}\thickbnd}{6 \sqrt{\curvabsbnd}} \right\}.
  \end{align}
  If
  \begin{enumerate}
  \item For every $p \in \pts$, the vertices of $\str{p}$ are contained in $\ballM{p}{h}$, and the balls $\{\ballM{p}{h}\}_{p \in \pts}$ cover $\man$.
  \item For every $p \in \pts$, the restriction of the inverse of the exponential map $\exp_p^{-1}$ to the vertices of $\str{p} \subset \acplx$ defines a piecewise linear embedding of $\carrier{\str{p}}$ into $\tanspace{p}{\man}$, realising $\str{p}$ as a full star such that every simplex $\splxs(p)$ has thickness $\thickness{\splxs(p)} \geq \thickbnd$.
  \end{enumerate}
  then $\acplx$ triangulates $\man$, and the triangulation is given by the barycentric coordinate map on each simplex.
\end{dup}

The techniques employed to obtain \Thmref{thm:triangulation} exploit stronger bounds on the differential of the exponential map, and provide a slightly better bound for non-degeneracy than the one stated in \Thmref{thm:thick.nondegen.riem}, but at the expense of a stronger constraint on the allowed diameter of the simplex. This is the reason \Eqnref{eq:thickness.vs.length} appears as a stronger constraint on the thickness than the curvature controlled part of \Eqnref{eq:sampling.density}.

We refer to the criteria of \Thmref{thm:triangulation} as sampling criteria, even though they require a simplicial complex for their definition. Although there is no explicit constraint on the minimal distance between points of $\pts$, one is implicitly imposed by the quality constraint on the Riemannian simplices. The required sampling density depends on the quality of the Riemannian simplices, which leaves open the question of what kind of quality of simplices can we hope to attain. Recent work~\cite{boissonnat2013manmesh.inria} constructs a Delaunay complex conforming to the requirements of \Thmref{thm:triangulation} with the thickness $\thickbnd$ bounded by $\Omega(2^{-n^3})$. It would be interesting to see this improved.

The complex $\acplx$ in \Thmref{thm:triangulation} naturally admits a piecewise linear metric by assigning edge lengths to the simplices given by the geodesic distance in $\man$ between the endpoints. In \Secref{sec:pwf.metric} we observe that in order to ensure that this does in fact define a piecewise-flat metric, we need to employ slightly stronger constraints on the scale parameter $h$. In this case, the complex $\acplx$ becomes a good geometric approximation of the original manifold, and we find:
\begin{dup}[\Thmref{thm:metric.distortion} (Metric distortion)]
  If the requirements of \Thmref{thm:triangulation}, are satisfied with the scale parameter~\eqref{eq:sampling.density} replaced by
  \begin{equation*}
    \scale = \min \left\{ \frac{\injradM}{4},
      \frac{\thickbnd}{6 \sqrt{\curvabsbnd}} \right\},
  \end{equation*}
  then $\acplx$ is naturally equipped with a piecewise flat metric $\gdistA$ defined by assigning to each edge the geodesic distance in $\man$ between its endpoints.

  If $H:\carrier{\acplx} \to \man$ is the triangulation defined by the barycentric coordinate map in this case, then the metric distortion induced by $H$ is quantified as
  \begin{equation*}
    \abs{ \distM{H(x)}{H(y)} - \distA{x}{y} } \leq \frac{50 \curvabsbnd \scale^2}{\thickbnd^2} \distA{x}{y},
  \end{equation*}
for all $x,y \in \carrier{\acplx}$.
\end{dup}

The criteria of these three theorems can also be formulated in terms of the thickness of the Euclidean simplices defined by the geodesic edge lengths of the Riemannian simplices, rather than the Euclidean simplices we find in the tangent spaces.  In \Appref{sec:alt.criteria} we briefly mention this alternative formulation of our results. We also compare the thickness quality measure for simplices with a commonly used volumetric quality measure which we call fatness.

In \Appref{sec:toponogov} an alternate approach to non-degenerate Riemannian simplices is presented. This approach is based on bounding angles and edge lengths in geodesic triangles via the Toponogov comparison theorem.

% -*- LaTeX -*-
% riem_com.tex
% 20131107
% Riemannian centre of mass and Riemannian simplices

\section{Riemannian simplices}
\label{sec:riem.splxs}

In this section we summarise the results of the theory of Riemannian centres of mass that we need in order to define Riemannian simplices. We then give an explicit description of the barycentric coordinate map that is used to define these simplices. We take the view that if the barycentric coordinate map is well defined, then the simplex is well defined, but it may be degenerate. The geodesic finite elements employed by Sander~\cite{Sander1} are Riemannian simplices without a requirement of non-degeneracy. In \Secref{sec:implicit.fct.affine} we demonstrate that non-degeneracy of a Riemannian simplex $\riemsplxs$ is characterised by the affine independence of the vertices when lifted to the tangent space of any point in $\riemsplxs$.

\subsection{Riemannian centre of mass}
\label{sec:riem.com}

We work with an $n$-dimensional Riemannian manifold $\man$.  The centre of mass construction developed by Karcher~\cite{Karcher} hinges on the notion of convexity in a Riemannian manifold. A set $B \subseteq \man$ is \defn{convex} if any two points $x,y \in B$ are connected by a minimising geodesic $\gamma_{xy}$ that is unique in $\man$, and contained in $B$. For $c \in \man$, the geodesic ball of radius $r$ is the set $\ballM{c}{r}$ of points in $\man$ whose distance from $c$ is less than $r$, and we denote its closure by $\cballM{c}{r}$. If $r$ is small enough, $\cballM{c}{r}$ will be convex; the following lemma quantifies ``small enough''. 

In order to obtain non-degeneracy criteria for Riemannian simplices we require both an upper and a lower bound on the sectional curvatures, so it is convenient to work with a bound $\curvabsbnd$ on the absolute value of the sectional curvatures, $\abs{\seccurv} \leq \curvabsbnd$. However, the \emph{definition} of Riemannian simplices only requires an upper bound on the sectional curvatures. In order to emphasise this we introduce distinct symbols for the upper and lower bounds on the sectional curvatures. Thus $\curvlowbnd \leq \seccurv \leq \curvupbnd$, and $\curvabsbnd = \max \{\curvupbnd, - \curvlowbnd \}$.

We have \cite[Thm. IX.6.1]{Chavel}: 
% \rd{change citation to B and K 6.4.6; they only demand $\injrad{x}/2$ bound and curvatures on the ball -- maybe not}
%
% NoteRD: We have made the bound on the radius strict, but allowed the ball to be closed. This works, for otherwise, there would be some minimising geodesic not contained in the closed ball. Then we could find a slightly larger open ball that was not convex, contradicting the statement in Chavel.
\begin{lem}
  \label{lem:convexity.bnd}
  Suppose the sectional curvatures of $\man$ are bounded by $\seccurv \leq \curvupbnd$, and $\injradM$ is the injectivity radius. If
  \begin{equation*}
    r < \min \left \{ \frac{\injradM}{2},
      \frac{\pi}{2\sqrt{\curvupbnd}} \right \},
  \end{equation*}
  then $\cballM{x}{r}$ is convex. (If $\curvupbnd \leq 0$, we take $1 / \sqrt{\curvupbnd}$ to be infinite.)
\end{lem}
\begin{remark}
  \Lemref{lem:convexity.bnd} is stated in terms of global bounds on the injectivity radii and sectional curvatures (on a non-compact manifold, these may be useless), but really we only need these bounds in a neighbourhood of $x$. Let $K(x)$ be an upper bound on the sectional curvatures at $x$, and denote the injectivity radius at $x$ by $\injrad{x}$. Now define $I(x)$ and $\loccurvupbnd{x}$ to be the infimum and supremum respectively of $\injrad{y}$ and $K(y)$, where $y$ ranges over the ball $\ballM{x}{R}$ of radius
  \begin{equation*}
    R = \min \left\{ \frac{\injrad{x}}{2}, \frac{\pi}{2\sqrt{K(x)}} \right\}.
  \end{equation*}
  Then \Lemref{lem:convexity.bnd} holds if $\injradM$ and $\curvupbnd$ are replaced by $I(x)$ and $\loccurvupbnd{x}$ respectively in the bound on $r$. For simplicity, we will continue to refer to global bounds, but everywhere they occur a similar remark applies.

  Also, in all cases where an upper bound on the sectional curvatures is employed, this bound is only relevant when it is positive. If $\man$ has non-positive curvature, then $1 / \sqrt{\curvupbnd}$ may be assumed to be infinite.
\end{remark}

In our context, we are interested in finding a weighted centre of mass of a finite set $\asimplex{p_0, \ldots , p_j} \subset B \subset M$, where the containing set $B$ is open, and its closure $\close{B}$ is convex. The centre of mass construction is based on minimising the function $\genfctbc: \close{B} \to \reel$ defined by
\begin{equation}
  \label{eq:en.fct.defn}
  \enfctbc{x} = \frac{1}{2} \sum_i \gbarycoord{i} \distM{x}{p_i}^2,
\end{equation}
where the $\gbarycoord{i} \geq 0$ are non-negative weights that sum to $1$, and $\gdistM$ is the geodesic distance function on $\man$. Karcher's first simple observation is that the minima of $\genfctbc$ must lie in the interior of $\close{B}$, i.e., in $B$ itself.  This follows from considering the gradient of $\genfctbc$:
\begin{equation}
  \label{eq:en.fct.grad}
  \grad \enfctbc{x} = - \sum_i \gbarycoord{i} \exp^{-1}_x(p_i).
\end{equation}
At any point $x$ on the boundary of $\close{B}$, the gradient vector lies in a cone of outward pointing vectors. It follows that the minima of $\genfctbc$ lie in $B$. The more difficult result that the minimum is unique, Karcher showed by demonstrating that $\genfctbc$ is convex.  If $B \subseteq \man$ is a convex set, a function $f: B \to \reel$ is \defn{convex} if for any geodesic $\gamma: I \to B$, the function $f \circ \gamma$ is convex (here $I \subseteq \R$ is an open interval). If $f$ has a minimum in $B$, it must be unique. By \Eqnref{eq:en.fct.grad}, it is the point $x$ where
\begin{equation*}
  \sum_i \gbarycoord{i} \exp^{-1}_x(p_i) = 0.  
\end{equation*}

We have the following result \cite[Thm 1.2]{Karcher}:
\begin{lem}[Unique centre of mass]
  \label{RiemannianCentre}
  \label{lem:enfct.convex}
  If $\asimplex{p_0, \ldots, p_j} \subset B_{\rho} \subset \man$, and $B_{\rho}$ is an open ball of radius $\rho$ with
  \begin{equation*}
%    \label{eq:good.barymap.rad}
    \rho < \barymapradbnd = \min \left \{ \frac{\injradM}{2},
      \frac{\pi}{4\sqrt{\curvupbnd}} \right \},
  \end{equation*}
  then on any geodesic $\gamma: I \to B_{\rho}$, we have
  \begin{equation}
    \label{eq:en.fct.strictly.convex}
    \frac{\mathrm{d}^2}{\mathrm{d}t^2}\enfctbc{\gamma(t)} \geq C(\curvupbnd, \rho) > 0,
  \end{equation}
  where $C(\curvupbnd, \rho)$ is a positive constant depending only on $\curvupbnd$ and $\rho$. In particular, $\genfctbc$ is convex and has a unique minimum in $B_{\rho}$.
\end{lem}
Karcher gives an explicit expression for $C(\curvupbnd, \rho)$, but we will not need to refer to it here. Also, Karcher expresses the centre of mass concept in more generality by using an integral over a set whose measure is $1$, rather than a weighted sum over a finite set as we have used.

\subsection{The barycentric coordinate map}
\label{sec:bary.map}

Let $\stdsplx{j}$ denote the standard Euclidean $j$-simplex. This can be realised as the set of points $\ggbarycoord \in \reel^{j+1}$ whose components are non-negative, $\gbarycoord{i} \geq 0$, and sum to one: $\sum_i \gbarycoord{i} = 1$. We index the coordinates starting from zero: these are the \defn{barycentric coordinates} on the standard simplex.
\begin{de}[Riemannian simplex]
  \label{def:riem.splx}
  If a finite set $\splxs^j = \asimplex{p_0, \ldots, p_j} \subset \man$ in an $n$-manifold is contained in an open geodesic ball $B_{\rho}$ whose radius, $\rho$, satisfies \Eqnref{eq:good.barymap.rad}, then $\splxs^j$ is the set of vertices of a geometric \defn{Riemannian simplex}, denoted $\riemsplxs^j$, and defined to be the image of the map
  \begin{equation*}
    \begin{split}
      \gbarymap{\splxs^j} :
      &\stdsplx{j} \to \man \\
      &\ggbarycoord \mapsto \argmin_{x \in \close{B}_{\rho}} \enfct{\ggbarycoord}{x}.
    \end{split}
  \end{equation*}
  We say that $\riemsplxs^j$ is \defn{non-degenerate} if $\gbarymap{\splxs^j}$ is a smooth embedding; otherwise it is \defn{degenerate}.
\end{de}
Define an $i$-face of $\riemsplxs^j$ to be the image of an $i$-face of $\stdsplx{j}$.  Since an $i$-face of $\stdsplx{j}$ may be identified with $\stdsplx{i}$ (e.g., by an order preserving map of the vertex indices), the $i$-faces of $\riemsplxs^j$ are themselves Riemannian $i$-simplices. In particular, if $\splxt$ and $\splxu$ are the vertices of Riemannian simplices $\riemsplxt$ and $\riemsplxu$, and $\splxs^i = \splxt \cap \splxu$, then the Riemannian $i$-simplex $\riemsplxs^i$ is a face of both $\riemsplxt$ and $\riemsplxu$.
The \defn{edges} of a Riemannian simplex are the Riemannian $1$-faces. We observe that these are geodesic segments.
We will focus on full dimensional simplices, i.e., unless otherwise specified, $\riemsplxs$ will refer to a Riemannian simplex defined by a set $\splxs$ of $n+1$ vertices in our $n$-dimensional manifold $\man$.

\paragraph{Remarks}
%\label{sec:abt.riem.splxs}

The barycentric coordinate map $\gbarymap{\splxs}$ is differentiable. This follows from the implicit function theorem, as is shown by Buser and Karcher~\cite[\S 8.3.3]{buser1981}, for example. They work in local coordinates on the tangent bundle, and use the connection to split the derivative of $\grad \genfctbc: \man \to T\man$ into horizontal and vertical components.
% It is not clear to me why this is necessary, nor why their claim that the horizontal part is non-degenerate, nor why that claim is needed anyway.
The strict convexity condition~\eqref{eq:en.fct.strictly.convex} implies that the vertical component of the derivative is full rank, and permits the use of the implicit function theorem.

The argument of Buser and Karcher assumes that the map is defined on an open domain. We observe that $\gbarymap{\splxs}$ is well defined if we allow negative barycentric coordinates of small magnitude.  For a sufficiently small $\epsilon > 0$, \Lemref{RiemannianCentre} holds if the barycentric coordinates $\gbarycoord{i}$ satisfy $\sum \gbarycoord{i} = 1$ and $\gbarycoord{i} > - \epsilon$ for all $i \in \{0, \ldots, n\}$, albeit with $C(\curvupbnd, \rho)$ replaced with a smaller positive constant.  
% The integral in Karcher's proof of \Lemref{RiemannianCentre} is replaced by a sum in our context, where each term of the sum is weighted by a barycentric coordinate, but otherwise has the same form as the integrand in Karcher's expression. 
This follows from the observation that $\frac{d^2}{dt^2}\genfctbc$ is continuous in the barycentric coordinates, thus since it is strictly positive on the boundary of $\stdsplxn$, it can be extended to an open neighbourhood.  This means that $\gbarymap{\splxs}$ is smooth on the closed domain $\stdsplxn$, as defined in \Secref{sec:smooth}.

Karcher himself mentioned that his result can accommodate signed measures \cite[Remark 1.8]{Karcher}, and Sander has demonstrated this in some detail \cite{Sander2}. However, for our current purposes we are only claiming that we can accommodate arbitrarily small negative barycentric coordinates assuming the stated bound on $\barymapradbnd$ (\Eqnref{eq:good.barymap.rad}).

A Riemannian simplex is not convex in general, but as Karcher observed~\cite{Karcher}, being the image of the barycentric coordinate map, it will be contained in any convex set that contains the vertices of the simplex. Thus the Riemannian simplex is contained in the intersection of such sets.

\Eqnref{eq:good.barymap.rad} gives an upper bound on the size of a Riemannian simplex that depends only on the injectivity radius and an \emph{upper} bound on the sectional curvature. For example, in a non-positively curved manifold, the size of a well defined Riemannian simplex is constrained only by the injectivity radius. However, if the dimension $n$ of the manifold is greater than $2$, we will require also a \emph{lower} bound on the sectional curvatures in order to ensure that the simplex is non-degenerate. 
% This is a contrast with the two dimensional setting, where sampling conditions that depend only on an upper bound on the sectional curvature suffice to define a triangulation.

\Lemref{lem:enfct.convex} demands that a Riemannian simplex be contained in a ball whose radius is constrained by $\barymapradbnd$. Thus Riemannian simplices always have edge lengths less than $2\barymapradbnd$. If the longest edge length, $\longedge{\riemsplxs}$, of $\riemsplxs$ is less than $\barymapradbnd$, then $\riemsplxs$ must be contained in the closed ball of radius $\longedge{\riemsplxs}$ centred at a vertex. Indeed, any open ball centred at a vertex whose radius is larger than $\longedge{\riemsplxs}$, but smaller than $\barymapradbnd$, must contain the vertices and have a convex closure. The simplex is thus contained in the intersection of these balls.
%Therefore at least in this case, the diameter of the Riemannian simplex coincides with the longest edge length.
If $\longedge{\riemsplxs} \geq \barymapradbnd$, then a ball of  radius $\longedge{\riemsplxs}$ need not be convex. In this case we claim only that $\riemsplxs$ is contained in a ball of radius $2\barymapradbnd$ centred at any vertex.

    % -*- LaTeX -*-
% differential.tex
%
% 20140318 from rauch_non_degen.tex

\subsection{The affine independence criterion for non-degeneracy}
\label{sec:implicit.fct.affine}

In this subsection we show that a Riemannian simplex $\riemsplxs$ is non-degenerate if, and only if, for any $x \in \riemsplxs$, the lift of the vertices by the inverse exponential map yields a non-degenerate Euclidean simplex. We first introduce some notation and terminology to better articulate this statement.

\paragraph{Notation}
%\label{sec:notation}

A Euclidean simplex $\splxs$ of dimension $k$ is defined by a set of $k+1$ points in Euclidean space $\splxs = \asimplex{v_0, \ldots, v_k } \subset \ren$.  In general we work with abstract simplices, even though we attribute geometric properties to the simplex, inherited from the embedding of the vertices in the ambient space (see \Secref{sec:eucl.splxs}).  When we wish to make the dimension explicit, we write it as a superscript, thus $\splxs^k$ is a $k$-simplex.  Traditional ``filled in'' geometric simplices are denoted by boldface symbols; $\splxsE = \convhull{\splxs}$ is the convex hull of $\splxs$. If such a simplex is specified by a vertex list, we employ square brackets: $\splxsE = \simplex{v_0, \ldots, v_k}$.

The \defn{barycentric coordinate functions} $\{ \gbarycoord{i} \}$ associated to $\splxs$ are affine functions $\R^n \to \R$ that satisfy $\barycoord{i}{v_j} = \delta_{ij}$ and $\sum_{i=0}^n \gbarycoord{i} = 1$.  It is often convenient to choose one of the vertices, $v_0$ say, of $\splxs$ to be the origin. We let $P$ be the $n \times k$ matrix whose $i^{\text{th}}$ column is $v_i - v_0$. Then the barycentric coordinate functions $\{\gbarycoord{i}\}$ are linear functions for $i>0$, and they are dual to the basis defined by the columns of $P$. This means that if we represent the function $\gbarycoord{i}$ as a row vector, then the matrix $Q$ whose $i^{\text{th}}$ \emph{row} is $\gbarycoord{i}$ satisfies $QP = I_{k \times k}$. 

A full dimensional Euclidean simplex $\splxs$ is non-degenerate, if and only if the corresponding matrix $P$ is non-degenerate. In particular, if $\splxs$ is full dimensional (i.e., $k=n$), then $Q = \inv{P}$.  Suppose $\splxs \subset \ren$ is an $n$-simplex. If $\xi \in \ren$, let $\ggbarycoord(\xi) = \transp{(\barycoord{1}{\xi}, \ldots, \barycoord{n}{\xi})}$. Then $\ggbarycoord(\xi)$ is the vector of coefficients of $\xi - v_0$ in the basis defined by the columns of $P$. I.e., $\xi - v_0 = P\ggbarycoord(\xi)$.

We will be interested in Euclidean simplices that are defined by the vertices of a Riemannian simplex: If $\splxs = \asimplex{p_0, \ldots, p_n} \subset \bmapball \subset \man$ is the set of vertices of $\riemsplxs$, it is convenient to introduce the notation $v_i(x) = \exp_x^{-1}(p_i)$, and $\splxs(x) = \exp_x^{-1}(\splxs)$. Thus $\splxs(x) = \asimplex{v_0(x), \ldots, v_n(x)}$ is a Euclidean simplex in $\tanspace{x}{\man}$.

The norm of a vector $v$ in a Euclidean space is denoted $\norm{v}$. For example, if $v \in T_p\man$, then $\norm{v} = g(v,v)^{\frac{1}{2}}$, where $g$ is the Riemannian metric tensor on $\man$, and if $v \in \R^n$, then $\norm{v} = (\dotprod{v}{v})^{\frac{1}{2}}$. The differential of a map $F:\man \to \bar{\man}$ is denoted by $dF$; so $dF_x: T_x\man \to T_{F(x)}\bar{\man}$ is a linear map whose operator norm is $\onorm{dF_x}$. All differentiable maps, operators, and manifolds are assumed to be $\cinfty$.

\paragraph{An expression for the differential}

The expression for the differential obtained in \Eqnref{eq:diff.bary.map} below is obtained as a particular case of an argument presented by Buser and Karcher \cite[\S 8.3]{buser1981}. The argument was later exploited by Peters~\cite{peters1984} to sharpen bounds on Cheeger's finiteness theorem \cite{cheeger1970}. A thorough exposition appears also in Chavel~\cite[IX.8]{Chavel}. 
% The objective there is to define a differentiable map that is the average of a finite set of differentiable embeddings whose domains coincide. In our case, 

We work in a domain $U \subset \R^n$ defined by a chart $\phi: \man \supset W \to U$ such that $\bmapball \subset W$. Let $\tsplxs = \phi(\splxs)$ be the image of the vertices of a Riemannian $n$-simplex $\riemsplxs \subset \bmapball$. Label the vertices of $\tsplxs = \asimplex{v_0, \ldots, v_n}$ such that $v_i = \phi(p_i)$, and assume $v_0$ is at the origin. The affine functions $\gbarycoord{i}: u \mapsto \barycoord{i}{u}$ are the barycentric coordinate functions of $\tsplxs$.
%
% We identify $\tanspace{p_0}{\man}$ with $\ren$ by an arbitrary linear isometry, and work in a domain $U$ that contains $\exp_{p_0}^{-1}(\bmapball)$. \rd{I don't think this assumption is used: we don't need to assume the inverse exponential map is involved here, right? -- just need $v_0$ at the origin.} The affine functions $\gbarycoord{i}: u \mapsto \barycoord{i}{u}$ are the barycentric coordinate functions of $\splxs(p_0)$.
%
We consider $\grad \genfctbc$, introduced in \Eqnref{eq:en.fct.grad}, now to be a vector field that depends on both $u \in U$ and $x \in \bmapball$. Specifically, we consider the vector field $\nu: U \times \bmapball \to \tanspace{}{M}$ defined by
\begin{equation}
  \label{eq:defn.grad.field}
  \nu(u,x) = - \sum_{i=0}^n \barycoord{i}{u}v_i(x).
\end{equation}

Let $\bmap: \tsplxsE \to \riemsplxs$ be defined by $\bmap = \gbarymapsig \circ \gglinmap$, where $\gglinmap$ is the canonical linear isomorphism that takes the vertices of $\tsplxs$ to those of $\stdsplxn$, and $\gbarymapsig$ is the barycentric coordinate map introduced in \Defref{def:riem.splx}. This map is differentiable, by the arguments presented by Buser and Karcher, and $\nu(u, \bmap(u)) = 0$ for all $u \in \tsplxsE$. Regarded as a vector field along $\bmap$, the covariant differential $\conn\nu_{(u, \bmap(u))} = 0$ may be expanded as
\begin{equation*}
  \partial_u\nu + \left(\conn^{\man}\nu \right) d\bmap = 0,
\end{equation*}
where $\partial_u\nu$ denotes the differential of $\nu(u,x)$ with $x$ fixed, i.e.,
\begin{gather*}
  \partial_u\nu_{(u,x)}: \tanspace{u}{\ren} \to \tanspace{x}{\man}\\
  \left(\partial_u \nu_{(u,x)}\right)\dot{u}(0) = \frac{d}{d t} \nu(u(t),x) \big|_{t=0},
\end{gather*}
with $\dot{u}(0)$ denoting the tangent vector at $0$ to some curve $t \mapsto u(t)$ in $U\subset \ren$.
Similarly $\conn^{\man}\nu$ is the covariant differential when $u$ is fixed:
\begin{gather*}
  \conn^{\man} \nu_{(u,x)}: \tanspace{x}{\man} \to \tanspace{x}{\man}\\
  \left(\conn^{\man} \nu_{(u,x)}\right) \dot{x}(0) = D_t \nu(u, x(t)) \big|_{t=0},
\end{gather*}
where $D_t\nu$ is the covariant derivative along the curve $x(t)$. Finally $d\bmap: \tanspace{u}{\ren} \to \tanspace{x}{\man}$ is the differential of $b$, our barycentric coordinate map onto the Riemannian simplex $\riemsplxs$. 

Our objective is to exhibit conditions that ensure that $d\bmap$ is non-degenerate.  It follows from the strict convexity conditon~\eqref{eq:en.fct.strictly.convex} of \Lemref{lem:enfct.convex} that the map $\conn^{\man}\nu$ is non-degenerate. Indeed, if $v \in T_x\man$ for some $x \in \bmapball$, there is a geodesic $\gamma: I \to \bmapball$ with $\gamma'(0)=v$, and $\frac{d^2}{dt^2}\enfctbc{\gamma(t)}\big|_{t=0} = g(\conn^{\man}_{v}\nu,v)  > 0$.  Therefore, we have that
\begin{equation}
  \label{eq:diff.bary.map}
  d\bmap = - \left(\conn^{\man}\nu \right)^{-1} \partial_u \nu,
\end{equation}
and thus $d\bmap$ has full rank if and only if $\partial_u\nu$ has full rank.

% What the hell does that negative sign mean?
% Okay: it is because of the minus sign in the defn of $\nu$
% That is carried over to $\partial_u \nu$: it must be close to $-Id$.

\paragraph{The differential as a matrix}

Recalling \Eqnref{eq:defn.grad.field}, notice that when $x$ is fixed, $\nu$ is an affine map $\ren \supset U \to T_x\man$, and so $\left(\partial_u \nu \right)_v = \left(\partial_u \nu \right)_w$ for all $v,w \in U$.  We see that
\begin{equation*}
  \partial_u \nu = - \sum_{i=0}^n v_i(x)\, d\gbarycoord{i}.
\end{equation*}
Since $\sum_{i=0}^n \gbarycoord{i} = 1$, we have that $\sum_{i=0}^n \, d \gbarycoord{i} = 0$. We may thus write $d\gbarycoord{0} = - \sum_{i=1}^n d\gbarycoord{i}$, and so for $\xi \in \tanspace{u}{U}$, we have
\begin{equation}
  \label{eq:deriv.sum}
  \left(\partial_u \nu \right)\xi = - \sum_{i=1}^n  \big( v_i(x) - v_0(x) \big) d\barycoord{i}{\xi}\,.
\end{equation}

Now, since the domain of the barycentric coordinates is $U$, and the origin of $U \subset \ren$ coincides with $v_0$, the functions $\gbarycoord{i}$ for $i \in \{1,\ldots,n \}$ are \emph{linear} functions, and we use the canonical identification of tangent spaces in $\ren$ to conclude that $d\barycoord{i}{\xi} = \barycoord{i}{\xi}$, where in the right hand side we view $\xi$ as an element of $\ren$, rather than an element of $\tanspace{u}{\ren}$. As discussed above, we have $\ggbarycoord(\xi) = P^{-1}\xi$, where $\ggbarycoord(\xi) = \transp{(\barycoord{1}{\xi}, \ldots, \barycoord{n}{\xi})}$, and $P$ is the matrix whose $i^{\text{th}}$ column is $v_i$.
Thus, using an arbitrary linear isometry to get a coordinate system for $\tanspace{x}{\man}$, and letting $\tilde{P}$ be the matrix whose $i^{\text{th}}$ column is $(v_i(x) - v_0(x))$, we may rewrite \Eqnref{eq:deriv.sum} as
\begin{equation}
  \label{eq:matrix.of.deriv}
  \left(\partial_u \nu \right)\xi  = - \tilde{P}\ggbarycoord(\xi)
  = - \tilde{P}P^{-1}\xi.
\end{equation}

From \Eqnref{eq:matrix.of.deriv} we conclude that $\partial_u \nu$ is full rank if and only if $\tilde{P}$ is of full rank, and this is the case if and only if $\splxs(x)$ is a non-degenerate Euclidean simplex, i.e., its vertices $\{v_i(x)\}$ are affinely independent.

We observe that if $d\bmap$ is non-degenerate on $\riemsplxs$, then $\bmap$ must be injective. Indeed, if $x = \bmap(u)$, then $\{ \barycoord{i}{u} \}$, the barycentric coordinates of $u$ with respect to $\tsplxs$, are also the barycentric coordinates of the origin in $\tanspace{x}{\man}$, with respect to the simplex $\splxs(x)$. Thus if $\bmap(u) = x = \bmap(\tilde{u})$, then $\barycoord{i}{u} = \barycoord{i}{\tilde{u}}$, and we must have $\tilde{u} = u$ by the uniqueness of the barycentric coordinates.

In summary, we have
\begin{prop}
  \label{prop:affine.indep}
  A Riemannian simplex $\riemsplxs \subset \man$ is non-degenerate if and only if $\splxs(x) \subset T_x\man$ is non-degenerate for every $x \in \riemsplxs$. 
\end{prop}

% -*- LaTeX -*-
% non_degen.tex
% 20131119
%

\section{Non-degeneracy criteria}
\label{sec:non.degen}

In this section we exploit \Propref{prop:affine.indep} to establish geometric criteria that ensure that a Riemannian simplex is non-degenerate. In \Secref{sec:implicit.fct.affine} we worked in an arbitrary coordinate chart $\phi: \man \supset W \to \ren$, where the convex ball $\bmapball$ containing $\riemsplxs$ is contained in $W$. Now we will choose $\phi$ to be the inverse of the exponential map at some point $p \in \bmapball$. Specifically, we set $\phi = u \circ \exp_p^{-1}: W \to \rem$, where $u: T_p\man \to \ren$ is an arbitrary linear isometry. The coordinate function $u$ serves to represent a generic point in $U = \phi(W)$.  The Euclidean simplex $\tilde{\splxs}$ in the coordinate domain can now be identified with $\splxs(p)$, and we observe that $\exp_x \circ \exp_p^{-1}$ maps $\splxs(p)$ to $\splxs(x)$.
% \begin{equation}
%   \label{eq:exp.trans.fct}
%   \exp_x \circ \exp_p^{-1}: \splxs(p) \mapsto \splxs(x).
% \end{equation}

In \Secref{sec:eucl.splxs} we review some properties of Euclidean simplices, including \defn{thickness}, the quality measure that we use, and recall a lemma that bounds the difference in thickness between two simplices whose corresponding edge lengths are almost the same. Thus given an assumed thickness of $\splxs(p)$, the question of whether or not $\riemsplxs$ is degenerate becomes a question of how much the exponential transition function~\eqref{eq:exp.transition} distorts distances. In order to address this question, we exploit the Rauch comparison theorem, which we discuss in \Secref{sec:rauch.thm}. We put these observations together in \Secref{sec:non.degen.riem} to obtain explicit bounds on the required quality of $\splxs(p)$, relative to its size (longest edge length) and the sectional curvatures in $\bmapball$.

\subsection{The stability of Euclidean simplex quality}
\label{sec:eucl.splxs}

A Euclidean simplex $\splxs = \asimplex{v_0, \ldots, v_k} \subset \R^n$ has a number of geometric attributes.  An $i$-face of $\splxs$ is a subset of $i+1$ vertices, and a $(k-1)$ face of a $k$-simplex is a \defn{facet}. The facet of $\splxs$ that does not have $v_i$ as a vertex is denoted $\opface{v_i}{\splxs}$.  The \defn{altitude} of $v_i \in \splxs$ is the distance from $v_i$ to the affine hull of $\opface{v_i}{\splxs}$, denoted $\splxalt{v_i}{\splxs}$, and the longest edge length is denoted $\longedge{\splxs}$. When there is no risk of confusion, we will simply write $\glongedge$, and $\gsplxalt_i$.

The simplex quality measure that we will use is the \defn{thickness} of a $k$-simplex $\splxs$, defined as
\begin{equation}
  \label{eq:defn.thickness}
  \thickness{\splxs} =
  \begin{cases}
    1& \text{if $k=0$} \\
    \min_{v \in \splxs} \dfrac{\gsplxalt_{v}}{k
      \glongedge}& \text{otherwise.}
  \end{cases}
\end{equation}
If $\thickness{\splxs} = 0$, then $\splxs$ is \defn{degenerate}.  We say that $\splxs$ is $\thickbnd$-thick, if $\thickness{\splxs} \geq \thickbnd$. If $\splxs$ is $\thickbnd$-thick, then so are all of its faces. We write $\gthickness$ for the thickness if the simplex in question is clear. 

As discussed in \Secref{sec:implicit.fct.affine}, we can associate a matrix $P$ to a Euclidean simplex.  The quality of a simplex $\splxs$ is closely related to the quality of $P$, which can be quantified by means of its
 \defn{singular values}.  In fact, we are only interested in the smallest and largest singular values. The smallest singular value, $\sing{k}{P} = \inf_{\norm{x}=1}\norm{Px}$, vanishes if and only if the matrix $P$ does not have full rank. The largest singular value is the same as the operator norm of $P$, i.e., $\sing{1}{P} = \onorm{P} = \sup_{\norm{x}=1}\norm{Px}$. We have the following result \cite[Lem. 2.4]{boissonnat2013stab1} relating the thickness of $\splxs$ to the smallest singular value of $P$:
\begin{lem}[Thickness and singular value]
  \label{lem:bound.skP}
  Let $\splxs = \asimplex{v_0, \ldots, v_k}$ be a non-de\-generate
  $k$-simplex in $\ren$, with $k>0$, and let $P$ be the $n \times k$
  matrix whose $i^\text{th}$ column is $v_i - v_0$. Then the
  $i^{\text{th}}$ row of the pseudo-inverse $\pseudoinv{P} =
  (\transp{P}P)^{-1}\transp{P}$ is given by $\transp{w}_i$, where
  $w_i$ is orthogonal to $\affhull{\opface{v_i}{\splxs}}$, and
  \begin{equation*}
    \norm{w_i} = \gsplxalt_{i}^{-1}.
  \end{equation*}
  We have the following bound on the smallest singular value of $P$:
  \begin{equation*}
    % \label{eq:bound.skP}
    \sing{k}{P} \geq \sqrt{k} \gthickness\glongedge.
  \end{equation*}
\end{lem}
The appearance of the dimension $k$ in the denominator in the definition of thickness is a convention introduced so that $\gthickness$ provides a clean bound on the condition number of $P$: Since the columns of $P$ have norm bounded by $\glongedge$, we have that $\sing{1}{P} \leq \sqrt{k}\glongedge$, and thus \Lemref{lem:bound.skP} implies $\frac{\sing{1}{P}}{\sing{k}{P}} \leq \gthickness^{-1}$. Although we adhere to definition~\eqref{eq:defn.thickness} in this work, we acknowledge that this normalisation convention may obscure the relationship between simplex quality and dimension. We frequently make use of the fact that for a $k$-simplex $\splxs$, we have $k\thickness{\splxs} \leq 1$. 

The crucial property of thickness for our purposes is its stability. If two Euclidean simplicies with corresponding vertices have edge lengths that are almost the same, then their thicknesses will be almost the same. This allows us to quantify a bound on the smallest singular value of the matrix associated with one of the simplices, given a bound on the other. To be precise, we have the following consequence of the more general \Lemref{lem:abstract.eucl.splx} demonstrated in \Secref{sec:abstract.eucl.splx}:
%, as shown in the following Lemma~\cite[Lem. A.3]{boissonnat2013manmesh.inria}:
\begin{lem}[Thickness under distortion]
  \label{lem:intrinsic.thick.distortion}
  Suppose that $\splxs = \asimplex{v_0, \ldots, v_k}$ and $\tsplxs =
  \asimplex{\tilde{v}_0, \ldots, \tilde{v}_k}$ are two $k$-simplices in
  $\ren$ such that
  \begin{equation*}
    \abs{\norm{v_i - v_j} - \norm{\tilde{v}_i - \tilde{v}_j}} \leq
    \metlipconst \longedge{\splxs}
  \end{equation*}
  for all $0 \leq i < j \leq k$.  
  Let $P$ be the matrix whose
  $i^{\text{th}}$ column is $v_i - v_0$, and define $\tilde{P}$
  similarly.
 
  If
  \begin{equation*}
    \metlipconst =  \frac{\eta \thickness{\splxs}^2}{4}
   \qquad \text{ with } \quad 0 \leq \eta \leq 1,
  \end{equation*}
  then
  \begin{equation*}
    \sing{k}{\tilde{P}} \geq (1 - \eta) \sing{k}{P}.
  \end{equation*}
  % and
  % \begin{equation*}
  %   \thickness{\tsplxs}\longedge{\tsplxs} \geq \frac{1}{\sqrt{k}}(1 - \eta^2)
  %   \thickness{\splxs}\longedge{\splxs},
  % \end{equation*}
  and
  \begin{equation*}
    \thickness{\tsplxs} \geq \frac{4}{5\sqrt{k}}(1 - \eta)
    \thickness{\splxs}. 
  \end{equation*}
\end{lem}

\subsection{The Rauch Comparison Theorem}
\label{sec:rauch.thm}

The Rauch comparison theorem gives us bounds on the norm of the differential of the exponential map. This in turn implies a bound on how much the exponential map can distort distances. It is called a comparison theorem because it is implicitly comparing the exponential map on the given manifold to that on a space of constant sectional curvatures. In this context we encounter the functions
\begin{equation*}
  \bfS_{\cnstcurv}(r) = 
  \begin{cases}
    (1/ \sqrt{\kappa})\sin \sqrt{\kappa}r &\qquad \cnstcurv>0\\
    r &\qquad \cnstcurv=0\\
    (1/ \sqrt{-\kappa})\sinh \sqrt{-\kappa}r &\qquad \cnstcurv<0,
  \end{cases}
\end{equation*}
parameterised by $\cnstcurv$, which can be thought of as representing a constant sectional curvature.

The Rauch theorem can be found in Buser and Karcher \cite[\S 6.4]{buser1981} or in Chavel~\cite[Thm. IX.2.3]{Chavel}, for example. In the statement of the theorem we implicitly use the identification between the tangent spaces of a tangent space and the tangent space itself.
\begin{lem}[Rauch theorem]
  Radially the exponential map $\exp_p: \tanspace{p}{\man} \to \man$ is an isometry:
  \begin{equation*}
    \norm{(d\exp_p)_{v} v} = \norm{v}.
  \end{equation*}

  Assume the sectional curvatures, $\seccurv$, are bounded by $\curvlowbnd \leq \seccurv \leq \curvupbnd$.  Taking $\norm{v}=1$, one has for any $w$ perpendicular to $v$
  \begin{equation*}
    \frac{\bfS_{\curvupbnd}(r)}{r} \norm{w} \leq \norm{(d\exp_p)_{rv}w} \leq 
    \frac{\bfS_{\curvlowbnd}(r)}{r} \norm{w}.
  \end{equation*}
  The inequalities hold when $r < 2\barymapradbnd$ (defined in \Eqnref{eq:good.barymap.rad}). Also, if $\curvlowbnd < 0$, then the right inequality is valid for all $r$, and if $\curvupbnd > 0$, then the left inequality is valid for all $r$. 
\end{lem}

For convenience, we will use a bound on the absolute value of the sectional curvatures, rather than separate upper and lower bounds. Thus $\abs{\seccurv} \leq \curvabsbnd$, where $\curvabsbnd = \max\{ \curvupbnd, - \curvlowbnd \}$. We use
%the Lagrange form of the remainder in 
Taylor's theorem to obtain
\begin{align*}
  \bfS_{-\curvabsbnd}(r) &\leq r +  \frac{\curvabsbnd r^3}{2} \qquad \text{when } 0 \leq r < \frac{\pi}{2 \sqrt{\curvabsbnd}}\\
  \bfS_{\curvabsbnd}(r) &\geq r - \frac{\curvabsbnd  r^3}{6} \qquad \text{for all } r \geq 0.
\end{align*}
We can restate the Rauch theorem in a weaker, but more convenient form:
\begin{lem}
  \label{lem:poly.rauch}
  Suppose the sectional curvatures in $\man$ are bounded by $\abs{\seccurv} \leq \curvabsbnd$. If $v \in T_p\man$ satisfies $\norm{v} = r < \frac{\pi}{2\sqrt{\curvabsbnd}}$, then for any vector $w \in T_v(T_pM) \cong T_pM$, we have
  \begin{equation*}
    \left(1 - \frac{\curvabsbnd  r^2}{6} \right)\norm{w}
    \leq \norm{(d\exp_p)_{v}w} \leq 
    \left(1 + \frac{\curvabsbnd  r^2}{2} \right)\norm{w}.
  \end{equation*}
\end{lem}

\subsection{Non-degenerate Riemannian simplices}
\label{sec:non.degen.riem}

Our goal now is to estimate the metric distortion incurred when we map a simplex from one tangent space to another via the exponential maps
\begin{equation*}
  \exp_x^{-1} \circ \exp_{p}: T_{p}\man \to T_x\man,
\end{equation*}
and this is accomplished by the bounds on the differential. Specifically, if $F:\R^n \to \R^n$ satisfies $\onorm{dF} \leq \eta$, then the length of the image of the line segment between $x$ and $y$ provides an upper bound on the distance between $F(x)$ and $F(y)$:
\begin{equation}
  \label{eq:up.bnd.length}
  \norm{F(y) - F(x)} \leq \int_0^1 \norm{dF_{x +s(y-x)}(y-x)} \, ds \leq \eta \norm{y-x}.
\end{equation}

If $x,p, y \in \bmapball$, with $y = \exp_{p}(v)$, then $\norm{v} < 2\rho$, and $\norm{\exp^{-1}_x(y)} < 2\rho$. Then, if $\rho < \barymapradbnd$ given in \Eqnref{eq:good.barymap.rad}, \Lemref{lem:poly.rauch} tells us that
\begin{equation*}
  \onorm{d\left(\exp_x^{-1} \circ \exp_{p} \right)_{v}}
  \leq   \onorm{\left(d\exp_x^{-1}\right)_y}
  \onorm{ \left(d\exp_{p} \right)_{v}}
  \leq \left(1 + \frac{\curvabsbnd(2\rho)^2}{3} \right)
   \left(1 + \frac{\curvabsbnd(2\rho)^2}{2} \right)
  \leq 1 + 5\curvabsbnd \rho^2.
\end{equation*}
Therefore \eqref{eq:up.bnd.length} yields
\begin{equation*}
  \norm{v_i(x) - v_j(x)} \leq (1 + 5\curvabsbnd\rho^2) \norm{v_i(p) - v_j(p)}.
\end{equation*}
We can do the same argument the other way, so
\begin{equation*}
  \norm{v_i(p) - v_j(p)} \leq (1 + 5\curvabsbnd\rho^2) \norm{v_i(x) - v_j(x)},
\end{equation*}
and we find
\begin{equation}
  \begin{split}
    \label{eq:bnd.edge.distort}
    \big| \norm{v_i(x) - v_j(x)} - \norm{v_i(p) - v_j(p)} \big|
    &\leq 5\curvabsbnd\rho^2(1 + 5\curvabsbnd\rho^2) \norm{v_i(p) - v_j(p)}\\
    &\leq 21\curvabsbnd\rho^2 \norm{v_i(p) - v_j(p)} \qquad \text{when } \rho < \barymapradbnd.
  \end{split}
\end{equation}

Letting $P$ be the matrix associated with $\splxs(p)$, and using $\metlipconst = 21\curvabsbnd\rho^2$, in \Lemref{lem:intrinsic.thick.distortion}, we find that the matrix $\tilde{P}$ associated with $\splxs(x)$ in \Propref{prop:affine.indep} is non-degenerate if $\splxs(p)$ satisfies a thickness bound of $\thickbnd > 10 \sqrt{\curvabsbnd}\rho$, and we have
\begin{mainthm}
  \label{thm:thick.nondegen.riem}
  Suppose $\man$ is a Riemannian manifold with sectional curvatures $\seccurv$ bounded by $\abs{\seccurv} \leq \curvabsbnd$, and $\riemsplxs$ is a Riemannian simplex, with $\riemsplxs \subset \bmapball \subset \man$, where $B_{\rho}$ is an open geodesic ball of radius $\rho$ with
  \begin{equation*}
    \rho < \barymapradbnd = \min \left \{ \frac{\injradM}{2},
      \frac{\pi}{4\sqrt{\curvabsbnd}} \right \}.
  \end{equation*}
  Then $\riemsplxs$ is non-degenerate if there is a point $p \in \bmapball$ such that the lifted Euclidean simplex $\splxs(p)$ has thickness satisfying
 \begin{equation*}
   \thickness{\splxs(p)} > 10 \sqrt{\curvabsbnd}\rho.
 \end{equation*}
The ball $\bmapball$ may be chosen so that this inequality is necessarily satisfied if
 \begin{equation*}
   \thickness{\splxs(p)} > 10 \sqrt{\curvabsbnd}\longedge{\riemsplxs},
 \end{equation*}
 where $\longedge{\riemsplxs}$ is the geodesic length of the longest edge in $\riemsplxs$.
\end{mainthm}

% \begin{restatable}[Non-degeneracy criteria]{mainthm}{nondegenthm}
%   \label{thm:thick.nondegen.riem}%
% Suppose $\man$ is a Riemannian manifold with sectional curvatures bounded by $\abs{\seccurv} \leq \curvabsbnd$, and $\riemsplxs$ is a Riemannian simplex, with $\riemsplxs \subset \bmapball \subset \man$, where $B_{\rho}$ is an open geodesic ball of radius $\rho$ with
%   \begin{equation}
%     \label{eq:good.barymap.rad}
%     \rho < \barymapradbnd = \min \left \{ \frac{\injradM}{2},
%       \frac{\pi}{4\sqrt{\curvabsbnd}} \right \}.
%   \end{equation}
%   Then $\riemsplxs$ is non-degenerate if there is a point $p \in \bmapball$ such that the lifted Euclidean simplex $\splxs(p)$ has thickness satisfying
%  \begin{equation*}
%    \thickness{\splxs(p)} > 10 \sqrt{\curvabsbnd}\rho.
%  \end{equation*}
% The ball $\bmapball$ may be chosen so that this inequality is necessarily satisfied if
% \begin{equation}
%   \label{eq:thickness.vs.length}
%    \thickness{\splxs(p)} > 10 \sqrt{\curvabsbnd}\longedge{\riemsplxs},
%  \end{equation}
%  where $\longedge{\riemsplxs}$ is the geodesic length of the longest edge in $\riemsplxs$.
% \end{restatable}

% \nondegenthm

The last assertion follows from the remark at the end of \Secref{sec:bary.map}: If $\longedge{\riemsplxs} < \rho_0$, then $\riemsplxs$ is contained in a closed ball of radius $\longedge{\riemsplxs}$ centred at one of the vertices.

\begin{remark}
  Using \Propref{prop:exp.trans.bnd} and \Lemref{lem:metric.distort.convex} of \Secref{sec:triangulation}, we can replace \Eqnref{eq:bnd.edge.distort} with
  \begin{equation*}
    \big| \norm{v_i(x) - v_j(x)} - \norm{v_i(p) - v_j(p)} \big|
    \leq 6\curvabsbnd\rho^2 \norm{v_i(p) - v_j(p)} \qquad \text{when } \rho < \frac{1}{2}\barymapradbnd,
  \end{equation*}
  and we find, that the Riemannian simplex $\riemsplxs$ of \Thmref{thm:thick.nondegen.riem} is non-degenerate if
 \begin{equation*}
   \thickness{\splxs(p)} > 5 \sqrt{\curvabsbnd}\rho,
 \end{equation*}
 but with the caveat that $\rho$ must now satisfy $\rho \leq \frac{1}{2}\barymapradbnd$.
\end{remark}

\paragraph{Orientation}

In Euclidean space $\E^n$ we can define an orientation as an equivalence class of frames, two frames being equivalent if the linear transformation between them has positive determinant. We can likewise associate an orientation to a (non-degenerate) Euclidean $n$-simplex $\splxs = \asimplex{v_0, \ldots, v_n}$: it is the orientation associated with the basis $\{(v_i - v_0)\}_{i \in \{1,\ldots,n\}}$. The orientation depends on how we have indexed the points. Any even permutation of the indices yields the same orientation.

In a manifold, we can assign an orientation locally, in a neighbourhood $U \subset \man$ on which the tangent bundle admits a local trivialisation, for example. Then we can define an orientation by defining an orientation on $\tanspace{p}{M}$ for some $p \in U$. If $\splxs = \asimplex{p_0, \ldots, p_n} \subset U$ defines a non-degenerate Riemannian simplex, then we can associate an orientation to that simplex: it is the orientation of $\splxs(p_0) \subset \tanspace{p_0}{\man}$. Again, we will get agreement on the orientation if we perform any even permutation of the vertex indices. The reason is that our non-degeneracy assumption implies that the orientation of $\splxs(p_i)$ will agree with the orientation of $\splxs(p_j)$ for any $i,j \in \{0, \ldots, n\}$.

In the particular case discussed in this section, where $\phi = u \circ \exp_{p}^{-1}$, for $p \in \bmapball$, the barycentric map $\bmap: \splxsE(p) \to \riemsplxs$ is orientation preserving. Since $\exp_x^{-1}\circ \exp_y$ is orientation preserving for any $x,y \in \bmapball$, it is enough to consider the case where $p=p_0 \in \splxs$. Consider \Eqnref{eq:diff.bary.map}:
\begin{equation*}
  (db)_{v_0(p_0)} = - \left(\conn_x\nu \right)^{-1} \partial_u \nu_{v_0(p_0)}.
\end{equation*}
By \Eqnref{eq:matrix.of.deriv}, we have $(\partial_u \nu)_{v_0(p_0)} = - \Id$. Also, it follows from \Lemref{lem:enfct.convex} that $\conn^{\man}\nu$ has positive determinant.  (Buser and Karcher~\cite[p.132]{buser1981} show that $\conn^{\man}\nu$ is bounded near the identity), and thus so must $db$ everywhere, since it does not vanish on its domain.

% -*- LaTeX -*-
% triangulation.tex
% 20140610, fr. pwsmooth.tex (20140220)
%

\section{Triangulation criteria}
\label{sec:triangulation}

We are interested in the following scenario. Suppose we have a finite set of points $\pts \subset \man$ in a compact Riemannian manifold, and an (abstract) simplicial complex $\acplx$ whose vertex set is $\pts$, and such that every simplex in $\acplx$ defines a non-degenerate Riemannian simplex. When can we be sure that $\acplx$ triangulates $\man$? Consider a convex ball $\bmapball$ centred at $p \in \pts$. We require that, when lifted to $T_p\man$, the simplices near $p$ triangulate a neighbourhood of the origin. If we require that the simplices be small relative to $\bmaprad$, and triangulate a region extending to near the boundary of the lifted ball, then Riemannian simplices outside of $\bmapball$ cannot have points in common with the simplices near the centre of the ball, and it is relatively easy to establish a triangulation. 

Instead, we aim for finer local control of the geometry. We establish geometric conditions (\Lemref{lem:embed.star}) that ensure that the complex consisting of simplices incident to $p$, (i.e., the star of $p$) is embedded by a given map into the manifold. In order to achieve this result we require a strong constraint on the differential of the map in question. Since we work locally, in a coordinate chart, we consider maps $F: \R^n \supseteq U \to \R^n$. We demand that for some linear isometry $T:\R^n \to \R^n$ we have
\begin{equation}
  \label{eq:strong.diff.bnd}
  \onorm{dF_u - T} \leq \eta,
\end{equation}
for some $0 \leq \eta \leq 1$, and all $u \in U$. This is stronger than the kind of bounds found, for example, in the Rauch theorem (\Lemref{lem:poly.rauch}), which have the form
\begin{equation}
  \label{eq:bound.sing.vals}
  (1-\eta)\norm{w} \leq \norm{dF_uw} \leq (1 + \eta)\norm{w}.
\end{equation}
Whereas \eqref{eq:bound.sing.vals} implies that $dF_u$ is close to a linear isometry at every $u\in U$, \Eqnref{eq:strong.diff.bnd} means that $dF_u$ is close to \emph{the same} linear isometry for all $u \in U$. 

Using a local constraint of the form~\eqref{eq:strong.diff.bnd} to establish the embedding of vertex stars, we demonstrate, in \Secref{sec:generic.tri}, generic criteria which ensure that a map from a simplicial complex to a Riemannian manifold is a homeomorphism. We then turn our attention to the specific case where the map in question is the barycentric coordinate map on each simplex. 

Using a refinement of the Rauch theorem established by Buser and Karcher~\cite{buser1981} we show, in \Secref{sec:diff.exp.trans}, that transition functions arising from the exponential map (\Eqnref{eq:exp.transition}) are subject to bounds of the form~\eqref{eq:strong.diff.bnd}.  Then in \Secref{sec:triang.riem.splxs} we observe that the barycentric coordinate map is also subjected to such bounds, and thus yield \Thmref{thm:triangulation} as a particular case of the generic triangulation criteria.

\subsection{Generic triangulation criteria}
\label{sec:generic.tri}

%\subsection{Maps with bounded differentials}
\label{sec:smooth}% definition of smooth maps from simplices and cplxs

We say that a map $F: \rem \to \ren$ is \defn{smooth} if it is of class $\cinfty$.  If $A \subset \rem$, then $F: A \to \ren$ is \defn{smooth} on $A$ if $F$ can be extended to a function that is smooth in an open neighbourhood of $A$. (I.e., there exists an open neighbourhood $U \in \rem$ and a smooth map $\tilde{F}: U \to \ren$ such that $A \subseteq U$, and $\tilde{F}|_A = F$.)  This definition is independent of the ambient space $\rem$ containing $A$. In particular, if $A \subseteq \R^k \subseteq \rem$, then the smoothness of $F$ does not depend on whether we consider $A$ to be a subset of $\R^k$ or of $\rem$. In the case that $A$ is the closure of a non-empty open set, continuity of the partial derivatives implies that they are well defined on all of $A$ and independent of the chosen extension. See Munkres~\cite[\S 1]{munkres1968} for more details. 

For our purposes, we are interested in smooth maps from non-degenerate closed Euclidean simplices of dimension $n$ into an $n$-dimensional manifold $\man$. We will work within coordinate charts, so our primary focus will be on maps of the form
\begin{equation*}
  F: \splxsE^n \to \R^n,
\end{equation*}
such that \Eqnref{eq:strong.diff.bnd} is satisfied for all $u \in \splxsE^n$. As an example of how we can exploit this bound, we observe that a map satisfying \Eqnref{eq:strong.diff.bnd} is necessarily an embedding with bounded metric distortion if its domain is convex:
\begin{lem}
  \label{lem:metric.distort.convex}
  Suppose $A \subset \ren$ is convex, and
  $F: A \to \ren$ is a smooth map such that, for some non-negative $\eta <1$,
  \begin{equation*}
    \onorm{dF_u - T} \leq \eta,
  \end{equation*}
  for all $u \in A$, and some linear isometry $T:\ren \to \ren$.
  Then 
  \begin{equation*}
    \abs{\,  \norm{F(u) - F(v)} - \norm{u - v} \,} \leq \eta \norm{u - v}
    \quad \text{for all } u,v \in A.
  \end{equation*}
\end{lem}
\begin{proof}
  We observe that it is sufficient to consider the case $T = \Id$, because if $\tilde{F} = T^{-1} \circ F$, then $\onorm{d\tilde{F} - \Id} = \onorm{dF - T}$, and $\norm{\tilde{F}(u) - \tilde{F}(v)} = \norm{F(u) - F(v)}$.

  Assume $u\neq v$. For the lower bound we consider the unit vector $\hat{u} = \frac{u-v}{\norm{u-v}}$, and observe that
  \begin{equation*}
    \dotprod{dF(u-v)}{\hat{u}} \geq (1-\eta)\norm{u-v} > 0,
  \end{equation*}
  so, by integrating along the segment $\seg{u}{v}$, we find
  \begin{equation*}
    \norm{F(u)- F(v)} \geq \dotprod{\big(F(u) - F(v) \big)}{\hat{u}} \geq (1-\eta) \norm{u-v}.
  \end{equation*}
  For the upper bound we employ the unit vector $\hat{w} = \frac{F(u) - F(v)}{\norm{F(u) - F(v)}}$:
  \begin{equation*}
    \dotprod{ \big( F(u)-F(v) \big)}{\hat{w}} = \norm{F(u) - F(v)} \leq (1 + \eta) \dotprod{(u-v)}{\hat{w}} \leq (1 + \eta) \norm{u-v}.
  \end{equation*}
\end{proof}

For our purposes we will be free to choose a coordinate system so that $F$ keeps a vertex fixed. We will have use for the following observation, which can be demonstrated with an argument similar to the proof of \Lemref{lem:metric.distort.convex}:
\begin{lem}
  \label{lem:bound.disp}
  If $A \subseteq \ren$ is a convex set and $F: A \to \ren$ is a smooth map with a fixed point $p \in A$ and 
  \begin{equation*}
    \onorm{dF_u - \Id} \leq \eta, \qquad \text{for all } u \in A,
  \end{equation*}
  then
  \begin{equation*}
    \norm{F(u) - u } \leq \eta{\norm{u-p}} \qquad \text{for all } u \in A.
  \end{equation*}
%  \noproof
\end{lem}

\paragraph{Embedding complexes}
\label{sec:embed.cplx}

In preparation for considering triangulations we first consider the problem of mappings of complexes into $\ren$.

A simplicial complex $\acplx$ is a set of abstract simplices such that if $\splxs \in \acplx$, then $\splxt \in \acplx$ for every face $\splxt \subset \splxs$. We will only consider finite simplicial complexes. A subcomplex of $\acplx$ is a subset that is also a simplicial complex. The \defn{star} of a simplex $\splxs \in \acplx$ is the smallest subcomplex of $\acplx$ consisting of all simplices that have $\splxs$ as a face, and is denoted $\str{\splxs}$.
% subcomplex $\ccplx \subset \acplx$ is the smallest subcomplex of $\acplx$ that contains all the simplices that share a face with a simplex of $\ccplx$; the definition is different from a common definition in algebraic topology. The star of $\ccplx$ is denoted $\str{\ccplx}$. 
In particular, if $p$ is a vertex of $\acplx$, then $\str{p}$ is the set of simplices that contain $p$, together with the faces of these simplices.

The \defn{carrier} (``geometric realisation'') of $\acplx$ is denoted $\carrier{\acplx}$.  We are interested in complexes endowed with a \defn{piecewise flat metric}. This is a metric on $\carrier{\acplx}$ that can be realised by assigning lengths to the edges in $\acplx$ such that each simplex $\splxs \in \acplx$ is associated with a Euclidean simplex $\splxsE \subset \carrier{\acplx}$ that has the prescribed edge lengths. Certain constraints on the edge lengths must be met in order to define a valid piecewise flat metric, but for our current purposes we will have a metric inherited from an embedding in Euclidean space.

We say that $\acplx$ is \defn{embedded} in $\ren$ if the vertices lie in $\ren$ and the convex hulls of the simplices in $\acplx$ define a geometric simplicial complex. In other words, to each $\splxs,\splxt \in \acplx$ we associate $\splxsE = \convhull{\splxs}$, $\splxtE = \convhull{\splxt}$, and we have $\splxsE \cap \splxtE =  \convhull{\splxs \cap \splxt}$. The topological boundary of a set $B \subset \ren$ is denoted by $\bdry{B}$, and its topological interior by $\intr{B}$. If $\acplx$ is embedded in $\R^n$, and $p$ is a vertex of $\acplx$, we say that $\str{p}$ is a \defn{full star} if $p \in \intr{\carrier{\str{p}}}$.

The \defn{scale} of $\acplx$ is an upper bound on the length of the longest edge in $\acplx$, and is denoted by $\scale$. 
% This is often called the \emph{mesh} of $\acplx$ in the literature.
We say that $\acplx$ is $\thickbnd$-thick if each simplex in $\acplx$ has thickness greater than $\thickbnd$.  The \defn{dimension} of $\acplx$ is the largest dimension of the simplices in $\acplx$. We call a complex of dimension $n$ an \defn{$n$-complex}. If every simplex in $\acplx$ is the face of an $n$-simplex, then $\acplx$ is a \defn{pure $n$-complex}.

A map $F: \carrier{\acplx} \to \ren$ is \defn{smooth on $\acplx$} if for each $\splxs \in \acplx$ the restriction $F \big|_{\splxsE}$ is smooth. This means that $d(F\big|_{\splxsE})$ is well defined, and even though $dF$ is not well defined, we will use this symbol when the particular restriction employed is either evident or unimportant. When the underlying complex on which $F$ is smooth is unimportant, we simply say that $F$ is \defn{piecewise smooth}.

$F$ is \defn{piecewise linear} if its restriction to each simplex is an affine map. The \defn{secant map} of $F$ is the piecewise linear map defined by the restriction of $F$ to the vertices of $\acplx$.

We are interested in conditions that ensure that $F: \carrier{\acplx} \to \ren$ is a topological embedding. Our primary concern is with the behaviour of the boundary. The reason for this is captured by the following variation of a lemma by Whitney~\cite[Lem AII.15a]{whitney1957}:
\begin{lem}[Whitney]
  \label{lem:whitney.cplx.embed}
  Let $\acplx$ be a (finite) simplicial complex embedded in $\ren$ such that $\intr{\carrier{\acplx}}$ is non-empty and connected, and $\bdry{\carrier{\acplx}}$ is a compact $(n-1)$-manifold. Suppose $F: \carrier{\acplx} \to \ren$ is smooth on $\acplx$ and such that $\det d(F\big|_{\splxsE}) > 0$ for each $n$-simplex $\splxs \in \acplx$. If the restriction of $F$ to $\bdry{\carrier{\acplx}}$ is an embedding, then $F$ is a topological embedding.
\end{lem}
\begin{proof}
  The assumptions on $\intr{\carrier{\acplx}}$ and $\bdry{\carrier{\acplx}}$ imply that $\acplx$ is a pure $n$-complex, and that each $(n-1)$ simplex is either a boundary simplex, or the face of exactly two $n$-simplices.  Whitney showed \cite[Lem AII.15a]{whitney1957} that any $x \in \intr{\carrier{\acplx}}$ admits an open neighbourhood $U \subset \intr{\carrier{\acplx}}$ such that the restriction of $F$ to $U$ is a homeomorphism. In particular, $F(\intr{\carrier{\acplx}})$ is open.

  By the Jordan-Brouwer separation theorem~\cite[\S IV.7]{outerelo2009degree}, $\R^n \setminus F(\bdry{\carrier{\acplx}})$ consists of two open components, one of which is bounded. Since $F(\carrier{\acplx})$ is compact, $F(\intr{\carrier{\acplx}})$ must coincide with the bounded component, and in particular $F(\intr{\carrier{\acplx}}) \cap F(\bdry{\carrier{\acplx}}) = \emptyset$, so $F(\intr{\carrier{\acplx}})$ is a single connected component.

  We need to show that $F$ is injective. First we observe that the set of points in $F(\intr{\carrier{\acplx}})$ that have exaclty one point in the preimage is non-empty. It suffices to look in a neigbhourhood of a point $y \in F(\bdry{\carrier{\acplx}})$. Choose $y = F(x)$, where $x$ is in the relative interior of $\splxsE^{n-1} \subset \bdry{\carrier{\acplx}}$. Then there is a neighbourhood $V$ of $y$ such that $V$ does not intersect the image of any other simplex of dimension less than or equal to $n-1$. Let $\splxs^n$ be the unique $n$-simplex that has $\splxs^{n-1}$ as a face. Then $F^{-1}(V \cap F(\carrier{\acplx}) \subset \splxsE^n$, and it follows that every point in $V \cap \intr{\carrier{\acplx}}$ has a unique point in its image.

  Now the injectivity of $F$ follows from the fact that the number of points in the preimage is locally constant on $F(\intr{\carrier{\acplx}}) \setminus F(\bdry{\carrier{\acplx}})$, which in our case is connected. This is a standard argument in degree theory~\cite[Prop. IV.1.2]{outerelo2009degree}: A point $a \in F(\intr{\carrier{\acplx}})$ has $k$ points $\{x_1, \ldots, x_k\}$ in its preimage. There is a neighbourood $V$ of $a$ and disjoint neighbourhoods $U_i$ of $x_i$ such that $F|_{U_i}: U_i \to V$ is a homeomorphism for each $i \in \{1,\ldots,k\}$. It follows that the number of points in the preimage is $k$ for every point in the open neighbourhood of $a$ defined as
  \begin{equation*}
    W = V \setminus  F(\carrier{\acplx} \setminus \bigcup_i U_i).
  \end{equation*}
\end{proof}

\begin{lem}[Embedding a star]
  \label{lem:embed.star}
  Suppose $\acplx = \str{p}$ is a $\thickbnd$-thick, pure $n$-complex embedded in $\ren$ such that all of the $n$-simplices are incident to a single vertex, $p$, and $p \in \intr{\carrier{\acplx}}$ (i.e., $\str{p}$ is a full star). If $F: \carrier{\acplx} \to \ren$ is smooth on $\acplx$, 
  and satisfies 
  \begin{equation*}
    \onorm{dF - \Id}  < n\thickbnd
  \end{equation*}
  on each $n$-simplex of $\acplx$, then $F$ is an embedding.
\end{lem}
\begin{proof}
  If $\carrier{\acplx}$ is convex, then the claim follows immediately from \Lemref{lem:metric.distort.convex}.

  From the definition of thickness, we observe that $n\thickbnd \leq 1$, and therefore $\onorm{dF} > 0$.  By \Lemref{lem:whitney.cplx.embed}, it suffices to consider points $x,y \in \bdry{\carrier{\acplx}}$. Rather than integrating the differential of a direction, as we did implicitly in the proof of \Lemref{lem:metric.distort.convex}, we will integrate the differential of an angle.

  Let $Q$ be the $2$-dimensional plane defined by $p$, $x$, and $y$. We define the angle function $\phi: \ren \to \R$ as follows: For $z \in \ren$, let $\check{z}$ be the orthogonal projection of $z$ into $Q$. Then $\phi(z)$ is the angle that $\check{z}-p$ makes with $x-p$, where the orientation is chosen so that $\phi(y) < \pi$. (We can assume that $x$ and $y$ are not colinear with $p$, since in that case $[x,y]$ must be contained in $\carrier{\acplx}$, and the arguments of \Lemref{lem:metric.distort.convex} ensure that we would not have $F(x)=F(y)$.)

  Let $\alpha$ be the piecewise linear curve obtained by projecting the segment $[x,y] \subset Q$ onto $\bdry{\carrier{\acplx}}$ via the radial rays emanating from $p$. Parameterise $\alpha$ by the angle $\phi$, i.e., by the arc between $\frac{x-p}{\norm{x-p}}$ and $\frac{y-p}{\norm{y-p}}$ on the unit circle in $Q$. Then $d\phi(\alpha') = 1$, and we have
  \begin{align*}
    \phi(F(y)) - \phi(F(x)) &= \int_{F(\alpha)}d\phi = \int_0^{\phi(y)} d\phi(dF\alpha'(s)) \, ds.
  \end{align*}
  We will show that $d\phi(dF\alpha'(s)) > 0$; it follows that $\phi(F(y)) > \phi(F(x))$ and hence $F(y) \neq F(x)$.

  We need an observation about thick simplices: Suppose that $p$ is a vertex of a $\thickbnd$-thick $n$-simplex $\splxsE$, and $\splxtE$ is the facet opposite $p$. Then a line $r$ through $p$ and $\splxtE$ makes an angle $\theta$ with $\splxtE$ that is bounded by
  \begin{equation*}
    \sin \theta \geq n\thickbnd.
  \end{equation*}
  Indeed, the altitude of $p$ satisfies $\gsplxalt_{p} \geq n\thickbnd \glongedge$ by the definition of thickness, and the distance between $p$ and the point of intersection of $r$ with $\splxtE$ is less than $\glongedge$.

  If $\alpha(s)$ is the point of intersection of $r$ with $\splxtE$, we observe that $d\phi(\alpha'(s)) = \norm{\alpha'(s)}\sin \theta$, i.e., the magnitude of the component of $\alpha'(s)$ orthogonal to $r$. The angle $\vartheta$ between $dF\alpha'(s)$ and $\alpha'(s)$ satisfies $\sin \vartheta \leq \eta < n\thickbnd$. Therefore
  \begin{align*}
    d\phi(dF\alpha'(s))
    &\geq \norm{dF\alpha'(s)}(\sin \theta - \sin \vartheta)\\
    &\geq \norm{dF\alpha'(s)}(n\thickbnd - \eta)\\
    &> 0.
  \end{align*}
\end{proof}

\paragraph{Triangulations}

A simplicial complex $\acplx$ is a  \defn{manifold simplicial complex} if $\carrier{\acplx}$ is a topological manifold (without boundary). 
% If $\acplx$ is a manifold $n$-complex, it is \defn{piecewise linear} if for each vertex in $p \in \acplx$ there is a piecewise linear embedding (i.e., linear on each simplex) of $\str{p}$ into $\ren$. Not all manifold complexes are piecewise linear (See the discussion by Thurston~\cite[Example 3.2.11]{thurston1997}, for example.)
A \defn{triangulation} of a manifold $\man$ is a homeomorphism $H: \carrier{\acplx} \to \man$, where $\acplx$ is a simplicial complex. 
% If $\acplx$ is piecewise linear, then $H$ is a \defn{piecewise linear triangulation} of $\man$. 
If $\man$ is a differentiable manifold, then $H$ is a \defn{smooth triangulation} if it is smooth on $\acplx$, i.e., the restriction of $H$ to any simplex in $\acplx$ is smooth.  We are concerned with smooth triangulations of compact Riemannian manifolds (without boundary).

Our homeomorphism argument is based on 
% the observation that if $H: \carrier{\acplx} \to \man$ is such that for every vertex in $\acplx$, the restriction of $H$ to $\carrier{\str{p}}$ is an embedding, then $H$ is a covering map. So if it is injective, it is a triangulation.
the following observation:
\begin{lem}
  Let $\acplx$ be a manifold simplicial complex of dimension $n$ with finite vertex set $\pts$, and let $\man$ be a compact $n$-manifold.  Suppose $H: \carrier{\acplx} \to \man$ is such that for each $p \in \pts$, $H \big|_{\carrier{\str{p}}}: \carrier{\str{p}} \to \man$ is an embedding. If for each connected component $\man_i$ of $\man$ there is a point $y \in \man_i$ such that $h^{-1}(y)$ contains exactly one point in $\carrier{\acplx}$, then $H$ is a homeomorphism.
\end{lem}
\begin{proof}
  The requirement that the star of each vertex be embedded means that $H$ is locally a homeomorphism, so it suffices to observe that it is bijective. It is surjective by Brouwer's invariance of domain; thus $H$ is a covering map. The requirement that each component of $\man$ has a point with a single point in its pre-image implies that $H: \carrier{\acplx} \to \man$ is a single-sheeted covering, and therefore a homeomorphism.
\end{proof}

The following proposition generically models the situation we will work with when we describe a triangulation by Riemannian simplices:
\begin{prop}[Triangulation]
  \label{prop:generic.triangulation}
  Let $\acplx$ be a  manifold simplicial $n$-complex with finite vertex set $\pts$, and $\man$ a compact Riemannian manifold with an atlas $\{(W_p,\phi_p)\}_{p \in \pts}$ indexed by $\pts$. Suppose 
  \begin{equation*}
    H:\carrier{\acplx} \to \man
  \end{equation*}
  satisfies:
  \begin{enumerate}
  \item For each $p \in \pts$ the secant map of $\phi_p \circ H$ restricted to $\carrier{\str{p}}$ is a piecewise linear embedding $\lmap_p:\carrier{\str{p}} \to \ren$ such that each simplex $\splxs \in \ccplx_p = \lmap_p(\str{p})$ is $\thickbnd$-thick, and $\carrier{\ccplx_p} \subset \ballEn{\lmap_p(p)}{h}$, with $\lmap_p(p) \in \intr{\carrier{\ccplx_p}}$.  The scale parameter $h$ must satisfy $h < \frac{\injradM}{4}$, where $\injradM$ is the injectivity radius of $\man$.
  \item For each $p \in \pts$, $\phi_p: W_p \stackrel{\cong}{\longrightarrow} U_p \subset \ren$ is such that $\close{B} = \cballEn{\lmap_p(p)}{\frac{3}{2}h} \subseteq U_p$, and $\onorm{(d\phi_p^{-1})_u} \leq \frac{4}{3}$, for every $u \in \close{B}$. 
  \item
    The map 
    \begin{equation*}
      F_p = \phi_p \circ H \circ \lmap_p^{-1}: \carrier{\ccplx_p} \to \ren
    \end{equation*}
    satisfies
    \begin{equation*}
      \onorm{(dF_p)_u-\Id} \leq \frac{n\thickbnd}{2}
    \end{equation*}
    on each $n$-simplex  $\splxs \in \ccplx_p$, and every $u \in \splxsE$.
  \end{enumerate}
  Then $H$ is a smooth triangulation of $\man$.
\end{prop}
\begin{proof}
  By \Lemref{lem:embed.star}, $F_p$ is a homeomorphism onto its image. It follows then that $H\big|_{\carrier{\str{p}}}$ is an embedding for every $p\in \pts$. Therefore, since $\carrier{\acplx}$ is compact, $H: \carrier{\acplx} \to \man$ is a covering map.

  Given $x \in \carrier{\acplx}$, with $x \in \splxsE$, and $p$ a vertex of $\splxsE$, let $\tilde{x} = \lmap_p(x) \in \carrier{\ccplx_p}$. Then the bound on $dF$ implies that $\norm{F_p(\tilde{x}) - \lmap_p(p)} \leq \left(1 + \frac{n\thickbnd}{2} \right)h \leq \frac{3}{2}h$, so $F_p(\tilde{x}) \in \close{B}$. Since $\phi_p^{-1} \circ F_p(\tilde{x}) = H(x)$, and
  \begin{align*}
    \norm{(d\phi_p^{-1})_{F(u)}(dF_p)_u} \leq \frac{4}{3} \left( 1 + \frac{n\thickbnd}{2} \right) \leq 2
  \end{align*}
  for any $u \in \splxsE \subset \carrier{\ccplx_p}$, we have that $\distM{H(p)}{H(x)} \leq 2h$.

  Suppose $y \in \carrier{\acplx}$ with $H(y) = H(x)$. Let $\splxt \in \acplx$ with $y \in \splxtE$, and $q \in \splxt$ a vertex. Then $\distM{H(p)}{H(q)} \leq 4h < \injradM$. Thus there is a path $\gamma$ from $H(x)$ to $H(p)$ to $H(q)$ to $H(y) = H(x)$ that is contained in the topological ball $\ballM{H(p)}{\injradM}$, and is therefore null-homotopic. Since $H$ is a covering map, this implies that $x = y$. Thus $H$ is injective, and therefore defines a smooth triangulation.
\end{proof}

    % -*- LaTeX -*-
% bary_tri.tex
% 20140610, fr. pwsmooth.tex (20140220)
%

\subsection{The differential of exponential transitions}
\label{sec:diff.exp.trans}

If there is a unique minimising geodesic from $x$ to $y$, we denote the parallel translation along this geodesic by $T_{yx}$.  As a preliminary step towards exploiting \Propref{prop:generic.triangulation} in the context of Riemannian simplices, we show here that the estimates of Buser and Karcher \cite[\S 6]{buser1981} imply
\begin{prop}[Strong exponential transition bound]
  \label{prop:exp.trans.bnd}
  Suppose the sectional curvatures on $\man$ satisfy $\abs{\seccurv} \leq \curvabsbnd$. Let $v \in T_p\man$, with $y = \exp_p(v)$.  If $x,y \in \ballM{p}{\rho}$, with
  \begin{equation*}
    \rho <
    \frac{1}{2}\rho_0 = \frac{1}{2}\min \left\{ \frac{\injradM}{2}, \frac{\pi}{4\sqrt{\curvabsbnd}} \right\},
  \end{equation*}
  then
  \begin{equation*}
    \onorm{d(\exp^{-1}_x \circ \exp_p)_v - T_{xp}} \leq 6\curvabsbnd \rho^2.
  \end{equation*}
\end{prop}

The primary technical result that we use in the demonstration of \Propref{prop:exp.trans.bnd} is a refinement of the Rauch theorem demonstrated by Buser and Karcher~\cite[\S 6.4]{buser1981}. We make use of a simplified particular case of their general result:
\begin{lem}[Strong Rauch theorem]
  \label{lem:clean.strong.rauch}
  Assume the sectional curvatures on $\man$ satisfy $\abs{\seccurv} \leq \curvabsbnd$, and suppose there is a unique minimising geodesic between $x$ and $p$.  If $v = \exp_p^{-1}(x)$, and
  \begin{equation*}
    \norm{v} = \distM{p}{x} = r \leq \frac{\pi}{2\sqrt{\curvabsbnd}},
  \end{equation*}
  then
  \begin{equation*}
    \onorm{(d\exp_p)_v - T_{xp}} \leq \frac{\curvabsbnd  r^2}{2}.
  \end{equation*}
\end{lem}
\begin{proof}
  Given distinct upper and lower bounds on the sectional curvatures, $\curvlowbnd \leq \seccurv \leq \curvupbnd$, the result of
  Buser and Karcher~\cite[\S 6.4.2]{buser1981} is stated as
  \begin{equation*}
    \norm{ (d\exp_p)_{v}w - T_{xp} \left(\frac{\bfS_{\kappa}(r)w}{r} \right)}
    \leq \norm{w}
    \left(\frac{\bfS_{\kappa-\lambda}(r) - \bfS_{\kappa}(r)}{r} \right),
  \end{equation*}
  for any vector $w$ perpendicular to $v$, and as long as $\bfS_{\kappa}$ is nonnegative. Here $\kappa$ is arbitrary, and $\lambda = \max\{\curvupbnd - \kappa, \kappa - \curvlowbnd \}$.

  We take $\curvabsbnd = \max \{ \curvupbnd, - \curvlowbnd \}$, and $\kappa = 0$. The stated bound results since now $\bfS_{\kappa}(r) = \bfS_0(r) = r$, and the constraint $r \leq \frac{\pi}{2\sqrt{\curvabsbnd}}$ ensures that
  \begin{equation*}
    \frac{\bfS_{-\curvabsbnd}(r) - r}{r} \leq \frac{\curvabsbnd  r^2}{2},  
  \end{equation*}
  as observed in \Secref{sec:rauch.thm}. The result applies to all vectors since the exponential preserves lengths in the radial direction.
\end{proof}

We obtain a bound on the differential of the inverse of the exponential map from \Lemref{lem:clean.strong.rauch} and the following observation:
\begin{lem}
  \label{lem:inverse.bnd}
  Suppose $A: \ren \to \ren$ is a linear operator that satisfies
  \begin{equation*}
    \onorm{A-T} \leq \eta,
  \end{equation*}
  for some linear isometry $T:\ren \to \ren$. If $\eta \leq \frac{1}{2}$, then
  \begin{equation*}
    \onorm{A^{-1} - T^{-1}} \leq 2\eta.
  \end{equation*}
\end{lem}
\begin{proof}
  We first bound $\onorm{A^{-1}}= \sing{n}{A}^{-1}$, the inverse of the smallest singular value. Since $\sing{n}{A} = \sing{n}{T^{-1}A}$, and $\onorm{T^{-1}A - \Id} \leq \eta$, we have $\abs{\sing{n}{A} - 1} \leq \eta$. Thus $\sing{n}{A}^{-1} \leq (1-\eta)^{-1} \leq 1 + 2\eta$.

  Now write $A = T + \eta E$, where $\onorm{E} \leq 1$. The trick \cite[p. 50]{golub2012matrix} is to observe that
  \begin{align*}
    A^{-1} &= T^{-1} - A^{-1}(A-T)T^{-1}\\
    &= T^{-1} - \eta A^{-1}ET^{-1},
  \end{align*}
  and the stated bound follows.
\end{proof}

\begin{lem}
  \label{lem:raw.exp.trans.bnd}
  Suppose the sectional curvatures on $\man$ satisfy $\abs{\seccurv} \leq \curvabsbnd$. Let $v \in T_p\man$, with $y = \exp_p(v)$.  If $x,y \in \ballM{p}{\rho}$, with
  \begin{equation*}
    \rho \leq \min \left\{\frac{\injradM}{2},\frac{1}{2\sqrt{\curvabsbnd}} \right\},
  \end{equation*}
  then
  \begin{equation*}
    \onorm{d(\exp^{-1}_x \circ \exp_p)_v - T_{xy}T_{yp}} \leq 5\curvabsbnd \rho^2.
  \end{equation*}
\end{lem}
\begin{proof}
  By \Lemref{lem:convexity.bnd}, $\ballM{p}{\rho}$ is convex. Since $\distM{x}{y} < 2\rho$, \Lemref{lem:clean.strong.rauch} yields $\onorm{(d\exp_x)_w - T_{yx}} < 2\curvabsbnd \rho^2 \leq \frac{1}{2}$, where $w = \exp_x^{-1}(y)$. Since $T_{xy} = T_{yx}^{-1}$, we may use \Lemref{lem:inverse.bnd} to write
  \begin{equation*}
    (d\exp_x^{-1})_y = T_{xy} + \left( 4\curvabsbnd \rho^2 \right)E,
  \end{equation*}
  where $E$ satisfies $\onorm{E} \leq 1$.  We obtain the result by composing this with
  \begin{equation*}
    (d\exp_p)_v
    = T_{yp} + \left( \frac{\curvabsbnd \rho^2}{2} \right)\tilde{E},    
  \end{equation*}
 where $\onorm{\tilde{E}} \leq 1$.
\end{proof}

In order to compare $T_{xp}$ with $T_{xy}T_{yp}$ we exploit further estimates demonstrated by Buser and Karcher. If $\alpha: [0,1] \to \man$ is a curve,
let $T_{\alpha(t)}: T_{\alpha(0)}\man \to T_{\alpha(t)}\man$ denote the parallel translation operator (we do not require that $\alpha$ be a minimising geodesic). Buser and Karcher~\cite[\S 6.1, \S 6.2]{buser1981} bound the difference in the parallel translation operators between two homotopic curves:
\begin{lem}[Parallel translation comparison]
  \label{lem:parallel.trans.comp}
  Let $c_i: [0,1] \to \man$ be piecewise smooth curves from $p$ to $q$, and let
  \begin{equation*}
    c: [1,2] \times [0,1] \to \man
  \end{equation*}
  be a piecewise smooth homotopy between $c_1$ and $c_2$, i.e., $c(1,t) = c_1(t)$, and $c(2,t) = c_2(t)$. Let $\mathfrak{a} = \int \det dc_{(s,t)} \,dsdt$ be the area of the homotopy. If the sectional curvatures are bounded by $\abs{\seccurv} \leq \curvabsbnd$, then
  \begin{equation*}
    \onorm{T_{c_2(1)} - T_{c_1(1)}} \leq \frac{4}{3}\curvabsbnd \mathfrak{a}.
  \end{equation*}
\end{lem}

In our case the two curves of interest form the edges of a geodesic triangle. A \defn{geodesic triangle} in $\man$ is a set of three points (vertices) such that each pair is connected by a unique minimising geodesic, together with these three minimising geodesics (edges). Any three points in a convex set are the vertices of a geodesic triangle. Buser and Karcher \cite[\S 6.7]{buser1981} demonstrate an estimate of A. D. Aleksandrow that says that the edges of a small geodesic triangle are the boundary of a topological disk whose area admits a natural bound:
\begin{lem}[Small triangle area]
  \label{lem:small.tri.area}
  Let $p,x,y \in \man$ be the edges of a geodesic triangle whose edge lengths, $\ell_{px},\ell_{xy},\ell_{yp}$ satisfy
  \begin{equation*}
    \ell_{px} + \ell_{xy} + \ell_{yp}
    \leq \min \left\{ \injradM, \frac{2\pi}{\sqrt{\curvupbnd}} \right\},
  \end{equation*}
  where $\curvupbnd$ is an upper bound on the sectional curvatures of $\man$, and $\injradM$ is the injectivity radius. Then the edges of triangle $pxy$ form the boundary of an immersed topological disk whose area $\mathfrak{a}$ satisfies
  \begin{equation*}
    \mathfrak{a} \leq \mathfrak{a}_{\curvupbnd},
  \end{equation*}
  where $\mathfrak{a}_{\curvupbnd}$ is the area of a triangle with the same edge lengths in the sphere of radius $\frac{1}{\sqrt{\curvupbnd}}$.
\end{lem}
% FIXME -- probably the disk is embedded, not just immersed

Consider $x,y \in \ballM{p}{\rho}$, where $\rho < \frac{1}{2}\rho_0$
% \begin{equation*}
%   \rho < \frac{1}{2}\rho_0 = \frac{1}{2}\min \left\{ \frac{\injradM}{2}, \frac{\pi}{4\sqrt{\curvabsbnd}} \right\},
% \end{equation*}
and as usual $\curvabsbnd$ is a bound on the absolute values of the sectional curvatures. In this case, Buser and Karcher \cite[\S 6.7.1]{buser1981} observe that the area of the triangle in the sphere of radius $\frac{1}{\sqrt{\curvabsbnd}}$ that has the same edge lengths as $pxy$ satisfies %FIXME check
\begin{equation*}
  \mathfrak{a}_{\curvabsbnd} \leq \frac{5}{8}\rho^2.
\end{equation*}
It follows then, from \Lemref{lem:parallel.trans.comp}, that
\begin{equation*}
  \onorm{T_{xp} - T_{xy}T_{yp}} \leq \frac{5}{6}\curvabsbnd \rho^2.
\end{equation*}
This, together with \Lemref{lem:raw.exp.trans.bnd}, yields
\begin{align*}
  \onorm{d(\exp^{-1}_x \circ \exp_p)_v - T_{xp}}
  \leq 5\curvabsbnd \rho^2 + \frac{5}{6}\curvabsbnd \rho^2
  \leq 6\curvabsbnd \rho^2,
\end{align*}
and we obtain \Propref{prop:exp.trans.bnd}.

\subsection{Triangulations with Riemannian simplices}
\label{sec:triang.riem.splxs}

We now exploit \Propref{prop:exp.trans.bnd} to demonstrate that a bound of the form~\eqref{eq:strong.diff.bnd} is satisfied by the differential~\eqref{eq:diff.bary.map} of the barycentric coordinate map defining a Riemannian simplex $\riemsplxs$
\begin{equation*}
  d\bmap = - \left(\conn^{\man}\nu \right)^{-1} \partial_u \nu,
\end{equation*}
and find a bound on the scale that allows us to exploit \Propref{prop:generic.triangulation}.

Choose a vertex $p_0$ of $\riemsplxs$, and an arbitrary linear isometry $u: T_{p_0}\man \to \ren$ to establish a coordinate system on $\tanspace{p_0}{\man}$ so that $v_0(p_0)$ remains the origin. Let $P$ be the matrix whose $i^{\text{th}}$ column is $v_i(p_0)$. For $x \in \bmapball$, rather than placing an arbitrary coordinate system on $T_x\man$, we identify $T_{p_0}\man$ and $T_x\man$ by the parallel translation operator $T_{p_0x}$, i.e., use $u \circ T_{p_0x}$ for coordinates.  Let $\tilde{P}$ be the matrix whose $i^{\text{th}}$ column is $v_i(x) - v_0(x)$.

Now the map 
\begin{equation*}
  F: v \mapsto \exp_x^{-1} \circ \exp_{p_0}(v) - v_0(x)
\end{equation*}
can be considered as a map $\ren \supset U \to \ren$, and the matrix whose $i^{\text{th}}$ column is $F(v_i(p_0))$ is $\tilde{P}$.  It follows from \Propref{prop:exp.trans.bnd} that if $\scale < \frac{1}{2}\rho_0$, then for any $u \in \ballEn{0}{\scale}$, we have $\onorm{(dF)_u - \Id} \leq \eta$, with $\eta = 6\curvabsbnd \scale^2$.

\Lemref{lem:bound.disp} implies a bound on the difference of the column vectors of $P$ and $\tilde{P}$:
\begin{equation*}
  \norm{v_i(p_0) - (v_i(x) - v_0(x))} \leq \eta \norm{v_i(p_0)} \leq \eta \longedge{\splxs(p_0)}.
\end{equation*}
 It follows that $\onorm{P - \tilde{P}} \leq \sqrt{n}\eta \longedge{\splxs(p_0)}$. Assume also that  $\thickness{\splxs(p_0)} \geq \thickbnd$. Then, recalling \Eqnref{eq:matrix.of.deriv}, and recognising that $T_{xp_0}$ is represented by the identity matrix in our coordinate systems, we have
 \begin{align*}
   \onorm{- \left(\partial_u \nu \right) - T_{xp_0}}
   &= \onorm{  \tilde{P}P^{-1} - PP^{-1} }\\
   &= \onorm{  \left(\tilde{P} - P \right) P^{-1} }\\
   &\leq \sqrt{n}6 \curvabsbnd \scale^2 \longedge{\splxs(p_0)} \onorm{P^{-1}}\\
   &\leq \frac{\sqrt{n}6\curvabsbnd \scale^2 \longedge{\splxs(p_0)}}{\sqrt{n} \thickbnd \longedge{\splxs(p_0)}}
   \qquad \text{by \Lemref{lem:bound.skP}}\\
   &\leq \frac{6 \curvabsbnd \scale^2}{\thickbnd}.
 \end{align*}

Buser and Karcher show \cite[\S 8.1.3]{buser1981} that for any $x \in \ballM{p}{h}$, with $h < \rho_0$, we have
\begin{align*}
%  \label{eq:bnk.hessian.bnd}
  \onorm{(\conn^{\man}\nu)_x - \Id} \leq 2\curvabsbnd h^2.
\end{align*}
When $h < \frac{1}{2}\rho_0$, we have $2\curvabsbnd h^2 < \frac{1}{2}$, and \Lemref{lem:inverse.bnd} yields
\begin{equation*}
  \onorm{(\conn^{\man}\nu)^{-1} - \Id} \leq 4\curvabsbnd h^2.  
\end{equation*}

Therefore we have, when $\bmap(u) = x$
\begin{equation}
  \label{eq:bnd.db}
  \begin{split}
  \onorm{d\bmap_u - T_{xp_0}}
  &= \onorm{- \left(\conn^{\man}\nu \right)^{-1} \partial_u \nu - T_{xp_0}}\\
  &\leq 4\curvabsbnd h^2 + \frac{6 \curvabsbnd h^2}{\thickbnd}
  + 4\curvabsbnd h^2
  \left( \frac{6 \curvabsbnd h^2}{\thickbnd} \right)\\
  &\leq \frac{14 \curvabsbnd h^2}{\thickbnd},
  \end{split}
\end{equation}
using $h < \frac{\pi}{8\sqrt{\curvabsbnd}}$.

Finally, in order to employ \Propref{prop:generic.triangulation} we consider the composition $\exp^{-1}_{p_0} \circ \bmap$. From \Lemref{lem:clean.strong.rauch} and \Lemref{lem:inverse.bnd} we have that
\begin{align*}
  \onorm{ (d \exp^{-1}_{p_0})_x - T_{p_0x}} \leq \curvabsbnd h^2.
\end{align*}
Therefore, since $T_{p_0x} = T_{xp_0}^{-1}$ we have
\begin{align*}
  \onorm{ d(\exp_{p_0}^{-1} \circ \bmap)_u - \Id }
  &\leq \curvabsbnd h^2 + \frac{14 \curvabsbnd h^2}{\thickbnd}
  + \curvabsbnd h^2 \left( \frac{14 \curvabsbnd h^2}{\thickbnd} \right)\\
  &\leq \frac{17 \curvabsbnd h^2}{\thickbnd}.
\end{align*}

In order to meet the conditions of \Propref{prop:generic.triangulation}, we require
\begin{align*}
  \frac{17 \curvabsbnd \scale^2}{\thickbnd}
  \leq \frac{1}{2}n\thickbnd,
\end{align*}
or
\begin{align*}
\scale \leq \frac{\sqrt{n}\thickbnd}{6 \sqrt{\curvabsbnd}}.
\end{align*}

%\rd{define barycentric coordinate map}
We obtain

\begin{mainthm}
  \label{thm:triangulation}
  Suppose $\man$ is a compact $n$-dimensional Riemannian manifold with sectional curvatures $\seccurv$ bounded by $\abs{\seccurv} \leq \curvabsbnd$, and $\acplx$ is an abstract simplicial complex with finite vertex set $\pts \subset \man$.  
Define a quality parameter $\thickbnd > 0$, and let
    \begin{align*}
      \scale = \min \left\{ \frac{\injradM}{4},
      \frac{\sqrt{n}\thickbnd}{6 \sqrt{\curvabsbnd}} \right\}.
    \end{align*}
If
  \begin{enumerate}
  \item For every $p \in \pts$, the vertices of $\str{p}$ are contained in $\ballM{p}{h}$, and the balls $\{\ballM{p}{h}\}_{p \in \pts}$ cover $\man$.
  \item For every $p \in \pts$, the restriction of the inverse of the exponential map $\exp_p^{-1}$ to the vertices of $\str{p} \subset \acplx$ defines a piecewise linear embedding of $\carrier{\str{p}}$ into $\tanspace{p}{\man}$, realising $\str{p}$ as a full star such that every simplex $\splxs(p)$ has thickness $\thickness{\splxs(p)} \geq \thickbnd$.
% and longest edge length $\longedge{\splxs(p)} < \scale$.
  \end{enumerate}
  then $\acplx$ triangulates $\man$, and the triangulation is given by the barycentric coordinate map on each simplex.
\end{mainthm}

% -*- LaTeX -*-
% outmetric.tex
% 20140418
% 

\section{The piecewise flat metric}
\label{sec:pwf.metric}

The complex $\acplx$ described in \Thmref{thm:triangulation} naturally inherits a piecewise flat metric from the construction. The length assigned to an edge $\asimplex{p,q} \in \acplx$ is the geodesic distance in $\man$ between its endpoints: $\ell_{pq} = \distM{p}{q}$. We first examine, in \Secref{sec:abstract.eucl.splx}, conditions which ensure that this assignment of edge lengths does indeed make each $\splxs \in \acplx$ isometric to a Euclidean simplex. With this piecewise flat metric on $\acplx$, the barycentric coordinate map is a bi-Lipschitz map between metric spaces $H:\carrier{\acplx} \to \man$. In \Secref{sec:metric.distortion} we estimate the metric distortion of this map.

Several of the lemmas in this section are generalisations of lemmas that appeared in \cite[\S A.1]{boissonnat2013manmesh.inria}. The arguments are essentially the same, but we have included the proofs here for convenience.

\subsection{Euclidean simplices defined by edge lengths}
\label{sec:abstract.eucl.splx}

If $G$ is a symmetric positive definite $n \times n$ matrix, then it can be written as a Gram matrix, $G = \transp{P}P$ for some $n \times n$ matrix $P$. Then $P$ describes a Euclidean simplex with one vertex at the origin, and the other vertices defined by the column vectors. The matrix $P$ is not unique, but if $G= \transp{Q}Q$, then $Q = OP$ for some linear isometry $O$. Thus a symmetric positive definite matrix defines a Euclidean simplex, up to isometry. 

If $\splxs = \asimplex{p_0, \ldots, p_n} \subset \bmapball$, is the vertex set of a Riemannian simplex $\riemsplxs$, we define the numbers $\ell_{ij} = \distM{p_i}{p_j}$. These are the edge lengths of a Euclidean simplex $\splxsE$ if and only if the matrix $G$ defined by
\begin{equation}
  \label{eq:G.mat}
  G_{ij} = \frac{1}{2}(\ell_{0i}^2 + \ell_{0j}^2 - \ell_{ij}^2)
\end{equation}
is positive definite.

We would like to use the smallest eigenvalue of $G$ to estimate the thickness of $\splxsE$, however, an unfortunate choice of vertex labels can prevent us from doing this easily. We make use of the following observation:
\begin{lem}
  \label{lem:low.bnd.thick}
  Suppose $\splxs = \asimplex{v_0, \ldots, v_k} \subset \R^n$ is a Euclidean $k$-simplex, and let $P$ be the $n\times k$ matrix whose $i^{\text{th}}$ column is $v_i - v_0$. If for some $i\neq 0$, an altitude at least as small as $\gsplxalt_0$ is realised, i.e,  $\gsplxalt_i \leq \gsplxalt_0$, then
  \begin{equation*}
    \thickness{\splxs} \geq \frac{\sing{k}{P}}{k \longedge{\splxs}}.
  \end{equation*}
\end{lem}
\begin{proof}
  We assume that $\splxs$ is non-degenerate, since otherwise the bound
  is trivial.  If $v_i$ is a vertex of minimal altitude, then by
  \Lemref{lem:bound.skP}, the $i^{\text{th}}$ row of the
  pseudo-inverse $\pseudoinv{P}$ is
  given by $\transp{w_i}$, where
  \begin{equation*}
    \norm{w_i} = \gsplxalt_i^{-1} = (k\gthickness \glongedge)^{-1}.
  \end{equation*}
  It follows then that $\sing{1}{\pseudoinv{P}} \geq (k \gthickness \glongedge)^{-1}$, and therefore $\sing{k}{P} \leq k \gthickness \glongedge$, yielding the stated bound.
\end{proof}

If $G$ is positive definite, then we may write $G= \transp{P}P$, where $P$ is a matrix describing $\splxsE = \simplex{v_0, \ldots, v_k}$, with the edge lengths $\{\ell_{ij}\}$ dictating the vertex labelling. If $\mu_k(G) = \sing{k}{P}^2$ is the smallest eigenvalue of $G$, then \emph{provided some vertex other than $v_0$ realises the smallest altitude in $\splxsE$}, \Lemref{lem:low.bnd.thick} yields
\begin{equation}
  \label{eq:eval.bnd.thick}
  \thickness{\splxsE} \geq \frac{\sqrt{\mu_k(G)}}{k \longedge{\splxs}}.
\end{equation}

For our current purposes, we can ensure the existence, and bound the thickness of $\splxsE$ by comparing it with a related simplex such as $\splxs(p_0)$. To this end we employ the following observation (where $\tilde{\splxs}$ plays the role of $\splxs(p_0)$):
\begin{lem}
  \label{lem:splx.gram}
  Suppose $\tilde{\splxs} = \asimplex{\tilde{v}_0, \ldots, \tilde{v}_k}$ is a Euclidean $k$-simplex, and $\{{\ell}_{ij}\}$ is a set of positive numbers defined for all $0 \leq i\neq j \leq k$ such that ${\ell}_{ij} = {\ell}_{ji}$, and
  \begin{equation*}
    \abs{\norm{\tilde{v}_i - \tilde{v}_j} - {\ell}_{ij}} \leq
    \metlipconst \longedge{\tilde{\splxs}}.
  \end{equation*}
  Let $\tilde{P}$ be the matrix whose $i^{\text{th}}$ column is $\tilde{v}_i - \tilde{v}_0$, and define the matrix $G$ by
  \begin{equation*}
    G_{ij} = \frac{1}{2}({\ell}_{0i}^2 + {\ell}_{0j}^2 - {\ell}_{ij}^2).    
  \end{equation*}
Let $E$ be the matrix that records the difference between $G$ and the Gram matrix $\transp{\tilde{P}}\tilde{P}$:
  \begin{equation*}
    G = \transp{\tilde{P}} \tilde{P} + E.
  \end{equation*}
  If $\metlipconst \leq \frac{2}{3}$, then the entries of $E$ are
  bounded by $\abs{E_{ij}} \leq 4 \metlipconst \longedge{\tilde{\splxs}}^2$,
  and in particular
  \begin{equation}
    \label{eq:gram.err.bnd}
    \onorm{E} \leq 4 k \metlipconst \longedge{\tilde{\splxs}}^2.
  \end{equation}
\end{lem}
\begin{proof}
  Let $\tilde{\ell}_{ij} = \norm{\tilde{v}_i- \tilde{v}_j}$. By the cosine rule we have
  \begin{equation*}
    \left[\transp{\tilde{P}} \tilde{P} \right]_{ij} = \frac{1}{2}(\tilde{\ell}_{0i}^2 + \tilde{\ell}_{0j}^2 - \tilde{\ell}_{ij}^2),
  \end{equation*}
  and we obtain a bound on the
  magnitude of the coefficients of $E$:
  \begin{equation*}
    \begin{split}
      \abs{ G_{ij} - \left[\transp{\tilde{P}} \tilde{P} \right]_{ij} }
      &\leq
      \frac{1}{2} \left( \abs{ \ell_{0i}^2 - \tilde{\ell}_{0i}^2 }
        + \abs{ \ell_{0j}^2  - \tilde{\ell}_{0j}^2 }
        + \abs{ \ell_{ij}^2 - \tilde{\ell}_{ij}^2 }
      \right) \\
      &\leq \frac{3}{2}(2 + \metlipconst)\metlipconst
      \longedge{\tilde{\splxs}}^2\\
      &\leq 4 \metlipconst \longedge{\tilde{\splxs}}^2.
    \end{split}
  \end{equation*}

  This leads us to a bound on $\sing{1}{E} = \norm{E}$. Indeed, the
  magnitude of the column vectors of $E$ is bounded by $\sqrt{k}$
  times a bound on the magnitude of their coefficients, and the
  magnitude of $\sing{1}{E}$ is bounded by $\sqrt{k}$ times a bound on
  the magnitude of the column vectors. We obtain
  \Eqnref{eq:gram.err.bnd}.
\end{proof}

We have the following extension of the ``Thickness under distortion'' \Lemref{lem:intrinsic.thick.distortion}  (\cite[\S 4.2]{boissonnat2013manmesh.inria}):
\begin{lem}[Abstract Euclidean simplex]
  \label{lem:abstract.eucl.splx}
  Suppose $\tilde{\splxs} = \asimplex{\tilde{v}_0, \ldots, \tilde{v}_k} \subset \R^n$, and $\{\ell_{ij}\}_{0\leq i,j \leq k}$ is a set of positive numbers defined for all $0 \leq i\neq j \leq k$ such that ${\ell}_{ij} = {\ell}_{ji}$, and such that
  \begin{equation*}
    \abs{\norm{\tilde{v}_i - \tilde{v}_j} - \ell_{ij}} \leq
    \metlipconst \longedge{\tilde{\splxs}}
  \end{equation*}
  for all $0 \leq i < j \leq k$.  

  If
  \begin{equation}
    \label{eq:defn.metlipconst}
    \metlipconst =  \frac{\eta \thickness{\tilde{\splxs}}^2}{4}
   \qquad \text{ with } \quad 0 \leq \eta \leq 1,
  \end{equation}
  then there exists a Euclidean simplex $\splxs =
  \asimplex{{v}_0, \ldots, {v}_k}$ whose edge lengths are described by the numbers $\ell_{ij}$. Let $\tilde{P}$ and $P$ be matrices whose $i^{\text{th}}$ column is given by $\tilde{v}_i - \tilde{v}_0$, and $v_i-v_0$ respectively. Then
  \begin{equation*}
    \sing{k}{{P}} \geq (1 - \eta)\sing{k}{\tilde{P}},
  \end{equation*}
  and the thickness of $\splxs$ satisfies
  \begin{equation*}
    \thickness{\splxs} \geq \frac{4}{5\sqrt{k}}(1 - \eta)
    \thickness{\tsplxs}. 
  \end{equation*}
\end{lem}
\begin{proof}
  If $\tsplxs$ is degenerate, then by \eqref{eq:defn.metlipconst}, $\{\ell_{ij}\}$ is the set of edge lengths of $\tsplxs$ and there is nothing to prove. Therefore, assume $\thickness{\tsplxs}>0$.

  Let $G$ be the matrix defined by \Eqnref{eq:G.mat}, and
  define the matrix $E$ by $G= \transp{\tilde{P}}\tilde{P} + E$, and let $x$ be a unit eigenvector of $G$ associated with the smallest eigenvalue $\mu_k$. Then
  \begin{equation*}
    \begin{split}
      \mu_k = \transp{x}Gx
      &= \transp{x}\transp{\tilde{P}}\tilde{P}x + \transp{x}Ex\\
      &\geq \sing{k}{\tilde{P}}^2 - \sing{1}{E}\\
    &= \left(1 - \frac{\sing{1}{E}}{\sing{k}{\tilde{P}}^2} \right) \sing{k}{\tilde{P}}^2.
    \end{split}
  \end{equation*}
  From \Lemref{lem:bound.skP} we have $\sing{k}{\tilde{P}}^2 \geq k \thickness{\tsplxs}^2 \longedge{\tsplxs}^2$, and by \Lemref{lem:splx.gram} $\sing{1}{E} \leq 4k\metlipconst \longedge{\tsplxs}^2$, so by the definition~\eqref{eq:defn.metlipconst} of $\metlipconst$, we have that 
  \begin{equation*}
    \mu_k \geq (1 - \eta)\sing{k}{\tilde{P}}^2,
  \end{equation*}
  and thus $G$ is positive semi-definite, and the first inequality is satisfied because $\mu_k = \sing{k}{P}^2$ and $\sqrt{1-\eta} \geq 1-\eta$.

  In order to obtain the thickness bound, we employ \Lemref{lem:low.bnd.thick}. Since thickness is independent of the vertex labelling, we may assume that some vertex other than ${v}_0$ realises the minimal altitude in ${\splxs}$ (if necessary, we relabel the vertices of $\tsplxs$ and ${\splxs}$, maintaining the correspondence).  Then using \Lemref{lem:low.bnd.thick} and \Lemref{lem:bound.skP} we have
  \begin{equation*}
    k \thickness{{\splxs}} \longedge{{\splxs}}
    \geq \sing{k}{{P}} \geq (1 - \eta)\sing{k}{\tilde{P}}
    \geq (1- \eta)\sqrt{k}\thickness{\tsplxs}\longedge{\tsplxs}.
  \end{equation*}
  The stated thickness bound follows since
  $\frac{\longedge{\tsplxs}}{\longedge{\splxs}} \geq \frac{1}{1 +
    \metlipconst} \geq \frac{4}{5}$.
\end{proof}

%\subsection{Criteria for a piecewise flat metric}

Now we examine whether the simplices of the complex $\acplx$ of \Thmref{thm:triangulation} meet the requirements of \Lemref{lem:abstract.eucl.splx}.  If $\splxs \in \acplx$, with $p \in \splxs$, then we can use the Rauch theorem~\ref{lem:poly.rauch} to compare $\splxs$ with the Euclidean simplex $\splxs(p) \in T_p\man$. Under the assumptions of \Thmref{thm:triangulation}, we have $\onorm{d\exp_p} \leq 1 +\frac{\curvabsbnd \scale^2}{2}$, and $\onorm{d\exp_p^{-1}} \leq 1 + \frac{\curvabsbnd \scale^2}{3}$. Thus
\begin{align*}
  \ell_{ij} - \norm{v_i(p) - v_j(p)} \leq \frac{\curvabsbnd \scale^2}{2}
  \norm{v_i(p) - v_j(p)},
\end{align*}
and
\begin{align*}
  \norm{v_i(p) - v_j(p)} - \ell_{ij}
  \leq \frac{\curvabsbnd \scale^2}{3} \ell_{ij}
  \leq \frac{\curvabsbnd \scale^2}{3}
  \left(1 + \frac{\curvabsbnd \scale^2}{2} \right)
  \norm{v_i(p) - v_j(p)}
  \leq \frac{\curvabsbnd \scale^2}{2} \norm{v_i(p) - v_j(p)},
\end{align*}
and we can use
\begin{equation}
  \label{eq:explicit.metlipconst}
  \metlipconst = \frac{\curvabsbnd \scale^2}{2}
\end{equation}
in \Lemref{lem:abstract.eucl.splx}. Thus in order to guarantee that the $\ell_{ij}$ describe a non-degenerate Euclidean simplex, we require that
\begin{equation*}
  \curvabsbnd \scale^2 = \frac{\eta\thickbnd^2}{2},
\end{equation*}
for some non-negative $\eta < 1$.

Under the conditions of \Thmref{thm:triangulation} we may have $\scale^2 \curvabsbnd = \frac{n \thickbnd^2}{36}$, which gives us $\eta = \frac{n}{18}$. Thus when $n \geq 18$ we require stronger bounds on the scale than those imposed by \Thmref{thm:triangulation} if we wish to ensure the existence of a piecewise flat metric on $\acplx$. Reducing the curvature controlled constraint on $\scale$ in \Thmref{thm:triangulation} by a factor of $1/\sqrt{n}$ gives us $\eta = \frac{1}{18}$, and \Lemref{lem:abstract.eucl.splx} yields:
\begin{prop}
  \label{prop:pwflat.metric.exists}
  If the requirements of \Thmref{thm:triangulation}, are satisfied when the scale parameter~\eqref{eq:sampling.density} is replaced with
  \begin{equation*}
    \scale = \min \left\{ \frac{\injradM}{4},
      \frac{\thickbnd}{6 \sqrt{\curvabsbnd}} \right\},
  \end{equation*}
  then the geodesic distances between the endpoints of the edges in $\acplx$ defines a piecewise flat metric on $\acplx$ such that each simplex $\splxs \in \acplx$ satisfies
  \begin{equation*}
    \thickness{\splxs} > \frac{3}{4\sqrt{n}}\thickbnd.
  \end{equation*}
\end{prop}

\subsection{Metric distortion of the barycentric coordinate map}
\label{sec:metric.distortion}

In the context of \Thmref{thm:triangulation} the barycentric coordinate map on each simplex defines a piecewise smooth homeomorphism $H: \carrier{\acplx} \to \man$. If the condition of \Propref{prop:pwflat.metric.exists} is also met, then $\acplx$ is naturally endowed with a piecewise flat metric. We wish to compare this metric with the Riemannian metric on $\man$.  It suffices to consider an $n$-simplex $\splxs \in \acplx$, and establish bounds on the singular values of the differential $dH$. If $p \in \splxs$, then we can write $H \big|_{\splxsE} = \bmap \circ \lmap_p$, where $\lmap_p: \splxsE \to \splxsE(p)$ is the linear map that sends $\splxs \in \acplx$ to $\splxs(p) \in T_p\man$.

A bound on the metric distortion of a linear map that sends one Euclidean simplex to another is a consequence of the following (reformulation of \cite[Lem A.4]{boissonnat2013manmesh.inria}):
\begin{lem}[Linear distortion bound]
  \label{lem:good.isometry}
  Suppose that $P$ and $\tilde{P}$ are non-degenerate $k \times k$
  matrices such that
  \begin{equation}
    \label{eq:gram.perturb}
    \transp{\tilde{P}}\tilde{P} = \transp{P} P + E.
  \end{equation}
  Then there exists a linear isometry $\Phi: \reel^k \to \reel^k$ such
  that 
  \begin{equation*}
    \onorm{\tilde{P}P^{-1} - \Phi }
    \leq \frac{\sing{1}{E}}{\sing{k}{P}^2}. 
  \end{equation*}
\end{lem}
\begin{proof}
  Multiplying by $\invtransp{P} := \inv{ ( \transp{P})}$ on the left,
  and by $\inv{P}$ on the right, we rewrite \Eqnref{eq:gram.perturb} as
  \begin{equation}
    \label{eq:A.F}
   \transp{A}A = I + F,
  \end{equation}
  where $A = \tilde{P} \inv{P}$, and $F = \invtransp{P}E \inv{P}$.
  Using the singular value decomposition $A = U_A \Sigma_A
  \transp{V}_A$, we let $\Phi = U_A \transp{V}_A$ so that
  \begin{equation}
    \label{eq:expand.mat.diff}
    (A - \Phi)  = U_A( \Sigma_A - I ) \transp{V}_A .
  \end{equation}
  From \Eqnref{eq:A.F} we deduce that $\sing{1}{A}^2 \leq 1 + \sing{1}{F}$, and also that $\sing{k}{A}^2 \geq 1 - \sing{1}{F}$. Using these two inequalities we find
  \begin{equation*}
    \max_i \abs{\sing{i}{A} - 1} \leq \frac{\sing{1}{F}}{1 +
      \sing{i}{A}} \leq \sing{1}{F},
  \end{equation*}
  and thus
 \begin{equation*}
    \onorm{\Sigma_A - I} \leq \sing{1}{F} \leq \sing{1}{\inv{P}}^2
    \sing{1}{E} = \sing{k}{P}^{-2} \sing{1}{E}.
  \end{equation*}
  The result now follows from \Eqnref{eq:expand.mat.diff}.
\end{proof}

\Lemref{lem:good.isometry} implies:
\begin{lem}
  \label{lem:affine.metric.distort}
  Suppose $\splxs = \asimplex{v_0, \ldots, v_n}$ and $\tsplxs = \asimplex{\tilde{v}_0, \ldots, \tilde{v}_n}$ are two Euclidean simplices in $\R^n$ such that
  \begin{align*}
    \abs{\, \norm{\tilde{v}_i - \tilde{v}_j} - \norm{v_i - v_j} \,}
    \leq \metlipconst \longedge{\splxs}.
  \end{align*}

  If $A: \R^n \to \R^n$ is the affine map such that $A(v_i) = \tilde{v}_i$ for all $i$, and $\metlipconst \leq \frac{2}{3}$, then for all $x,y \in \R^n$,
  \begin{align*}
    \abs{\, \norm{A(x) - A(y)} - \norm{x-y} \,}
    \leq \eta  \norm{x-y},
  \end{align*}
  where
  \begin{equation*}
    \eta =\frac{4 \metlipconst}{\thickness{\splxs}^2}.
  \end{equation*}
\end{lem}
\begin{proof}
  Let $P$ be the matrix whose $i^{\text{th}}$ column is $v_i - v_0$, and let $\tilde{P}$ be the matrix whose $i^{\text{th}}$ column is $\tilde{v}_i - \tilde{v}_0$. Then we have the matrix form $A(x) = \tilde{P}P^{-1}x +(\tilde{v}_0 - \tilde{P}P^{-1}v_0)$. It follows then from \Lemref{lem:good.isometry} that $\eta \leq \sing{n}{P}^{-2}\sing{1}{E}$, where $E = \transp{\tilde{P}}\tilde{P} - \transp{P}P$.

  By \Lemref{lem:splx.gram}, $\sing{1}{E} \leq 4n \metlipconst \longedge{\splxs}^2$, and by \Lemref{lem:bound.skP}, $\sing{n}{P}^2 \geq n \thickness{\splxs}^2\longedge{\splxs}^2$, and the result follows.
\end{proof}

%NoteRD: the constants in the following calculations could be tightened somewhat. For example $(1-x)^{-1} \leq 1 + ((1+k)/k)x$ if $x \leq 1/(1+k)$ can be exploited for a much larger $k$ than the $k=1$ that we use.
%
Observe that if $A$ in \Lemref{lem:affine.metric.distort} is a linear map, then the lemma states that $\sing{1}{A} \leq 1+\eta$ and $\sing{n}{A} \geq 1 - \eta$. We use this to estimate the metric distortion of $H|_{\splxsE} = \bmap \circ \lmap_p$.  Under the assumption of \Propref{prop:pwflat.metric.exists}, specifically, given that $h \leq \frac{\thickbnd}{6\sqrt{\curvabsbnd}}$, we again exploit \Eqnref{eq:explicit.metlipconst}, so
\begin{equation*}
  \onorm{\lmap_p^{-1}} \leq 1 + \frac{2\curvabsbnd \scale^2}{\thickbnd^2}
\end{equation*}
and, since $\onorm{\lmap_p}^{-1} = \sing{n}{\lmap_p^{-1}} \geq 1 - \frac{2\curvabsbnd \scale^2}{\thickbnd^2}$, and the second term is less than $\frac{1}{2}$, we also have
\begin{equation*}
  \onorm{\lmap_p} \leq 1 + \frac{4\curvabsbnd \scale^2}{\thickbnd^2}.
\end{equation*}
Using \Eqnref{eq:bnd.db} we have
\begin{equation*}
  \onorm{d\bmap} \leq 1 + \frac{14\curvabsbnd \scale^2}{\thickbnd},
\end{equation*}
and
\begin{equation*}
  \onorm{d\bmap^{-1}} \leq 1 + \frac{28\curvabsbnd \scale^2}{\thickbnd}.
\end{equation*}
Recalling that $dH \big|_{\splxsE} = (d\bmap)  \lmap_p$, and $\scale^2 \leq \frac{\thickbnd^2}{36\curvabsbnd}$, we obtain
\begin{equation*}
  \onorm{dH} \leq 1 + \frac{20\curvabsbnd \scale^2}{\thickbnd^2},
\end{equation*}
and
\begin{equation*}
  \onorm{dH^{-1}} \leq 1 + \frac{32\curvabsbnd \scale^2}{\thickbnd^2}.
\end{equation*}

The bound on the differential of $H$ and its inverse enables us to estimate the Riemannian metric on $\man$ using the piecewise flat metric on $\acplx$. The metric distortion bound on $H$ is found with the same kind of calculation as exhibited in \Eqnref{eq:bnd.edge.distort}, for example. We find:
\begin{mainthm}[Metric distortion]
  \label{thm:metric.distortion}
  If the requirements of \Thmref{thm:triangulation}, are satisfied with the scale parameter~\eqref{eq:sampling.density} replaced by
  \begin{equation*}
    \scale = \min \left\{ \frac{\injradM}{4},
      \frac{\thickbnd}{6 \sqrt{\curvabsbnd}} \right\},
  \end{equation*}
  then $\acplx$ is naturally equipped with a piecewise flat metric $\gdistA$ defined by assigning to each edge the geodesic distance in $\man$ between its endpoints.

  If $H:\carrier{\acplx} \to \man$ is the triangulation defined by the barycentric coordinate map in this case, then the metric distortion induced by $H$ is quantified as
  \begin{equation*}
    \abs{ \distM{H(x)}{H(y)} - \distA{x}{y} } \leq \frac{50 \curvabsbnd \scale^2}{\thickbnd^2} \distA{x}{y},
  \end{equation*}
for all $x,y \in \carrier{\acplx}$.
% NoteRD: for ref, other way, I found
%   \begin{equation*}
%     \abs{ \distM{H(x)}{H(y)} - \distA{x}{y} } \leq \frac{31 \curvabsbnd \scale^2}{\thickbnd^2 \separation} \distM{H(x)}{H(y)} ...
%   \end{equation*}
% dunno what the best way to express ``metric distortion'' is ... our way doesn't seem to be standard
\end{mainthm}

% SLAG:
%   If $H:\carrier{\acplx} \to \man$ is the triangulation defined in \Thmref{thm:triangulation}, and the additional constraint of \Propref{prop:pwflat.metric.exists} is met, then $\acplx$ is naturally equipped with a piecewise flat metric $\gdistA$ defined by assigning to each edge the geodesic distance in $\man$ between its endpoints. The metric distortion induced by $H$ is quantified as
%   \begin{equation*}
%     \abs{ \distM{H(x)}{H(y)} - \distA{x}{y} } \leq \frac{41 \curvabsbnd \scale^2}{\thickbnd^2} \distA{x}{y},
%   \end{equation*}
% for all $x,y \in \carrier{\acplx}$.

%\clearpage
\appendix

% -*- LaTeX -*-
% alt_criteria.tex
% 20140427
%

\section{Alternate criteria}
\label{sec:alt.criteria}

We discuss alternative formulations of our results. In \Secref{sec:alt.intrinsic.criteria}, we consider defining the quality of Riemannian simplices in terms of Euclidean simplices defined by the geodesic edge lengths of the Riemannian simplices. In \Secref{sec:phat} we compare thickness with a volume-based quality measure for simplices that we call \defn{fatness}.

\subsection{In terms of the intrinsic metric}
\label{sec:alt.intrinsic.criteria}

We imposed a quality bound on a Riemannian simplex $\riemsplxs$ by imposing a quality bound on the Euclidean simplex $\splxs(p)$ that is the lift of the vertices of $\riemsplxs$ to $T_p\man$. This was convenient for our purposes, but the quality of $\riemsplxs$ could also be characterised directly by its geodesic edge lengths.

As discussed in \Secref{sec:abstract.eucl.splx}, we can use the smallest eigenvalue of the matrix~\eqref{eq:G.mat} $G$ to characterise the quality of $\riemsplxs$: When $\mu_n(G) \geq 0$, there is a Euclidean simplex $\splxsE$ with the same edge lengths as $\riemsplxs$, however we have the inconvenience that the lower bound~\eqref{eq:eval.bnd.thick} on $\thickness{\splxsE}$ with respect to $\mu_k(G)$ is not valid for all choices of vertex labels.

This inconvenience can be avoided if a volumetric quality measure is used, such as the fatness discussed in \Secref{sec:phat}. Determinant-based criteria for Euclidean simplex realisability are discussed by Berger~\cite[\S 9.7]{BergerGeometryI}, for example.

In any event, we will express the alternate non-degeneracy criteria for $\riemsplxs$ in terms of the thickness of the associated Euclidean simplex $\splxsE$. Using \Propref{prop:affine.indep}, and the Rauch theorem~\ref{lem:poly.rauch}, we have the following  reformulation of the non-degeneracy criteria of \Thmref{thm:thick.nondegen.riem}:
\begin{prop}[Non-degeneracy criteria]
  \label{prop:alt.non.degen}
  If $\rho < \barymapradbnd$ defined in \Eqnref{eq:good.barymap.rad}, and the geodesic edge lengths of $\riemsplxs \subset \bmapball \subset \man$ define a Euclidean simplex $\splxsE$ with 
  \begin{equation}
    \label{eq:intrinsic.nondegen}
    \thickness{\splxsE} \geq 3 \sqrt{\curvabsbnd} \longedge{\splxsE}
  \end{equation}
  then $\riemsplxs$ is non-degenerate. As in \Thmref{thm:thick.nondegen.riem}, the assertion holds if $\rho$ replaces $\longedge{\splxsE}$ in the lower bound~\eqref{eq:intrinsic.nondegen}. 
\end{prop}
\begin{proof}
  By \Lemref{lem:poly.rauch} we have for any $x \in \bmapball$
  \begin{equation*}
    \norm{v_i(x) - v_j(x)} \leq \left(1 + \frac{\curvabsbnd(2 \rho)^2}{3} \right) \ell_{ij},
  \end{equation*}
  and
  \begin{equation*}
    \ell_{ij}   \leq \left(1 + \frac{\curvabsbnd(2 \rho)^2}{2} \right) \norm{v_i(x) - v_j(x)}.
  \end{equation*}
  Therefore
  \begin{equation*}
    \begin{split}
    \abs{ \, \norm{v_i(x) - v_j(x)} - \ell_{ij} \,}
    &\leq
    \curvabsbnd 2 \rho^2 \left(1 + \frac{\curvabsbnd 4 \rho^2}{3} \right) \ell_{ij}\\
    &\leq
    4 \curvabsbnd \rho^2 \ell_{ij}.
    \end{split}
  \end{equation*}
  Then using $\metlipconst = 4 \curvabsbnd \rho^2$ in \Lemref{lem:intrinsic.thick.distortion}, we see that $\splxs(x)$ is non-degenerate if $\thickness{\splxsE} > \sqrt{8} \sqrt{\curvabsbnd} \rho$, and the result follows from \Propref{prop:affine.indep}, and the remarks at the end of \Secref{sec:bary.map}.
\end{proof}

The scale parameter $h$ in \Thmref{thm:triangulation} is in fact a strict upper bound on the geodesic edge lengths $\ell_{ij}$ in $\acplx$. A similar argument to the proof of \Propref{prop:alt.non.degen} allows us to restate \Thmref{thm:triangulation} by employing a thickness bound on the Euclidean simplices with edge lengths $\ell_{ij}$:
\begin{prop}[Triangulation criteria]
  Suppose $\man$ is a compact $n$-dimensional Riemannian manifold with sectional curvatures $\seccurv$ bounded by $\abs{\seccurv} \leq \curvabsbnd$, and $\acplx$ is an abstract simplicial complex with finite vertex set $\pts \subset \man$.  
  Define a quality parameter $\thickbnd > 0$, and let
  \begin{align*}
    \scale = \min \left\{ \frac{\injradM}{4},
      \frac{\thickbnd}{8 \sqrt{\curvabsbnd}} \right\}.
  \end{align*}
  If
  \begin{enumerate}
  \item For every simplex $\splxs = \asimplex{p_0, \ldots, p_n} \in \acplx$, the edge lengths $\ell_{ij} = \distM{p_i}{p_j}$ satisfy $\ell_{ij} < \scale$, and they define a Euclidean simplex $\splxsE$ with $\thickness{\splxsE} \geq \thickbnd$.
  \item The balls $\{\ballM{p}{h}\}_{p \in \pts}$ cover $\man$, and for each $p \in \pts$ the secant map of $\exp_p^{-1}$ realises $\str{p}$ as a full star.
  \end{enumerate}
  then $\acplx$ triangulates $\man$, and the triangulation is given by the barycentric coordinate map on each simplex.
%
%     Suppose $\man$ is a compact Riemannian manifold with sectional curvatures bounded by $\abs{\seccurv} \leq \curvabsbnd$, and $\acplx$ is an abstract simplicial complex with finite vertex set $\pts \subset \man$.  
% %
% If
%   \begin{enumerate}
%   \item For every simplex $\splxs = \asimplex{p_0, \ldots, p_n} \in \acplx$, the edge lengths $\ell_{ij} = \distM{p_i}{p_j}$ satisfy $\ell_{ij} < \scale$, and they define a Euclidean simplex $\splxsE$ with $\thickness{\splxsE} \geq \thickbnd$.
%   \item The balls $\{\ballM{p}{h}\}_{p \in \pts}$ cover $\man$, and for each $p \in \pts$ the secant map of $\exp_p^{-1}$ realises $\str{p}$ as a full star.
%   \item The scale parameter $\scale$ satisfies
%     \begin{align*}
%       \scale \leq \min \left\{ \frac{\injradM}{4},
%       \frac{\thickbnd}{8 \sqrt{\curvabsbnd}} \right\}.
%     \end{align*}
%   \end{enumerate}
%   then $\acplx$ triangulates $\man$, and the triangulation is given by the barycentric coordinate map on each simplex.
\end{prop}
\begin{proof}
  By the argument in the proof of \Propref{prop:alt.non.degen}, using $\scale$ instead of $2\rho$, we see that for any $x,y \in \ballM{p}{h}$ we have
  \begin{equation*}
    \abs{\norm{\exp_p^{-1}(x) - \exp_p^{-1}(y)} - \distM{x}{y}}
    \leq \curvabsbnd \scale^2 \distM{x}{y}.
  \end{equation*}
  Then using $\metlipconst =  \curvabsbnd \scale^2 = \frac{\eta^2 \thickbnd^2}{4}$ in \Lemref{lem:intrinsic.thick.distortion}, we get
  \begin{equation*}
    \eta^2 = \frac{4\curvabsbnd \scale^2}{\thickbnd^2} \leq \frac{1}{16}.
  \end{equation*}
  It follows that if $\splxs \in \acplx$, with $p \in \splxs$, then
  \begin{equation*}
    \thickness{\splxs(p)} \geq \frac{4}{5\sqrt{n}}(1 - \eta^2)\thickbnd \geq \frac{3}{4\sqrt{n}}\thickbnd.
  \end{equation*}
  The bound on $\scale$ then implies that $\scale \leq \frac{\sqrt{n}\thickness{\splxs(p)}}{6\sqrt{\curvabsbnd}}$, and so the result of \Thmref{thm:triangulation} applies.
\end{proof}

    % -*- LaTeX -*-
% fat.tex
%
% Discussion on fatness.
% Salvaged from Stab2 orphans

\newcommand{\splxvol}[2]{\vol^{#1}(#2)}

\subsection{In terms of fatness}
\label{sec:phat}

Many alternative quality measures for simplices have been employed in the literature. Thickness is employed by Munkres~\cite{munkres1968}, using a slightly different normalisation than ours. It is also very popular to use a volume-based quality measure such as that employed by Whitney~\cite{whitney1957}. In this section we introduce Whitney's quality measure, which we call \defn{fatness}, and we compare it with thickness.

If $\splxs$ is a $j$-simplex, then its \defn{volume}, may be defined for $j>0$ as
\begin{equation*}
  \splxvol{j}{\splxs} = \frac{1}{j!}\prod_{i=1}^{j} \sing{i}{P},
\end{equation*}
where $P$ is the $m \times j$ matrix whose $i^\text{th}$ column is
$p_i - p_0$ for $\splxs = \asimplex{p_0, \cdots, p_j} \subset \rem$.
If $j=0$ we define $\splxvol{0}{\splxs} = 1$.  Alternatively, the
volume may be defined inductively from the formula
\begin{equation}
  \label{eq:splx.vol.alt}
  \splxvol{j}{\splxs} = \frac{\splxalt{p}{\splxs}\splxvol{j-1}{\splxsp}}{j}.
\end{equation}
The \defn{fatness} of a $j$-simplex $\splxs$ is the dimensionless quantity
\begin{equation*}
  \fatness{\splxs} =
  \begin{cases}
    1& \text{if $j=0$} \\
    \frac{\splxvol{j}{\splxs}}{
      \longedge{\splxs}^j}& \text{otherwise.}
  \end{cases}
\end{equation*}
\begin{lem}[Fatness and thickness]
  \label{lem:thick.fat.thick}
  For any $j$-simplex $\splxs$
  \begin{equation*}
    \thickness{\splxs}^j \leq \fatness{\splxs}
    \leq
    \prod_{k=1}^j \thickness{\splxs^k}
    \leq
    \frac{\thickness{\splxs}}{(j-1)!},
  \end{equation*}
  where $\splxs = \splxs^j \supset \splxs^{j-1} \supset \cdots \supset \splxs^1$ is any chain of faces of $\splxs$ such that for each $i<j$, $\splxs^i$ has maximal volume amongst all the facets of $\splxs^{i+1}$.
\end{lem}
\begin{proof}
  It follows directly from the volume formula~\eqref{eq:splx.vol.alt} that if $\splxs^{k-1}$ is a face with maximal volume in $\splxs^k = \{p_k\} \cup {\splxs^{k-1}}$, then $p_k$ is a vertex with minimal altitude in $\splxs^k$. Order the vertices of $\splxs = \asimplex{p_0,\ldots,p_j}$ so that $\splxs^k = \asimplex{p_0, \ldots, p_k}$ for each $k\leq j$. Then, inductively expanding the volume formula~\eqref{eq:splx.vol.alt}, we get
  \begin{equation*}
    \volop{\splxs^j} =  \prod_{k=1}^j \frac{\splxalt{p_k}{\splxs^k}}{k}.
  \end{equation*}
  The inequality $\fatness{\splxs} \leq \prod_{k=1}^j \thickness{\splxs^k}$ then follows from the definitions of thickness and fatness, and the observation that $\longedge{\splxs} \geq \longedge{\splxs^k}$ for all $k\leq j$. Also from the definition of thickness we have the trivial bound $\thickness{\splxs^k} \leq \frac{1}{k}$, from which the rightmost inequality follows.

  The lower bound also follows from induction on \Eqnref{eq:splx.vol.alt}. Using the same chain of faces and vertex labelling we get
  \begin{equation*}
    \begin{split}
      \fatness{\splxs}  &=
    \frac{\splxalt{p_j}{\splxs}}{j\longedge{\splxs}}
    \frac{\volop{\splxs^{j-1}}}{\longedge{\splxs}^{j-1}}\\
    &=
    \thickness{\splxs}\fatness{\splxs^{j-1}}
    \frac{\longedge{\splxs^{j-1}}^{j-1}} {\longedge{\splxs}^{j-1}}\\
    &\geq \thickness{\splxs} \thickness{\splxs^{j-1}}^{j-1}
    \frac{\longedge{\splxs^{j-1}}^{j-1}} {\longedge{\splxs}^{j-1}}
    \qquad \text{inducive hypothesis}\\
    &= \thickness{\splxs} 
    \left( \frac{\splxalt{p_{j-1}}{\splxs^{j-1}}} {(j-1)\longedge{\splxs}}
    \right )^{j-1} \\
    &\geq \thickness{\splxs}^j.
   \end{split}
  \end{equation*}
\end{proof}

Although \Lemref{lem:thick.fat.thick} gives the impression that fatness corresponds roughly to a power of thickness, we observe that thickness and fatness coincide for triangles, as well as edges, and vertices.

\Lemref{lem:thick.fat.thick}, provides a way to express our results in terms of fatness instead of thickness. For example, the quality bound for non-degeneracy in \Thmref{thm:thick.nondegen.riem}
 \begin{equation*}
   \thickness{\splxs(p)} > 10 \sqrt{\curvabsbnd}\rho,
 \end{equation*}
is attained if
 \begin{equation*}
   \fatness{\splxs(p)} > \frac{10 \sqrt{\curvabsbnd}\rho}{(n-1)!}.
 \end{equation*}

% -*- LaTeX -*-
% toponogov.tex
% 20140611 
%
% Driver file for the toponogov appendix

\newcommand{\ud}{\mathrm{d}}
\newcommand{\cref}[1]{(\ref{#1})}
\newcommand{\arcsinh}{\mathrm{arcsinh}}
\newcommand{\distmax}{D}

%%%%%%%%%%%%%%%%%%%%%%%%%%%%%%%%%%%%%%%%%%%

\section{The Toponogov point of view}
\label{sec:toponogov}

\subsection{Introduction}

In this appendix we discuss a different approach to finding conditions that guarantee that a Riemannian simplex is non-degenerate, that is diffeomorphic to the standard simplex. It is based on the Toponogov comparison theorem, instead of Rauch's theorem. This comparison theorem says that if the sectional curvatures of a manifold $M$ are bounded from above by $\curvupbnd$ and below by $\curvlowbnd$ and there is a geodesic triangle in $M$ of which we know the lengths of the three geodesics, then the angles of the triangle are bounded by the angles for a geodesic triangle in a space of constant curvature $\curvlowbnd$ or $\curvupbnd$ whose geodesics have the same lengths. Similarly if we are given lengths of two geodesics and the enclosed angle in a geodesic triangle in $M$, the length of the third geodesic is bounded by the lengths of the third geodesic in a geodesic triangle with the same lengths for two geodesics and enclosed angle in a space of constant curvature $\curvlowbnd$ or $\curvupbnd$. The dimension of the manifold shall be denoted by $n$.

\begin{figure}[!htb]
    \centerline{
\mbox{\includegraphics[height=6.2in]{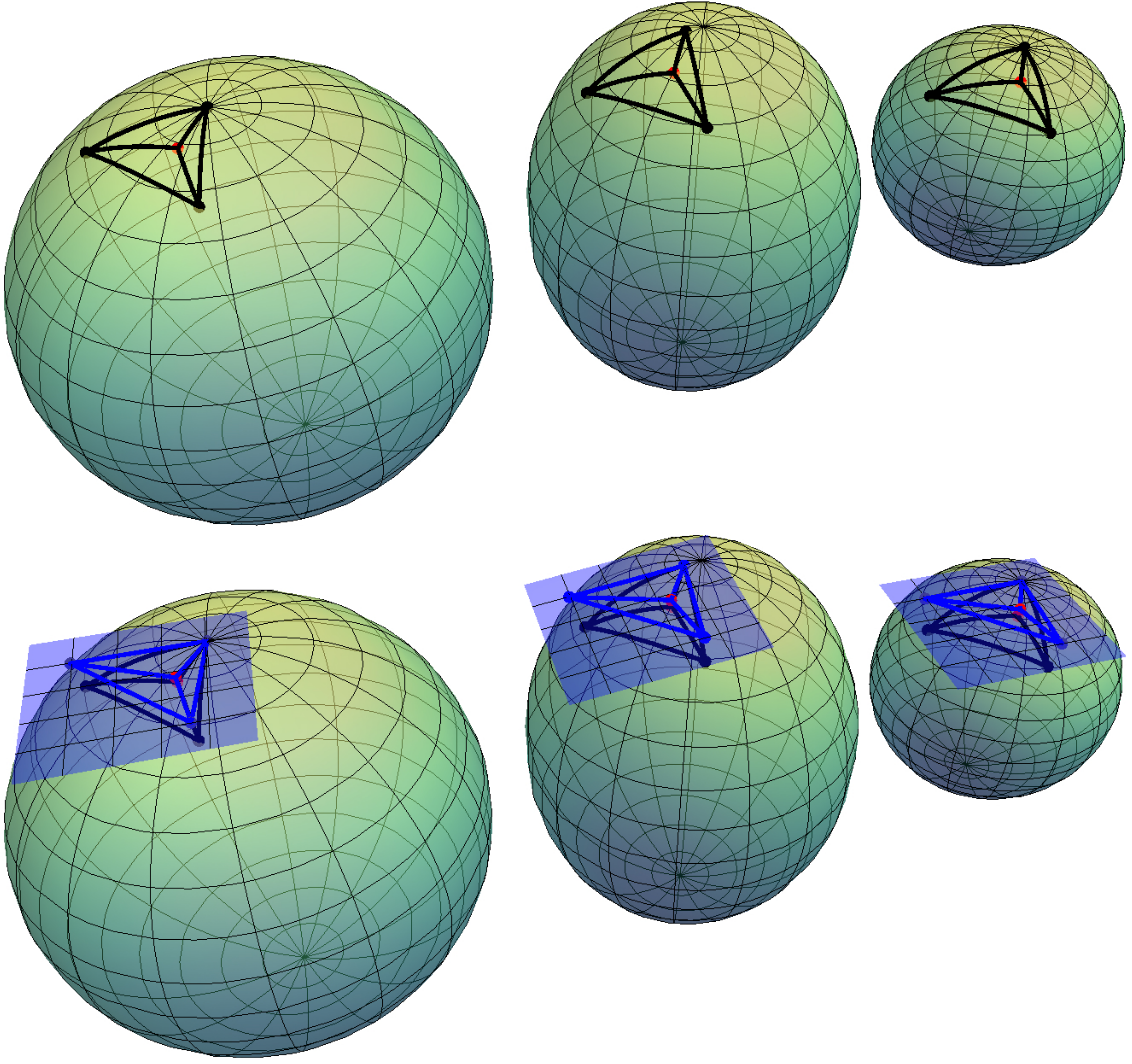}}
    }
\caption{ Pictorial overview of our approach: Given any point (red) and some vertices (black) we can first compare the angles between the geodesics at the red point on a surface of arbitrary bounded curvature (the ellipsoid in the middle) to the angles in the space of constant curvature. In these spaces of constant curvature, study the simplex by lifting to the tangent space at one of the vertices by the exponential map (bottom figure). }
\label{OverviewFig}
\end{figure}

Our non-degeneracy conditions are established in two steps:
First we note that if for any point $x$ in a neighbourhood of the vertices in the manifold there are $n$ tangents to geodesics connecting this point $x$ to some
subset of the vertices 
(the choice of subset does depend on $x$) that are linearly independent, then the Riemannian simplex is non-degenerate. This condition is equivalent with the tangent vectors being affinely independent, because for a point $x$ in a Riemannian simplex with barycentric coordinates $\lambda_i$ we have  
$ \sum \lambda_i v_i(x)=0,$ with $\sum \lambda_i= 1$, 
remember that we write $v_i(x) = \exp^{-1}_x(v_i)$. 
Secondly one has to find conditions on the vertex set in combination with geometric properties of the manifold, to be precise bounds on the sectional curvature, such that we can guarantee linear independence. We do so by looking at the angles between the tangent vectors discussed above. Using the Toponogov comparison theorem we can give estimates on the difference between these angles and the angles one would expect in a space of constant curvature $\curvupbnd$ and $\curvlowbnd$ respectively. These small neighbourhoods in spaces of constant curvature are in turn well approximated by small subsets in Euclidean space. To be precise we compare to the simplex we find by lifting the vertices to the tangent space at one of these vertices via the exponential map. Using these estimates we can prove that if the simplex is small enough compared to the quality of the Euclidean simplex, there is a linearly independent set of tangent vectors so that we have non-degeneracy. 
This approach puts the emphasis on the geodesics as apposed to barycentric coordinate functions, which provides us with a very concrete geometric picture, see for example figure \ref{OverviewFig}. 

The result to which this method leads to the following result:
\begin{dup}[Theorem~\ref{ConditionSimplices}]
Let $v_0, \ldots ,v_n$ be a set of vertices lying in a Riemannian manifold $M$,
whose sectional curvatures are bounded in absolute value by
$\curvabsbnd$, within a convex geodesic ball of radius $\distmax$ centred at one of the vertices ($v_r$) and such that $\sqrt{\curvabsbnd} \distmax <1/2$. If $\sigma^{\mathbb{E}} (v_r )$, the convex hull of $(\exp^{-1}_{v_r}(v_i))_{i=0}^{n}= (v_i(v_r))_{i=0}^{n}$,
satisfies
\begin{align} 
\left(\frac{ (n-1)! \textrm{vol} (\sigma^{\mathbb{E}} (v_r ) ) }{ 
(2\distmax)^n } \right)^2 > 160 n \sqrt{\curvabsbnd} \distmax, \tag{\ref{QualityCrit}}
\end{align}
then the Riemannian simplex with vertices $v_0, \ldots , v_n$ is non-degenerate, that is diffeomorphic to the standard $n$-simplex. 
\end{dup}

\subsection{Preliminaries}
\label{sectionPreliminaries}

The second step described in the introduction will use the
Toponogov Comparison Theorem \ref{LowerBoundTCT}. In particular, we use this result to provide bounds on the angles between the
vectors tangent to geodesics emanating from a point $x \in \sigma_M$
to $n$ (that is all but one) of the vertices of $\sigma_M$. These angle bounds are then
used to show that the reduced Gram
matrix associated with these vectors is
non-singular. Here we discuss the observations relating to (reduced)
Gram matrices and bounds on determinants that we will use.

%\subsubsection{Gram matrices} 
\paragraph{Gram matrices}
\label{sectionGram}

Gram matrices can be applied
to general finite dimensional inner product spaces
($\mathbb{R}^n_G$, if the innerproduct $G=\delta_{ij}$, where $\delta_{ij}$ denotes the Kronecker delta, we shall write $\mathbb{R}^n_{\delta_{ij}} = \mathbb{E}^n$), such as the tangent
spaces ($T_x M$) of Riemannian manifolds with inner product $g(x)$. In this setting we
have
\begin{align} \det ( \langle w_i , w_j \rangle_G ) = \det (G) \det
(w_1 , \ldots , w_n )^2,   \nonumber
\end{align} where $w_1, \ldots, w_n \in \mathbb{R}^n_G$, and $(w_1 , \ldots , w_n )$ denotes the matrix with $w_i$ as columns. One can
already think of the vectors $w_i$ as the tangent vectors $w_i=v_i(x)=\exp_{x}^{-1} (v_i)$ discussed in the introduction. This can be
seen by taking the determinant of
\begin{align} (  \langle w_i, w_j \rangle_G ) =
(w_1, \ldots, w_n )^{t} (G) (w_1, \ldots, w_n ). \nonumber \end{align} This is an expression of the fact that the choise of the metric does not influence linear independence. 
Because we will be interested
in the angles (as these feature in the Toponogov Comparison Theorem), we consider instead the Gram matrix associated with the
normalized (we assume that $w_i \neq 0$) vectors $w_i/|w_i|_G$:
\begin{align}
\det ( \cos \theta_{i,j} ) =&  \det \left( \frac{\langle w_i, w_j \rangle_G}{|w_i|_G |w_j|_G } \right) \nonumber\\
  = & \frac{\det((w_1 ,\ldots, w_n ) (G) (w_1, \ldots, w_n ))}{|w_1|_G^2 |w_2|_G^2 \ldots |w_n|_G^2 }, \label{ToConvexHull}
\end{align}
where $\theta_{i,j}$ denotes the angle between $w_i$ and $w_j$. The determinant
of $( \cos \theta_{i,j} )$ is zero if and only if the determinant of
$(w_1, \ldots ,w_n )$ is zero, i.e., if and only if $w_1, \ldots ,w_n$
are linearly dependent. We shall refer to the matrix of cosines of
angles as the
\defn{reduced Gram matrix}.

\paragraph{Bounds on determinants} \label{SectionFriedland}

The following result by Friedland \cite{Friedland}, see also Bhatia
and Friedland \cite{BhatiaFriedland}, Ipsen and Rehman \cite{Ipsen}
and Bhatia \cite{Bhatia} problem I.1.6, will be essential to some
estimates below:
\begin{align}
|det (A+ E ) -\det (A) | \leq n \max \{ \| A \|_p , \| A+ E \| _p \}
^{n-1} \| E \|_p \label{FriedlandBound}
\end{align}
where $A$ and $E$ are $n \times n$-matrices and $\| \cdot \| _p $ is the
$p$-norm, with $1\leq p\leq \infty$, for linear operators: \begin{align} \| A \|_p = \max_{x \in
\mathbb{R}^n } \frac{ | A x|_p }{|x|_p } , \nonumber \end{align}
with $| \cdot |_p $ the $p$-norm on $\mathbb{R}^n$. In our context
$A$ will be the reduced Gram matrix for the Euclidean case and
$E$ the matrix with the small angle deviations from the Euclidean case (or rather the deviations
of their cosines) due to the local geometry, of which each entry is
bounded by some~$\epsilon$.  \newline From \eqref{FriedlandBound} we
see that in this particular context \begin{align} |\det (A+E)| &\geq
|\det A| - n (\max \{ \| A \|_\infty , \| A+ E \| _\infty \} )^{n-1} \| E
\|_\infty \nonumber \\ &\geq |\det(A)| - n 
\epsilon ,
\label{formulaFriedland}
\end{align}
where we use that every entry of $A$ and $A+E$ is bounded in absolute value by
$1$ because it they are reduced Gram matrices, a matrices of cosines.

\paragraph{Toponogov comparison theorems and spaces of constant curvature}
\label{subsectionToponogov}

We shall now first give the Toponogov Comparison Theorem and the
definitions which go with it. Here we follow Karcher
\cite{Karcher2}. Then we give some results comparing the cosine rules in small neighbourhoods in spaces of constant curvature to those of Euclidean space. As mentioned in the introduction we assume that we
work in neighbourhoods that lie within the convexity radius unless
mentioned otherwise.

We shall use the notation $\mathbb{H}^n(\curvlowbnd)$ for the space simply connected of dimension $n$ with constant sectional curvature $\curvlowbnd$.
\begin{de}
A \emph{geodesic triangle} $T$ in a Riemannian manifold consists of 
three minimizing geodesics connecting three points, sometimes also referred to as vertices. We stress that a geodesic triangle does not include an interior. Assume lower curvature bounds
$\curvlowbnd \leq K$ (or upper bounds $K \leq \curvupbnd$). A triangle with
the same edge lengths as $T$ in $\mathbb{H}^n(\curvlowbnd)$ (or
$\mathbb{H}^n(\curvupbnd)$),
 is called an \emph{Alexandrov triangle} $T_{\curvlowbnd}
$ (or $T_{\curvupbnd}$) associated with T, named after Alexandrov who used these in his study of convex surfaces \cite{Karcher2}. Note that any two choices of $T_{\curvlowbnd}$ (or $T_{\curvupbnd}$) are equivalent due to the constant curvature of the space.
Two edges of a geodesic triangle and
the enclosed angle form a \emph{hinge}; a
\emph{Rauch hinge} in $\mathbb{H}^n(\curvlowbnd)$ (or
$\mathbb{H}^n(\curvupbnd)$) of a given hinge, consists of two geodesics
emanating from a single point with the same lengths and enclosed
angles as the original hinge. The edge closing the Rauch hinge in $\mathbb{H}^n(\curvlowbnd)$ (or
$\mathbb{H}^n(\curvupbnd)$), that is the minimizing geodesic connecting the two endpoint of the geodesics emanating from a single point with the same lengths and enclosed angles as the hinge in a space of arbitrary curvature, will be called the \emph{Rauch edge}. \label{GeodesicT}
\end{de}

\begin{figure}[!htb]
    \centerline{
\mbox{\includegraphics[height=2.2in]{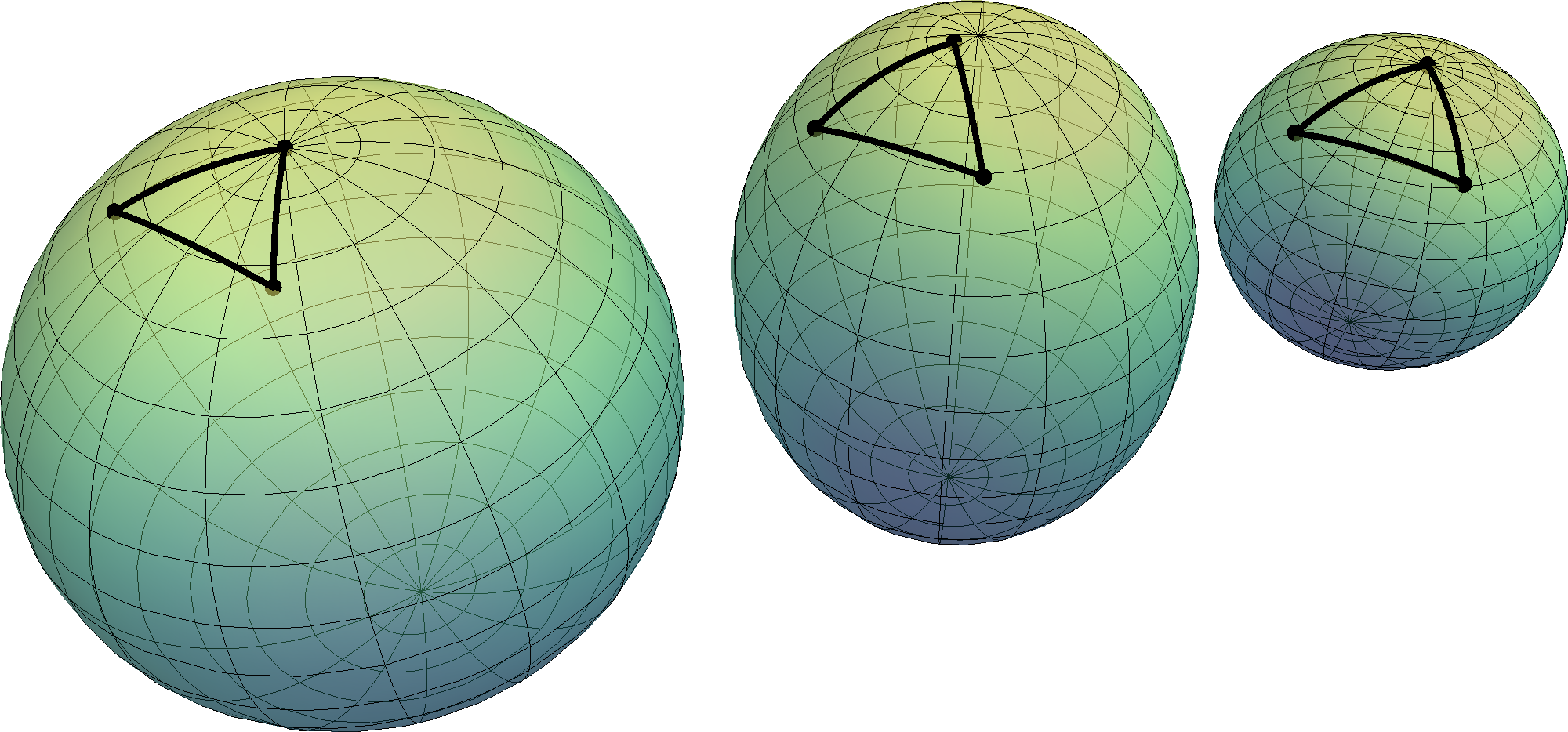}}
    }
\caption{An ellipsoid with a geodesic triangle and the Alexandrov
triangles in the spaces of constant curvatures, in this case both
elliptic spheres.} \label{FigureToponogov1}
\end{figure}

% 0.271892 = 0.3 cos(25 \pi /180),  0.126785= 0.3 sin(25 \pi /180)
% 0.0520945 = 0.3 cos(-80 \pi /180),  -0.295442= 0.3 sin(-80 \pi /180)
% -0.259808 = 0.3 cos(150 \pi /180),  0.15 = 0.3 sin(150 \pi /180)

	% \draw[thin] (A) ++ (0.271892, 0.126785) arc[start angle=25, end angle=75,radius=0.3]; 
	% \draw[thin] (C) ++ (0.0520945,-0.295442) arc[start angle=-80, end angle=-50,radius=0.3]; 
	% \draw[thin] (B) ++ (-0.259808, 0.15) arc[start angle=150, end angle=190,radius=0.3];  

\begin{figure}[!htb]
    \centerline{
\mbox{
\begin{tikzpicture}
  \coordinate [label={below left:$A$}] (A) at (0, 0);
  \coordinate [label={above left:$C$}] (C) at (0.5, 4);
  \coordinate [label={above right:$B$}] (B) at (4, 1.5);
	\draw[thick] (A) to [out=75,in=-80] (C);
 	\draw[thick] (A) to [out=25,in=190] (B);
	\draw[thick] (C) to [out=-50,in=150] (B);
	\draw[thin] (A) ++ (0.271892, 0.126785) arc (25:75:0.3);
	\draw[thin] (C) ++ (0.0520945,-0.295442) arc (-80:-50:0.3);
	\draw[thin] (B) ++ (-0.259808, 0.15) arc (150:190:0.3);
	\node at (0.25,2) {c};
	\node at (2,0.73) {b};
	\node at (2.25,2.75) {a};
  \node at (0.35, 0.35) {$\alpha$};
	\node at (3.4, 1.6) {$\beta$};
	\node at (0.7,3.5) {$\gamma$};
\end{tikzpicture}
}
    }
\caption{Triangle with the standard symbols for angles and lengths}
\label{HingeTriangleAnglePic}
\end{figure}
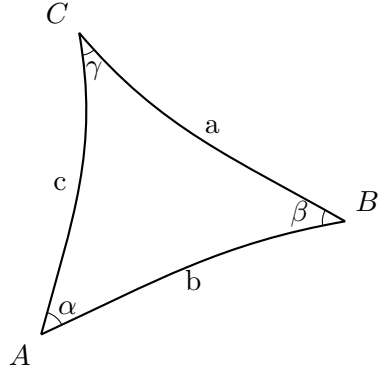

The Toponogov Comparison Theorem or Triangle Comparison Theorem reads
\begin{thm}[Toponogov Comparison Theorem]%[Triangle comparison theorem ] 
Let $T$ be a
geodesic triangle in $M$ and assume that the sectional curvatures $K$ of $M$ satisfy the bounds
 $ \curvlowbnd \leq K \leq \curvupbnd$.
If $\curvupbnd>0$, assume also 
that the triangle perimeter is less then $2 \pi \curvupbnd^{-1/2}$. Then 
Alexandrov triangles  $T_{\curvlowbnd}$ and $T_{\curvupbnd}$ exist. Moreover, any angle $\alpha$ of $T$ satisfies
\begin{equation*}
    \alpha_{\curvlowbnd}  \leq \alpha \leq \alpha_{\curvupbnd},
\end{equation*}
where $\alpha_{\curvlowbnd}$ and $\alpha_{\curvupbnd}$ are the corresponding angles in $T_{\curvlowbnd}$ and $T_{\curvupbnd}$ respectively.
The length $c$ of the third edge closing a hinge is bounded in length by the lengths of the Rauch edges, $c_{\curvlowbnd}$ and $c_{\curvupbnd}$, closing the Rauch hinges:
\begin{align}
  c_{\curvlowbnd} \geq c \geq c_{\curvupbnd} . \nonumber
\end{align}
\label{UpperBoundTCT} \label{LowerBoundTCT}
\end{thm}

We also give the cosine rule which is of use in explicit
calculations involving the Toponogov comparison theorem. The cosine
rule for elliptic spaces, that is positive curved spaces of constant curvature, is
given in Section 18.6 of Berger \cite{BergerGeometryII} and Section
12.7 of Coxeter \cite{Coxeter} and reads
\begin{align} \cos \frac{a}{k} = \cos \frac{b}{k} \cos \frac{c}{k} + \sin \frac{b}{k} \sin \frac{c}{k} \cos \alpha,
\label{CosineRuleElliptic}
\end{align} where we assume that the Gaussian (in two dimensions) or sectional curvature $K$ satisfies
$K= 1/k^2$. The cosine
rule for hyperbolic spaces, that is negatively curved spaces of constant curvature, is given in Section 19.3 of Berger
\cite{BergerGeometryII} and Section 12.9 of Coxeter \cite{Coxeter}
to be  \begin{align} \cosh \frac{a}{k} = \cosh \frac{b}{k} \cosh
\frac{c}{k} - \sinh \frac{b}{k} \sinh \frac{c}{k} \cos \alpha,
\label{CosineRuleHyperbolic}
\end{align} where we assume that the Gaussian or sectional curvature is $K=-1/k^2$.

We now prove two lemmas for geodesic triangles in a space of constant curvature. In the following we use the notation
$(c^{\mathbb{E}})^2= a^2 + b^2 - 2 a b \cos \gamma $, to denote the length of the (Rauch) edge closing a hinge with lengths $a$ and $b$ and enclosed angle $\gamma$ in Euclidean space. In general we shall always use the index $\mathbb{E}$ to indicate (comparisons to) Euclidean space.  
\begin{lem}
If $M$ is a space of constant curvature and two of the edge-lengths ($a$, $b$) of a hinge in $M$ satisfy 
\begin{align}
a, b& \leq d_\textrm{max}/2 
& d_\textrm{max}/k & < 1/2, \nonumber
\end{align}
here $d_\textrm{max}$ is some distance bound on the geodesics in the space of constant curvature, 
then we have that the length of the (Rauch) edge of the hinge $c$, where 
$c^2 = (c^{\mathbb{E}})^2  + E'$, and $|E'| \leq 5 d_\textrm{max}^4 /k^2$ a measure for the deviation from the Euclidean case.  \label{lemSpaceConstantCurvature1}
\end{lem}

\begin{proof} 
Taylor's theorem implies that we have that
\begin{align} \sin(y) & = y (1+ \tilde{E}_s(y)) &
\cos (y) &=1-\frac{1}{2} y^2 + E_c(y) \nonumber \\
\sinh(y) & = y (1+ \tilde{E}_{sh}(y)) & \cosh (y) &=1+ \frac{1}{2}
y^2 + E_{ch}(y), \nonumber
\end{align} here $E$ stands for the error. These errors are bounded; if
we assume that $y\leq \phi_m <1$ we have that \begin{align}
|\tilde{E}_s(y)| &\leq \frac{1}{3!} \phi_m^2  \leq \frac{1}{3!} e
\phi_m^2 &
|E_c (y)| & \leq \frac{1}{4!} \phi_m^4 \leq \frac{1}{4!} e \phi_m^4 \nonumber \\
|\tilde{E}_{sh}(y)| & \leq \frac{1}{3!} e \phi_m^2 & |E_{ch} (y)|
&\leq \frac{1}{4!} e \phi_m^4 ,\nonumber \end{align} where $e$ is
Euler's number. It is convenient to use only the weaker bounds found
from the hyperbolic functions as this affords a universal approach. We therefore drop the subscript and write $E$ and $\tilde{E}$.

We define 
\begin{align} \phi_1& =a/k & \phi_2&=b/k & \phi_3 &=c/k & \phi_m = d_{\textrm{max}}/k . \nonumber
\end{align}
Using these the cosine rules read
\begin{align} \cos   \phi_3 & = \cos \phi_1  \cos  \phi_2 + \sin \phi_1 \sin  \phi_2 \cos \gamma  & & \textrm{(elliptic)} \nonumber \\
\cosh   \phi_3 & = \cosh \phi_1  \cosh  \phi_2 - \sinh \phi_1 \sinh
\phi_2 \cos \gamma  & & \textrm{(hyperbolic)} .\nonumber
\end{align} We shall now bound $ \phi_3$, assuming $\phi_1 , \phi_2$ and $\gamma$ given.
Because $\phi_1, \phi_2 \leq \phi_m /2<1/4$ we have that 
implies that $ \phi_3 \leq  \phi_m$ by the triangle inequality. 
We find
\begin{align}\frac{1}{2}  \phi_3^2 + E ( \phi_3) =
\frac{1}{2} \phi_1^2 + \frac{1}{2}  \phi_2^2 + \frac{1}{4} \phi_1^2
 \phi_2^2 + E ( \phi_1) + E ( \phi_2)+ \frac{1}{2} E ( \phi_1)  \phi_2^2 + \frac{1}{2}
E ( \phi_2) \phi_1^2 \nonumber
\\+  E ( \phi_1) E ( \phi_2)- \phi_1  \phi_2 \cos \gamma -  \phi_1  \phi_2 \cos \gamma
(\tilde{E}( \phi_1)+ \tilde{E}( \phi_2)+ \tilde{E}( \phi_1)
\tilde{E}( \phi_2)) \nonumber
\end{align} and thus \begin{align}  \phi_3^2  = &
 \phi_1^2 +   \phi_2^2 - 2 \phi_1  \phi_2 \cos \gamma + \frac{1}{2}
 \phi_1^2  \phi_2^2 + 2 E ( \phi_1) + 2 E ( \phi_2)+  E ( \phi_1)  \phi_2^2 +  E ( \phi_2) \phi_1^2 \nonumber
\\&+  2 E ( \phi_1) E ( \phi_2) - 2  \phi_1  \phi_2 \cos \gamma
(\tilde{E}( \phi_1)+ \tilde{E}( \phi_2)+ \tilde{E}( \phi_1)
\tilde{E}( \phi_2))- E ( \phi_3) \nonumber \\ =& \phi_1^2 + \phi_2^2
- 2 \phi_1  \phi_2 \cos \gamma +E'( \phi_1, \phi_2, \phi_3)
\label{LengthApproxEucl1}
\end{align} with
\begin{align}| E'( \phi_1, \phi_2, \phi_3) | \leq & | \frac{1}{2}
 \phi_1^2  \phi_2^2 + 2 E ( \phi_1) + 2 E ( \phi_2)+  E ( \phi_1)  \phi_2^2 +  E ( \phi_2) \phi_1^2 \nonumber
\\& +  2 E ( \phi_1) E ( \phi_2) - 2  \phi_1  \phi_2 \cos \gamma
(\tilde{E}( \phi_1)+ \tilde{E}( \phi_2)+ \tilde{E}( \phi_1)
\tilde{E}( \phi_2))- E ( \phi_3)| \nonumber \\  \leq &  |
\frac{1}{2} \phi_1^2  \phi_2^2| + 2 |E ( \phi_1)| + 2 |E ( \phi_2)|+
|E ( \phi_1)|  + | E ( \phi_2) | \nonumber
\\& +  | E ( \phi_1)| +
2 | \phi_1  \phi_2| |\tilde{E}( \phi_1)|+ 2 | \phi_1  \phi_2|
|\tilde{E}( \phi_2)|+ 2| \phi_1  \phi_2| | \tilde{E}( \phi_2)|+| E (
\phi_3)| \nonumber
\\ = & | \frac{1}{2} \phi_1^2  \phi_2^2| + 4 |E ( \phi_1)| + 3 |E ( \phi_2)|  \nonumber
\\& +
2 | \phi_1  \phi_2| |\tilde{E}( \phi_1)|+ 2 | \phi_1  \phi_2|
|\tilde{E}( \phi_2)|+ 2| \phi_1  \phi_2| | \tilde{E}( \phi_2)|+| E ( \phi_3)|\nonumber \\
\leq  & \frac{1}{2} \phi_m^4 +\frac{7}{4! } e \phi_m^4 +
\frac{6}{3!} e \phi_m^4+ \frac{1}{4!} e \phi_m^4 \nonumber \\ \leq &
5 \phi_m^4 . \label{BoundErrorLength1}
\end{align}
\end{proof} 

We study a geodesic triangle in a space of constant curvature for which we have some estimates on the edge-lengths. To be precise we assume that the edges of the geodesic triangle are themselves the closing (Rauch) edges of some hinges. Moreover, we assume that the conditions of the previous lemma are satisfied so that have the estimates on the deviation of the lengths of these closing edges compared to the value expected in Euclidean space. We provide to have similar bounds on the angles of the geodesic triangle for which we have estimates on the edge-lengths, that is we give bounds on the deviation of the angles of the geodesic triangle compared to the value expected in Euclidean space. To this end we define for some lengths $a^{\mathbb{E}}$, $b^{\mathbb{E}}$ and $c^{\mathbb{E}}$ the angle $\alpha^{\mathbb{E}}$ by 
\begin{align} \cos \alpha^{\mathbb{E}} = \frac{(b ^{\mathbb{E}
})^2 + (c^{\mathbb{E}})^2 - (a^{\mathbb{E} })^2 }{ 2
b^{\mathbb{E} } c^{\mathbb{E} }}. \nonumber \end{align} 

\begin{lem} \label{lemSpaceConstantCurvature2}
If $M$ is a space of constant sectional curvature and the edge-lengths ($a$, $b$, $c$) of a geodesic triangle in $M$ satisfy
\footnote{We have included the factor $1/k$ to shorten the calculation below}
\begin{align} (a/k)^2 &= (a^{\mathbb{E}}/ k)^2 + E_1, & (b/k)^2 &= (b^{\mathbb{E}}/k)^2 + E_2,& (c/k)^2 &= (c^{\mathbb{E}}/k)^2 + E_3 , \end{align}
with that
$E_1,E_2,E_3 \leq  5 d_\textrm{max}^4 /k^4$, 
and
\begin{align} (d_\textrm{max} /k)^{3/2} <a^{\mathbb{E}}/k, b^{\mathbb{E}}/k, c^{\mathbb{E}}/k
< d_\textrm{max} /k < 1/2,  
\label{LowerBoundLengths}
\end{align} 
then 
\begin{align} | \cos \alpha- \cos \alpha^{\mathbb{E}} |  \leq  80
d_\textrm{max}/k. \label{ResultAnglesAppendix} \end{align}
\end{lem}

\begin{proof}
To avoid having to drag along the $1/k$ we shall write $\psi_1= a/k$,  $\psi_1^{\mathbb{E}}= a^{\mathbb{E}}/k$\, et cetera. 
Using the previous lemma we see that 
$\psi_1$, $\psi_2$, $\psi_3$ satisfy \begin{align} (\psi_1) ^2 =
(\psi_2) ^2+(\psi_3) ^2- \psi_2 \psi_3 \cos \alpha +E_ 4 , \nonumber
\end{align} PErforming a calculation similar to the one in the previous lemma we find that
\begin{align}
| \cos \alpha - \cos \alpha^{\mathbb{E}} | &= \left | \frac{
(\psi^{\mathbb{E}}_2)^2+ (\psi^{\mathbb{E}}_3)^2 -
(\psi^{\mathbb{E}}_1)^2 + E_2+E_3 +E_4 -E_1} {2 \psi^{\mathbb{E}}_2
\psi^{\mathbb{E}}_3 \sqrt{1 + \frac{E_2}{ (\psi^{\mathbb{E}}_2)^2}
}\sqrt{1 + \frac{E_3}{ (\psi^{\mathbb{E}}_3)^2} } } - \frac{
(\psi^{\mathbb{E}}_2)^2+ (\psi^{\mathbb{E}}_3)^2 -
(\psi^{\mathbb{E}}_1)^2 } {2 \psi^{\mathbb{E}}_2 \psi^{\mathbb{E}}_3
 } \right| \nonumber \\ & \leq \left| \frac{
(\psi^{\mathbb{E}}_2)^2+ (\psi^{\mathbb{E}}_3)^2 -
(\psi^{\mathbb{E}}_1)^2 } {2 \psi^{\mathbb{E}}_2 \psi^{\mathbb{E}}_3
 }  \frac{4 E_2}{ (\psi_2^{ \mathbb{E}})^2} \right | + \left| \frac{
(\psi^{\mathbb{E}}_2)^2+ (\psi^{\mathbb{E}}_3)^2 -
(\psi^{\mathbb{E}}_1)^2 } {2 \psi^{\mathbb{E}}_2 \psi^{\mathbb{E}}_3
 }  \frac{4 E_3}{ (\psi_3^{ \mathbb{E}})^2} \right | \nonumber \\ & \phantom{\leq} + \left | \frac{
 E_2+E_3 +E_4 -E_1} {2 \psi^{\mathbb{E}}_2 \psi^{\mathbb{E}}_3 } \right| \left| \left (1 + \frac{2 E_2}{
 (\psi^{\mathbb{E}}_2)^2} \right )
 \left (1 + \frac{E_3}{ (\psi^{\mathbb{E}}_3)^2} \right) \right|  \nonumber \\
& \leq \frac{ 4 E_2}{(\psi_2^{\mathbb{E}})^2} + \frac{ 4
E_3}{(\psi_3^{\mathbb{E}})^2} + 40 \phi_m  \nonumber \\ & \leq  80
\phi_m . \nonumber \end{align}
\end{proof}

\subsection{Relation with linear independence}
\label{sec:rel.w.aff.indep}

In Euclidean simplex is non-degenerate if and only if for any point $x$ in Euclidean space we can find $n$ vertices
such that the vectors from $x$ to the vertices are linearly independent. Linear independence likewise plays an
important role in the definition of a non-degenerate Riemannian
simplex. We remind ourselves, see Section \ref{sec:bary.map}, that a Riemannian simplex $\sigma_M$ is
\defn{non-degenerate} if the barycentric coordinate map
$\stdsplxn \to \sigma_M$ is a smooth embedding.

\begin{lem}
\label{ConsequencesLinearIndependence} If for any  $x$ in the image
of the map given in Definition \ref{def:riem.splx} ($\sigma_M$)
there are $n$ tangents to geodesics connecting this point $x$ to some
subset of the vertices $v_0 , \ldots , v_{j-1}, v_{j+1}, \ldots, v_n$
(this choice does depend on $x$) that are linearly independent then
\begin{itemize}
\item[$\bullet$] The map $\stdsplxn \to  \sigma_{M}$ is
bijective
\item[$\bullet$] The inverse of $\stdsplxn \to
\sigma_{M}$ is smooth \end{itemize}
\end{lem}

In the proof we shall need the following observation: Within any
ball smaller than the injectivity radius containing $v_i$, the vector field $v_i(x) =\exp_{x} ^{-1}(v_i)$
depends smoothly on the point $x$ for all $x \neq v_i$.  This is obvious if we
consider Riemannian normal coordinates at $v_i$. The geodesic
between $x$ and origin ($v_i$) is a straight line, that depends
smoothly on $x$. The same holds for the tangent to the geodesic at $x$, this is precisely
$v_i(x)=\exp_{x} ^{-1}(v_i)$.

\begin{proof}
We now prove the first of our claims by contradiction. Let us assume
that \begin{align} \sum \lambda_i v_i(x) =  \sum
\tilde{\lambda}_i v_i(x)  =0 \nonumber  \end{align} for
some $\lambda, \tilde{\lambda} \in \stdsplxn $, $\lambda
\neq \tilde{\lambda} $.  Because $v_0(x), \ldots
,v_{j-1}(x) ,v_{j+1} (x), \ldots , 
v_n(x) $ are assumed to be linearly independent we have $\lambda_j
\neq 0, \tilde{\lambda}_j \neq 0$. This mean that we can solve for
$v_j$ in both cases, so
\begin{align} \frac{ \lambda_0}{\lambda_j} v_0(x)+ \ldots+\frac{
\lambda_{j-1}}{\lambda_j} v_{j-1}(x)+
\frac{\lambda_{j+1}}{\lambda_j} v_{j+1}(x)+\ldots  + \frac{
\lambda_{n}}{\lambda_j} v_{n}(x)  = \nonumber \\
\frac{ \tilde{\lambda} _0}{\tilde{\lambda}_j} v_0(x)+ \ldots+\frac{
\tilde{\lambda}_{j-1}}{\tilde{\lambda}_j} v_{j-1}(x) +
\frac{\tilde{\lambda}_{j+1}}{\tilde{\lambda}_j} v_{j+1}(x) +\ldots  +
\frac{ \tilde{\lambda}_{n}}{\tilde{\lambda}_j} v_{n} (x) . \nonumber
\end{align}
This contradicts the assumption of linear independence. This
establishes injectivity.

We can use a similar argument to show that the inverse of
$\stdsplxn \to  \sigma_{M}$ is smooth. As we have seen
linear independence implies that $\lambda_j \neq 0$, which means
that we have
\begin{align} \lambda_0 v_0(x)+ \ldots +
\lambda_{j-1} v_{j-1}(x) + \lambda_{j+1} v_{j+1}(x) + \ldots + \lambda_{n} v_{n}(x) & =-\lambda_j
v_{j}(x). \nonumber
\end{align} We can now regard the left hand side as the product of
the matrix with columns $ (v_{i}(x))_{i \neq j}$ with the
vector $(\lambda_i)_{i\neq j}$. We can divide by $-\lambda_j$ and
bring the matrix to the right hand side by inverting, because
$\{v_{i}(x)\}_{i \neq j}$  is a linear independent set
this is possible. We now find \begin{align} (v_0(x)  ,
\ldots ,v_{j-1}(x)  ,v_{j+1}(x)  ,\ldots
v_{n}(x) ) ^{-1}& v_{j}(x)   =
\frac{1}{\lambda_j} (\lambda_0 , \ldots
,\lambda_{j-1},\lambda_{j+1} ,\ldots , \lambda_n)^{t} , \nonumber
\end{align} which is smooth because $v_{i}(x)$ is smooth
and $\{v_{i}(x)\}_{i \neq j}$ are linear independent by
assumption.
\end{proof}

In Lemma \ref{ConsequencesLinearIndependence} we refer to points
lying in $\sigma_M$, because $\sigma_M$ is not so easy to determine
a priori, we will need to determine a neighbourhood that contains
$\sigma_M$ where we can determine linear independence. To this end
we observe the following:
\begin{remark}
\label{RemarkNeighbourhood} $\sigma_M$ lies within a ball 
centred at any of the vertices $v_r$ of radius $\distmax$, where $\distmax= \max d_M(v_i,v_r)$, provided $\distmax$ is smaller than the injectivity radius and the ball is convex.
\end{remark}
Karcher \cite{Karcher} noted that the centre of mass of any mass distribution is contained in any convex set that contains the support of the mass distribution, so in particular this ball.

\subsection{Determining linear independence}
\label{SubsectionLinearIndependence}

In the previous subsection we established that if for any point $x
\in \sigma_M$ there are $n$ tangents to geodesics connecting this
point to some subset of the vertices $v_0 , \ldots , v_{j-1},
v_{j+1}, \ldots, v_n$ (depending on $x$) are linearly independent,
then the simplex $\sigma_M$ is well defined. In this subsection we shall
formulate conditions on the vertex set $v_0 , \ldots,  v_n$ such
that we can guarantee linear independence. These conditions are
simple for surfaces. For higher dimensional manifolds we shall need
bounds on the quality of the simplex found by taking the convex hull
of the image of the inverse exponential map at one of the vertices. The
quality of the simplex is considered good if the ratio between the
volume of the simplex and the $n^{\textrm{th}}$ power of the largest edge length is large, which we shall make precise in Theorem \ref{ConditionSimplices}, see also \cite{whitney1957, boissonnat2013manmesh.inria}.

As mentioned, linear dependence for surfaces is easy to determine.
Let us suppose $v_0(x)=\exp_{x}^{-1}(v_0)$, $v_1(x)$,
$ v_2(x) $ do not span $T_x M$. Because $M$ is a surface
it follows that $v_0(x)$, $v_1(x)$,
$v_2(x) $ are co-linear. This is in turn equivalent to
$v_0,v_1$ and $v_2$ lying on a geodesic. Using lemma
\ref{ConsequencesLinearIndependence} we find that $\sigma_M$ is
diffeomorphic to the standard simplex if all three vertices do not
lie on a geodesic. In the two dimensional setting bijection has been argued previously
by Rustamov \cite{Rustamov}.

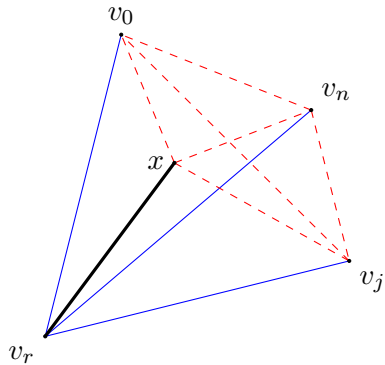
\begin{figure}[!htb]
    \centerline{
\mbox{
\begin{tikzpicture}
  \coordinate [label={below left:$v_r$}] (A) at (0, 0);
  \coordinate [label={above:$v_0$}] (B) at (1, 4);
  \coordinate [label={below right:$v_{j}$}] (C) at (4, 1);
  \coordinate [label={above right:$v_n$}] (D) at (3.5, 3);
  \coordinate [label={left:$x$}] (E) at (1.7, 2.3);
  \draw [color=blue] (A) -- (B);
  \draw [color=blue] (A) -- (C);
  \draw [color=blue] (A) -- (D);
 	\draw [very thick] (A) -- (E);
	\draw [color=red,dashed] (E) -- (B);
  \draw [color=red,dashed] (E) -- (C);
  \draw [color=red,dashed] (E) -- (D);
  \draw [color=red,dashed] (C) -- (B);
  \draw [color=red,dashed] (D) -- (B);
  \draw [color=red,dashed] (D) -- (C);
	\draw[fill] (A) circle (0.02);
	\draw[fill] (B) circle (0.02);
	\draw[fill] (C) circle (0.02);
	\draw[fill] (D) circle (0.02);
	\draw[fill] (E) circle (0.02);
\end{tikzpicture}
}
    }
\caption{A schematic depiction of $\sigma^{\mathbb{E}} (v_r )$, where we use red dotted
lines to indicate that these lengths of these edges are not equal to
the lengths of the corresponding edges in $\sigma_M$.  }
\label{EuclideanSimplex}
\end{figure}

Returning to manifolds of arbitraty dimension, we discuss conditions such that for any point $x$ in a ball of
radius $\distmax$ centred at the vertex $v_r$, the vectors $v_i(x)$
in the tangent space at $x$ form an affinely independent set. Because of Remark \ref{RemarkNeighbourhood} this is the neighbourhood of interest, because it suffices to show independence here. Assume that the sectional curvatures $K$ of $M$ are bounded in
absolute value: $|K| \leq \curvabsbnd$. 
Define $\sigma^{\mathbb{E}} (v_r )$ to be the convex hull of
$(v_i(v_r))_{i=0}^{n}$ in $T_{v_r}M$.  It will be on $\sigma^{\mathbb{E}}(v_r)$ that we impose condition to ensure that the Riemannian simplex $\sigma_M$ in non-degenerate. Note that given $\sigma^{\mathbb{E}}(v_r)$ we in particular have the lengths of all geodesics from $v_r$ to $v_i$ and the angles between their tangents. Using the Toponogov
comparison theorem, we bound $d_\man(x,v_i)$ for each $i$ by means of
Rauch hinges in $\mathbb{H}^n(\curvabsbnd)$ and
$\mathbb{H}^n(-\curvabsbnd)$ the lengths of the closing edges of the hinges 
are denoted by
$d_{ \mathbb{H}^n(\curvabsbnd) } (x,v_i) $ and $d_{ \mathbb{H}^n(-\curvabsbnd) } (x,v_i) $. Lemma \ref{lemSpaceConstantCurvature1} implies that 
\begin{align}  ( d_{ \mathbb{H}^n(\pm \curvabsbnd) }  (x,v_i))^2 & = | x(v_r) - v_i(v_r) |^2 + E_{(x,v_i)
, \pm \curvabsbnd}  , \nonumber
\end{align}
with $\exp_{v_r}^{-1}(x)= x(v_r)$ as usual, $E_{(x,v_i) , \pm \curvabsbnd}$ an error term satisfying the
bound $|E_{(x,v_i) , \pm \curvabsbnd}| < 5 \curvabsbnd  (2 \distmax
)^4 $, provided 
\begin{align}
|v_i(v_r) | , |x(v_r) |  & \leq \distmax & \text{ and } 
& & \sqrt{\curvabsbnd}  \distmax  < \frac{1}{2}. 
\nonumber
\end{align} Here the radius of the geodesic ball $\distmax$ is the maximum distance $d_{\textrm{max}}$ in the spaces of constant curvature $\mathbb{H}^n(\pm \curvabsbnd)$, introduced in Lemmas \ref{lemSpaceConstantCurvature1} and \ref{lemSpaceConstantCurvature2}. 
Because $|E_{(x,v_i) , \pm \curvabsbnd}| < 5 \curvabsbnd  (2 \distmax
)^4 $ we conclude that 
\begin{align}
 d_\man(x,v_i)^2
 = | x(v_r) - v_i(v_r) |^2
+ E_{(x,v_i)}. \nonumber \\
d_\man(v_l,v_k)^2= | v_l(v_r) - v_k(v_r) |^2 + E_{(v_l,v_k)} ,
\label{distanceBounds1}
\end{align}
with $|E_{(v_l,v_k)}| ,|E_{(0,v_i)}|< 5 \curvabsbnd ( 2 \distmax )^4 $.

\begin{figure}[!htb]
    \centerline{
\mbox{\includegraphics[height=2.2in]{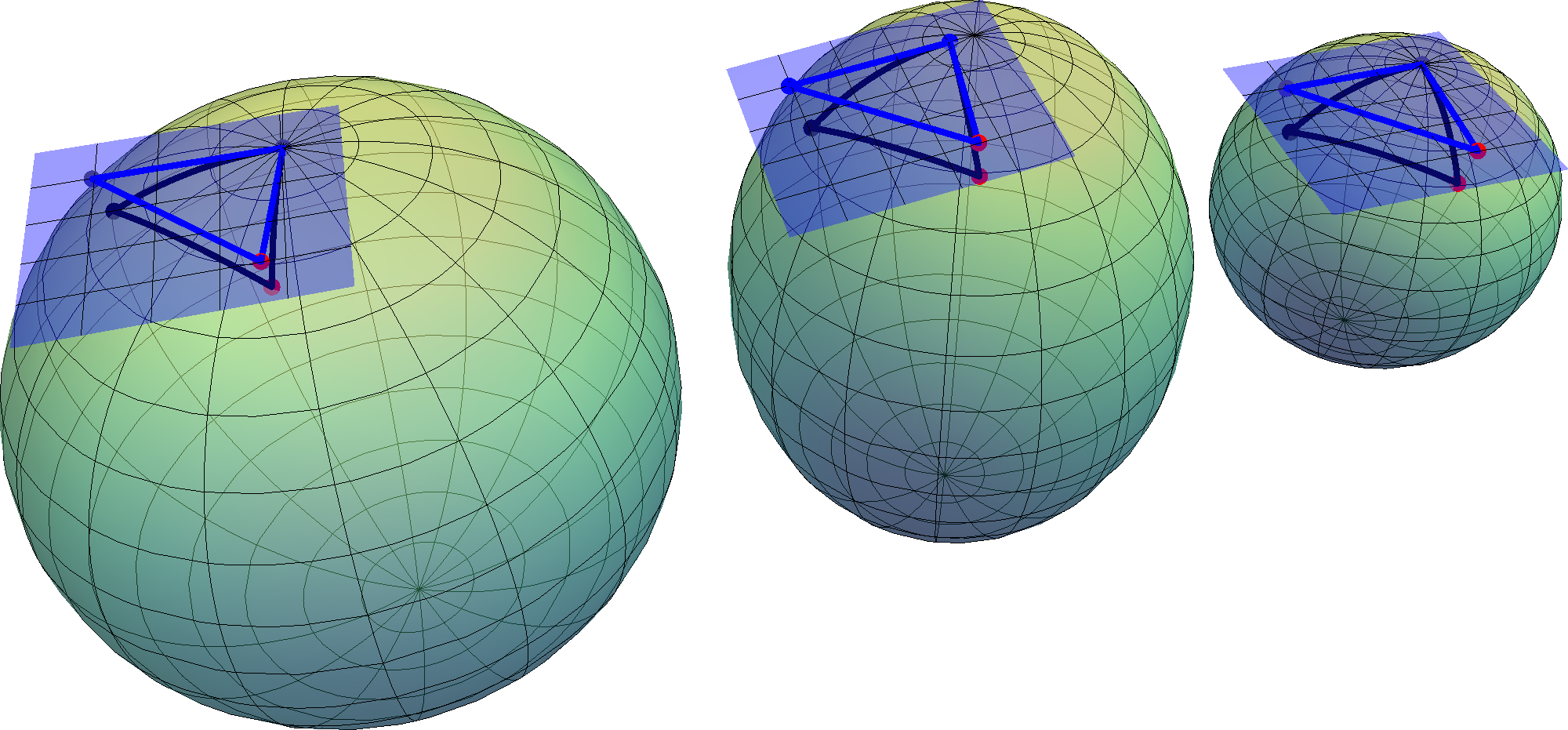}}
    }
\caption{A symbolic sketch of the procedure: the lengths of edges
and angles between geodesics in a manifold of arbitrary curvature
(symbolized by the ellipsiod in the centre) are approximated by those
in the spaces of constant curvature (the two spheres). Here in turn
the triangles are approximated by the Euclidean simplex in `the
tangent space'.  } \label{Toponogov2a}
\end{figure}

At this point we know all the lengths of the geodesics between the
points $x, v_0,\ldots, v_n$ in the manifold up to a small and
explicit deviation term, where the deviation is from the Euclidean space or $T_{v_r}M$ in which $x(v_r)$ and $\sigma^{\mathbb{E}}(v_r)$ lie. Any three points from the set $\{x, v_0,\ldots, v_n\}$ together with the
geodesics connecting them can be regarded as a geodesic triangle.
For a geodesic triangle of which we know all edge lengths the
Toponogov comparison theorem gives bounds on the angles in terms of
the Alexandrov triangles in the spaces $\mathbb{H}^n (\curvabsbnd)$
and $\mathbb{H}^n (-\curvabsbnd)$. Let us denote by $\theta_{il}^{\mathbb{H}^n}$ the angle
$\angle v_i x v_l$ between the geodesics in $\mathbb{H}^n(\pm
\curvabsbnd) $ and let $\theta_{il}^{\mathbb{E}}$ denote the angle
$\angle v_i(v_r) x(v_r) v_l(v_r) $ in $T_{v_r}M$, which we may regard as
Euclidean space. If we have a lower bound on the geodesic edge lengths in the simplex as well as on the distance between $x$ and the vertices under consideration, Lemma \ref{lemSpaceConstantCurvature1} in turn gives us bounds on the angles in $\mathbb{H}^n (\curvabsbnd)$
and $\mathbb{H}^n (-\curvabsbnd)$ compared to the corresponding angle in Euclidean space. 
To be precise 
\begin{align} 
d_{\mathbb{E}} (p,q) > \curvabsbnd^{1/4} (d_{\textrm{max}})^{3/2}, \label{LowerBoundDist2}
\end{align}
with $p,q \in \{ x(v_r), v_i(v_r),v_l (v_r) \mid i \neq j , l \neq j \}$, $p \neq q$, then the distance bounds \eqref{distanceBounds1} in $\mathbb{H}^d (\curvabsbnd)$ imply
\begin{align}
  |\cos \theta_{il}^{\mathbb{H}^n} - \cos \theta_{il}^{\mathbb{E}} |
  \leq 80  \sqrt{\curvabsbnd} d_{\textrm{max}} .  \label{angleBoundsS}
\end{align}

Formula \eqref{angleBoundsS} holds for the
upper and lower bounds that appear in the Toponogov comparison
theorem. Thus
\begin{align} 
  |\cos \theta_{il} - \cos
  \theta_{il}^{\mathbb{E}} | \leq 160  \sqrt{\curvabsbnd} \distmax  ,
\nonumber  %\label{angleBoundsM}
\end{align} 
where $\theta_{il}$ denotes the angle
$\angle v_i x v_l$ between the geodesics in the manifold, assuming
that the conditions above are satisfied. A sufficient condition on the simplex $\sigma^{\mathbb{E}}(v_r)$ for \eqref{LowerBoundDist2} to be satisfied (for some choice of $j$) is that for all $j$ the altitude $\textrm{Alt}_j$
\begin{align} 
\textrm{Alt}_j (\sigma^{\mathbb{E}}(v_r))> n  \curvabsbnd^{1/4} (\distmax )^{3/2} , \nonumber 
\end{align}
which can be weakened\footnote{by which we mean that following inequality implies the previous. } to
\begin{align} 
\frac{\textrm{Alt}_j (\sigma^{\mathbb{E}}(v_r))}{L(\sigma^{\mathbb{E}}(v_r) )} > n  \curvabsbnd^{1/4} ( \distmax )^{1/2}, \label{altcondition}
\end{align}
with $L(\sigma^{\mathbb{E}}(v_r) )$ the longest edge length of $\sigma^{\mathbb{E}}(v_r) $.   

Using reduced Gram matrices and the estimates by Friedland we now see:
\begin{align}  
|\det (\cos \theta_{il} )_j| \geq |\det (\cos \theta_{il}^{\mathbb{E}} )_j| - 160 n   \sqrt{\curvabsbnd}  \distmax ,
\nonumber
\end{align} 
with $(\cos \theta_{il} )_j$ and $(\cos \theta_{il}^{\mathbb{E}} )_j$ the matrix cosines of angles between the tangents of geodesics emanating from $x$ to $v_0, \ldots, v_{j-1}, v_{j+1}, \ldots ,v_n$ and corresponding cosines for $\sigma^{\mathbb{E}}(v_r)$, 
which is equivalent to, using \eqref{ToConvexHull},
\begin{align}  
|\det (\cos \theta_{il} )_j| & \geq \frac{\det(v_0(v_r)-x(v_r), \ldots ,v_{j-1}(v_r)-x(v_r), v_{j+1}(v_r)-x(v_r), \ldots, v_{n}(v_r)-x(v_r)  )^2 }{ |v_0(v_r)-x(v_r)|^2 \cdot  \ldots  \cdot |v_{j-1}(v_r)-x(v_r)|^2 \cdot | v_{j+1}(v_r)-x(v_r)|^2 \cdot  \ldots \cdot |v_{n}(v_r)-x(v_r) |^2 }   \nonumber\\ & -160 n    \sqrt{\curvabsbnd} \distmax , \nonumber
\end{align}
Lemma \ref{ConsequencesLinearIndependence} states that we have non-degeneracy of the simplex if for any $x$ in $B(v_r, \distmax) $ we have that $|\det (\cos \theta_{il} )_j| >0$ for some $j$, this means that if 
\begin{align}
\min_{x \in B(v_r,  \distmax )  } \max_{j \in \{ 0 , \ldots ,n\} }
\frac{\det(v_0(v_r)-x(v_r), \ldots ,v_{j-1}(v_r)-x(v_r), v_{j+1}(v_r)-x(v_r), \ldots, v_{n}(v_r)-x(v_r)  )^2 }{ |v_0(v_r)-x(v_r)|^2 \cdot  \ldots  \cdot |v_{j-1}(v_r)-x(v_r)|^2 \cdot | v_{j+1}(v_r)-x(v_r)|^2 \cdot  \ldots\cdot |v_{n}(v_r)-x(v_r) |^2 }   \nonumber\\  >160 n   \sqrt{\curvabsbnd}  \distmax  , \nonumber
\end{align}
non-degeneracy is established. This can be simplified using that $| v_{j+1}(v_r)-x(v_r)| \leq 2 \distmax$ and remarking that the mimimum of 
\begin{align}
\max_{j \in \{ 0 , \ldots ,n\} }
\det(v_0(v_r)-x(v_r), \ldots ,v_{j-1}(v_r)-x(v_r), v_{j+1}(v_r)-x(v_r), \ldots, v_{n}(v_r)-x(v_r)  )
\nonumber
\end{align} 
is attained in the barycenter and equals $(n-1 )! \textrm{vol} (\sigma^{\mathbb{E}}(v_r) )$. 

This means that we now have the condition for non-degeneracy
\begin{align} 
  \left(\frac{ (n-1)! \textrm{vol} (\sigma^{\mathbb{E}} (v_r ) ) }{ 
(2\distmax)^n } \right)^2  > 160 n \sqrt{\curvabsbnd} \distmax, \tag{\ref{QualityCrit}}
\end{align}
assuming that also \eqref{altcondition} is satisfied. Using lemma \ref{prop:alt.non.degen} we can prove that Equation \eqref{QualityCrit} cannot be
satisfied when \eqref{altcondition} is violated. Firstly note that 
\begin{align} 
\frac{\textrm{vol} (\sigma^{\mathbb{E}} (v_r ) ) }{ (2\distmax)^n } 
\leq  
\frac{\textrm{vol} (\sigma^{\mathbb{E}} (v_r ) ) }{ (L(\sigma^{\mathbb{E}} (v_r ) ))^n }
= \Theta (\sigma^{\mathbb{E}} (v_r ) ) 
\leq \frac{t}{(n-1)!}
= \frac{\min_j \textrm{Alt}_j (\sigma^{\mathbb{E}}(v_r))}{ n! L(\sigma^{\mathbb{E}}(v_r) )}, \nonumber 
\end{align} 
so that \eqref{QualityCrit} yields
\begin{align} 
\frac{\min_j \textrm{Alt}_j (\sigma^{\mathbb{E}}(v_r))}{  L(\sigma^{\mathbb{E}}(v_r) )}>  \sqrt{160} n^{3/2} \curvabsbnd^{1/4} \distmax^{1/2}, \nonumber 
\end{align} which implies \eqref{altcondition}.

We can now summarize 
\begin{thm} \label{ConditionSimplices}
Let $v_0, \ldots ,v_n$ be a set of vertices lying in a Riemannian manifold $M$,
whose sectional curvatures are bounded in absolute value by
$\curvabsbnd$, within a convex geodesic ball of radius $\distmax$ centred at one of the vertices ($v_r$) and such that $\sqrt{\curvabsbnd} \distmax <1/2$. If $\sigma^{\mathbb{E}} (v_r )$, the convex hull of $(\exp^{-1}_{v_r}(v_i))_{i=0}^{n}= (v_i(v_r))_{i=0}^{n}$,
satisfies
\begin{align} 
\left(\frac{ (n-1)! \textrm{vol} (\sigma^{\mathbb{E}} (v_r ) ) }{ 
(2\distmax)^n } \right)^2  > 160 n \sqrt{\curvabsbnd} \distmax, \label{QualityCrit}
\end{align}
then the Riemannian simplex with vertices $v_0, \ldots , v_n$ is non-degenerate, that is diffeomorphic to the standard $n$-simplex. 
\end{thm}

\begin{remark}
Equation \eqref{QualityCrit} can be weakened to 
\begin{align} \left(\frac{ (n-1)! \Theta (\sigma^{\mathbb{E}} (v_r ) ) }{ 
2  ^n } \right)^2  =
\left(\frac{ (n-1)! \textrm{vol} (\sigma^{\mathbb{E}} (v_r ) ) }{ 
(2 L( \sigma^{\mathbb{E}} (v_r ) ) )^n } \right)^2  > 160 n \sqrt{\curvabsbnd} \distmax, \nonumber 
\end{align}
or 
\begin{align}
\left(\frac{ (n-1)! \textrm{vol} (\sigma^{\mathbb{E}} (v_r ) ) }{ 
(2 L( \sigma^{M} ) )^n } \right)^2  > 160 n \sqrt{\curvabsbnd} \distmax, \nonumber
\end{align}
with $L( \sigma^{\mathbb{E}} (v_r ) )=\max_{i,j} d_{\mathbb{E} } (v_i (v_r) , v_j(v_r) )$, $L( \sigma^{M} )=\max_{ij} d_M (v_i,v_j)$ and $\Theta$ the fatness. 
\end{remark}

\paragraph{Acknowledgements}

We thank Stefan von Deylen for pointing out the work of Sander~\cite{Sander1}, and for stimulating discussions. We have also benefited from discussions with Arijit Ghosh.

This research has been partially supported by the 7th Framework Programme for Research of the European Commission, under FET-Open grant number 255827 (CGL Computational Geometry Learning).

% Normally AFTER appendix
\phantomsection
\bibliographystyle{alpha}
\addcontentsline{toc}{section}{Bibliography}
\bibliography{geomrefs}

\begin{thebibliography}{BDG13b}

\bibitem[BDG13a]{boissonnat2013manmesh.inria}
J.-D. Boissonnat, R.~Dyer, and A.~Ghosh.
\newblock {D}elaunay triangulation of manifolds.
\newblock Research Report RR-8389, INRIA, 2013.
\newblock (also: arXiv:1311.0117).

\bibitem[BDG13b]{boissonnat2013stab1}
J.-D. Boissonnat, R.~Dyer, and A.~Ghosh.
\newblock The stability of {D}elaunay triangulations.
\newblock Research Report RR-8276, INRIA, 2013.
\newblock (to appear in Int. J. Comp. Geom. \& Appl. special issue for SoCG
  2012).

\bibitem[Ber87a]{BergerGeometryI}
M.~Berger.
\newblock {\em Geometry {I}}.
\newblock Universitext. Springer-Verlag, 1987.

\bibitem[Ber87b]{BergerGeometryII}
M.~Berger.
\newblock {\em Geometry {II}}.
\newblock Universitext. Springer-Verlag, 1987.

\bibitem[Ber03]{Berger}
M.~Berger.
\newblock {\em A Panoramic View of Riemannian Geometry}.
\newblock Springer-Verlag, 2003.

\bibitem[BF81]{BhatiaFriedland}
R.~Bhatia and S.~Friedland.
\newblock Variation of {G}rassman powers and spectra.
\newblock {\em Linear Algebra and Applications}, 40:1--18, 1981.

\bibitem[BG14]{boissonnat2014tancplx.dcg}
J.-D. Boissonnat and A.~Ghosh.
\newblock Manifold reconstruction using tangential {D}elaunay complexes.
\newblock {\em Discrete and Computational Geometry}, 51(1):221--267, 2014.

\bibitem[Bha97]{Bhatia}
R.~Bhatia.
\newblock {\em Matrix Analysis}.
\newblock Number 169 in Graduate Texts in Mathematics. Springer-Verlag, 1997.

\bibitem[BK81]{buser1981}
P.~Buser and H.~Karcher.
\newblock {\em {G}romov's almost flat manifolds}, volume~81 of {\em
  Ast\'erique}.
\newblock Soci{\'e}t{\'e} math{\'e}matique de France, 1981.

\bibitem[BO05]{boissonnat2005gm}
J.-D. Boissonnat and S.~Oudot.
\newblock Provably good sampling and meshing of surfaces.
\newblock {\em Graphical Models}, 67(5):405--451, 2005.

\bibitem[Cai34]{cairns1934}
S.~S. Cairns.
\newblock On the triangulation of regular loci.
\newblock {\em Annals of Mathematics. Second Series}, 35(3):579--587, 1934.

\bibitem[Car29]{Cartan}
E.~Cartan.
\newblock Groupes simples clos et ouverts et g{\'e}om{\'e}trie riemanienne.
\newblock {\em Journal Math{\'e}matiques Pures et Appliqu{\'e}es}, 8:1--33,
  1929.

\bibitem[CDR05]{cheng2005}
S.-W. Cheng, T.~K. Dey, and E.~A. Ramos.
\newblock Manifold reconstruction from point samples.
\newblock In {\em SODA}, pages 1018--1027, 2005.

\bibitem[Cha06]{Chavel}
I.~Chavel.
\newblock {\em {R}iemannian Geometry: A Modern Introduction}.
\newblock Number~98 in Cambridge studies in advanced mathematics. Cambridge
  University Press, 2006.

\bibitem[Che70]{cheeger1970}
J.~Cheeger.
\newblock Finiteness theorems for {R}iemannian manifolds.
\newblock {\em Am. J. Math}, 92(1):61--74, 1970.

\bibitem[Che00]{chen2000}
Bang-Yen Chen.
\newblock Riemannian submanifolds.
\newblock In {\em Handbook of differential geometry, {V}ol. {I}}, pages
  187--418. North-Holland, Amsterdam, 2000.

\bibitem[Cox98]{Coxeter}
H.S.M. Coxeter.
\newblock {\em Non-{E}uclidean geometry}.
\newblock The mathematical association of America, 6 edition, 1998.

\bibitem[DZM08]{dyer2008sgp}
R.~Dyer, H.~Zhang, and T.~M\"{o}ller.
\newblock Surface sampling and the intrinsic {V}oronoi diagram.
\newblock {\em Computer Graphics Forum (Special Issue of Symp. Geometry
  Processing)}, 27(5):1393--1402, 2008.

\bibitem[ES97]{edelsbrunner1997rdt}
H.~Edelsbrunner and N.~R. Shah.
\newblock Triangulating topological spaces.
\newblock {\em Int. J. Comput. Geometry Appl.}, 7(4):365--378, 1997.

\bibitem[Fr{\'e}48]{Frechet}
M.~Fr{\'e}chet.
\newblock Les {\'e}l{\'e}ments al{\'e}atoires de nature quelconque dans un
  espace distanci{\'e}.
\newblock {\em Annales de l'Institut Henri Poincar{\'e}}, 10:215--310, 1948.

\bibitem[Fri82]{Friedland}
S.~Friedland.
\newblock Variation of tensor powers and spectra.
\newblock {\em Linear and Multilinear algebra}, 12:81--98, 1982.

\bibitem[GVL96]{golub2012matrix}
G.~H. Golub and C.~F. Van~Loan.
\newblock {\em Matrix computations}, volume~3.
\newblock JHU Press, 1996.

\bibitem[IR08]{Ipsen}
I.C.F. Ipsen and R.~Rehman.
\newblock Perturbation bounds for determinants and characteristic polynomials.
\newblock {\em Society for Industrial and Applied Mathematics' Journal on
  Matrix Analysis and Applications}, 30(2):762--776, 2008.

\bibitem[Kar77]{Karcher}
H.~Karcher.
\newblock Riemannian center of mass and mollifier smoothing.
\newblock {\em Communications on Pure and Applied Mathematics}, 30:509--541,
  1977.

\bibitem[Kar89]{Karcher2}
H.~Karcher.
\newblock Riemannian comparison constructions.
\newblock In S.S. Chern, editor, {\em Global Differential Geometry}, pages
  170--222. The mathematical association of America, 1989.

\bibitem[Ken90]{Kendall}
W.S. Kendall.
\newblock Probability, convexity, and harmonic maps with small image {I}:
  Uniqueness and fine existence.
\newblock {\em Procedings of the London Mathematical society}, s3-61 (Issue
  2):371--406, 1990.

\bibitem[Lei99]{leibon1999}
G.~Leibon.
\newblock {\em Random {D}elaunay triangulations, the {T}hurston-{A}ndreev
  theorem, and metric uniformization}.
\newblock PhD thesis, UCSD, 1999.
\newblock arXiv:\-math/\-0011016v1.

\bibitem[Mun68]{munkres1968}
J.~R. Munkres.
\newblock {\em Elementary differential topology}.
\newblock Princton University press, second edition, 1968.

\bibitem[OR09]{outerelo2009degree}
E.~Outerelo and J.~M. Ruiz.
\newblock {\em Mapping degree theory}, volume 108.
\newblock American Mathematical Soc., 2009.

\bibitem[Pet84]{peters1984}
S.~Peters.
\newblock {C}heeger's finiteness theorem for diffeomorphism classes of
  {R}iemannian manifolds.
\newblock {\em J. {R}eine {A}ngew. Math.}, 394:77--82, 1984.

\bibitem[Rus10]{Rustamov}
R.M. Rustamov.
\newblock Barycentric coordinates on surfaces.
\newblock {\em Eurographics Symposium on Geometry Processing}, 29(5), 2010.

\bibitem[San12]{Sander1}
Oliver Sander.
\newblock Geodesic finite elements on simplicial grids.
\newblock {\em International Journal for Numerical Methods in Engineering},
  92(12):999--1025, 2012.

\bibitem[San13]{Sander2}
Oliver Sander.
\newblock Geodesic finite elements of higher order, 2013.

\bibitem[vDar]{vonDeylen2014}
S.W. von Deylen.
\newblock {\em Numerische Approximation in Riemannschen Mannigfaltigkeiten
  mithilfe des Karcher'schen Schwerpunktes}.
\newblock PhD thesis, Freie Universit\"at Berlin, 2014 (to appear).

\bibitem[Whi40]{whitehead1940}
J.~H.~C. Whitehead.
\newblock On {$C^1$}-complexes.
\newblock {\em Ann. of Math}, 41(4), 1940.

\bibitem[Whi57]{whitney1957}
H.~Whitney.
\newblock {\em Geometric Integration Theory}.
\newblock Princeton University Press, 1957.

\end{thebibliography}

\end{document}